\documentclass[12pt]{amsart}
\usepackage[margin=3cm]{geometry}
\usepackage[normalem]{ulem}
\usepackage{amssymb}
\usepackage{amsmath,amsbsy,amssymb,amscd}
\usepackage[british]{babel}
\usepackage[all]{xy}
\usepackage{color}
\usepackage{amsfonts}
\usepackage{latexsym}
\usepackage{yfonts}
\usepackage{amsthm}
\usepackage[linewidth=1pt]{mdframed}
\usepackage{mathrsfs}
\usepackage{hyperref}
\usepackage{graphicx}
\usepackage{comment}

\usepackage{color}
\usepackage{caption}
\usepackage{subcaption}
\usepackage{enumerate}
\usepackage{xcolor}

\numberwithin{equation}{section}

\definecolor{orange}{rgb}{1,0.5,0}

\def\DS{\displaystyle}

\def \qed{\hfill$\square$}

\def\diam{{\rm diam}}

\def\eps{{\varepsilon}}
\def\teps{{\tilde\varepsilon}}

\def\mes{{\rm mes}}
\def\supp{{\rm supp}}

\def\Card{{\rm Card}}

\def\reals{\mathbb{R}}

\def\bK{\mathbf{K}}

\def\bb{\mathbf{b}}

\def\bp{\mathbf{p}}

\def\br{\mathbf{r}}

\def\brho{\boldsymbol{\rho}}

\def\brJ{{\bar J}}

\def\brN{{\bar N}}

\def\brcW{{\overline{\cW}}}

\def\bra{{\bar a}}

\def\brz{{\bar z}}

\def\breps{{\bar\eps}}

\def\breta{{\bar \eta}}

\def\brxi{{\bar \xi}}

\def\cB{\mathcal{B}}

\def\cC{\mathcal{C}}

\def\cD{\mathcal{D}}

\def\cF{\mathcal{F}}

\def\cK{\mathcal{K}}

\def\cL{\mathcal{L}}

\def\cP{\mathcal{P}}

\def\cR{\mathcal{R}}

\def\cT{\mathcal{T}}

\def\cW{\mathcal{W}}

\def\fB{\mathfrak{B}}

\def\fK{\mathfrak{K}}

\def\fP{\mathfrak{P}}

\def\fR{\mathfrak{R}}

\def\fT{\mathfrak{T}}

\def\fW{\mathfrak{W}}

\def\fb{\mathfrak{b}}

\def\hJ{{\hat J}}

\def\hK{{\hat K}}

\def\hP{{\hat P}}

\def\hcW{{\widehat{\cW}}}

\def\hgamma{{\hat\gamma}}

\def\heps{{\hat{\eps}}}

\def\hmu{{\hat\mu}}

\def\hnu{{\hat\nu}}

\def\heta{{\hat\eta}}

\def\tB{{\tilde B}}

\def\tC{{\tilde C}}

\def\tK{{\tilde K}}

\def\tP{{\tilde P}}

\def\tR{{\tilde R}}

\def\tW{{\widetilde{W}}}

\def\tb{{\tilde b}}

\def\tf{{\tilde f}}

\def\tr{{\tilde r}}

\def\tdelta{{\tilde\delta}}

\def\teps{{\tilde\varepsilon}}

\def\teta{{\tilde\eta}}

\def\tphi{{\tilde\phi}}

\usepackage{mathrsfs} 

\DeclareMathAlphabet{\mathpzc}{OT1}{pzc}{L}{it} 

\newtheorem{definition}{Definition}[section]
\newtheorem{theorem}[definition]{Theorem}

\newtheorem{proposition}[definition]{Proposition}
\newtheorem{corollary}[definition]{Corollary}
\newtheorem{lemma}[definition]{Lemma}

\newtheorem{remark}[definition]{Remark}
\newtheorem{question}[definition]{Question}

\def\cP{\mathcal{P}}
\def\geq{\geqslant}
\def\leq{\leqslant}
\def\R{\mathbb{R}}

\def\eps{\varepsilon}

\def\Z{\mathbb{Z}}
\def\N{\mathbb{N}}

\def\cB{\mathcal{B}}

\def\cF{\mathcal F}

\def\epsilon{\varepsilon}
\def\cC{\mathcal{C}}

\def\graph{{\mathrm{graph}}}

\newtheorem*{claim*}{Claim}
\newcommand{\bea}{\begin{eqnarray}}
  \newcommand{\eea}{\end{eqnarray}}
  \newcommand{\beab}{\begin{eqnarray*}}
  \newcommand{\eeab}{\end{eqnarray*}}

  \newcommand{\be}{\begin{equation}}
  \newcommand{\ee}{\end{equation}}

\author{D.\ Dolgopyat, A.\ Kanigowski, F.\ Rodriguez-Hertz}
\title{Exponential mixing implies Bernoulli}

\date{}

\begin{document}
\maketitle

\begin{abstract}Let $f$ be a $C^{1+\alpha}$ diffeomorphism of a compact manifold $M$ preserving a smooth measure $\mu$.
We show that if $f:(M,\mu)\to (M,\mu)$  is exponentially mixing then it is  Bernoulli.
\end{abstract}

\section{Introduction}
\subsection{Main result}

Let $f:(M,\mu)\to (M,\mu)$ be a $C^{1+\alpha}$ diffeomorphism  of a compact manifold $M$ that preserves a smooth measure $\mu$. We say that $f$ is {\em exponentially mixing}
(for smooth functions) if there exists
\footnote{A standard interpolation argument 
(see e.g. Lemma \ref{LmEMAllCr} in Appendix \ref{AppEM})
shows that if \eqref{EMr} holds for some $\br$ then it holds for all $\br>0$ (but taking $\br$ small would require making $\eta$ small).}
 ${\br}\in \N$, $C>0$, and $\eta>0$  such that 
for any $\phi,\psi\in C^{\br}(M)$ 
\begin{equation}
\label{EMr}
\left|\int_M\phi(x)\psi(f^nx)d\mu-\int_M\phi \,d\mu \int_M\psi \,d\mu \right|
\leq Ce^{-\eta n}\|\phi\|_\br \|\psi\|_\br,
\end{equation}
where $\|\cdot\|_\br$ is the norm on $C^\br(M)$. 

Recall that $(f,M,\mu)$ is a {\em Bernoulli system} (or Bernoulli) if for some $m$ it is measure theoretically isomorphic to the shift on $ \{1,\ldots, \ell\}^\Z$ with the measure $\bp^\Z$ where
$\bp=(p_1,\ldots, p_\ell)$ is a probability vector. Our main result is the following:

\begin{theorem}\label{thm:main} 
If $f$ is exponentially mixing then it is Bernoulli.
\end{theorem}

\subsection{Broader context}
\label{SSBroad}
One of the central discoveries made in the last century in the theory of dynamical systems is that smooth systems can exhibit chaotic behavior. The strongest ergodic property that describes chaoticity is
the Bernoulli property, i.e. being (measure-theoretically) isomorphic to a Bernoulli shift. Some weaker ergodic properties describing chaoticity are (see e.g. a survey article by Ya.\ Sinai, \cite{Sin}):
the $K$-property, positive entropy, mixing of all orders, mixing, weak mixing and ergodicity.  It is easy to see that Bernoulli implies $K$ and that mixing of all orders implies mixing which implies weak mixing which implies ergodicity. It follows by \cite{RS61} that $K$-property implies positive entropy. Moreover $K$-property also implies mixing of all orders, see e.g \cite{CFS} 
(this inclusion is probably least trivial from all the inclusions mentioned above). 
It is not known if mixing implies mixing of all orders; this is known as the Rokhlin problem, \cite{Roh49}. Except for the Rokhlin problem it is known that all the above inclusions are strict 
(also in the smooth setting see e.g.  the discussion in \cite{DDKN2}).
All the above mentioned properties do not require a smooth structure and can be defined for an arbitrary measure preserving system.

Classical statistical properties that require a smooth structure (see e.g. \cite{Sin}) are: central limit theorem, large deviations and exponential mixing. These properties provide quantitative information on the system. All the three properties imply ergodicity, but central limit theorem and large deviations do not imply weak mixing and hence also do not imply stronger ergodic properties, see e.g. \cite{DDKN2}. In this paper we focus on consequences of exponential mixing. Notice that trivially exponential mixing implies mixing. However it was not known if it implies any stronger ergodic properties. Our main result (see Theorem \ref{thm:main}) shows that exponential mixing implies the strongest ergodic property: Bernoullicity. In particular, it has the following non-trivial corollary:
\begin{corollary}
\label{CrEM-AllOrd}
Let $f\in C^{1+\alpha}(M)$. If $f$ is exponentially mixing 
 with respect to a smooth measure $\mu$
then it is mixing of all orders and also has positive entropy. 
\end{corollary}
We note that \cite{Host91} shows that mixing implies mixing of all orders for systems whose spectral 
measure is singular. Corollary \ref{CrEM-AllOrd} treats the opposite case where the spectral measure
has analytic density for smooth observables. \smallskip

 We in fact show in Section \ref{sec:posent} that if $f\in C^{1+\alpha}(M)$
is exponentially mixing  for a $f$-invariant measure $\mu$ which is not supported 
on a fixed point of $f$, then $f$ has a non-zero Lyapunov exponent with respect to $\mu$, i.e. we have:
\begin{proposition}
\label{PrLyapNZ}
 If $f: M\to M$ is a $C^{1+\alpha}$ diffeomorphism which is exponentially mixing with respect to a non atomic measure $\mu$ then $f$ has at least one positive Lyapunov exponent (for the measure $\mu$).
\end{proposition}

The following questions are  natural:
\begin{question}
Let $f:X\to X$ be a $C^{1+\alpha}$ map preserving a non atomic measure $\mu$.  Assume that $(f, \mu)$ is exponentially mixing for H\"older observables. Does $f$ have positive topological entropy? Does $(f, \mu)$ have positive metric entropy? Is $(f, \mu) $ a K system? Is it Bernoulli?

\end{question}
Our main result provides positive answers to all those questions if the measure $\mu$ 
is smooth, but it is interesting to weaken assumptions on the invariant measure. \\

We also remark that while our results show that exponential mixing implies mixing of all orders we do 
not get any quantitative 
bounds on the rate of multiple mixing. 
In particular, the following question is natural. We say 
that $(f, \mu)$ is {\em exponentially mixing of order $k$} if there exist constants $\br_k, C_k$, $\eta_k$ such that
if $\phi_1, \phi_2, \dots, \phi_k\in C^\br(M)$ then
\begin{equation}
\label{MEMr}
\left|\int_M \left(\prod_{j=1}^k \phi_j (f^{n_j} x)\right)d\mu-\prod_{j=1}^k \int_M\phi_j \,d\mu  \right|
\leq C_k e^{-\eta_k L} \prod_{j=1}^k \|\phi_j\|_{\br_k},
\end{equation}
where $\DS L=\min_{1\leq j\leq k-1} (n_{j+1}-n_j).$
\begin{question}
Does exponential mixing imply exponential mixing of all orders?
\end{question}

We note that exponential mixing of all orders implies several statistical properties such as
the Central Limit Theorem \cite{BG} and Poisson Limit Theorem for close returns \cite{DFL}.

The reason why our method does not provide quantitative bounds on multiple mixing is because 
we rely on the Pesin theory, which in particular uses a Multiplicative Ergodic Theorem which is a 
non constructive result. It seems of interest to obtain quantitative bounds assuming some estimates
on the measure of points where the convergence in the Multiplicative Ergodic Theorem is slow.
Such results were previously obtained in \cite{ALP, AP10} where instead of exponential
mixing the authors assume non-uniform hyperbolicity and dominated splitting.


The Bernoulli property was shown to hold for many classes of natural dynamical systems: ergodic toral automorphisms \cite{Kat71},  Axiom $A$ diffeomorphisms \cite{Bow74},  
quadratic maps\footnote{For non invertible systems the Bernoulli property means that the natural
extension of $f$ is isomorphic to a Bernoulli shift.}
 with absolutely continuous invariant measure \cite{Led81}, geodesic flows on surfaces of constant negative curvature  
\cite{OrnsteinWeiss}, geodesic flows on higher dimensional manifolds (without focal points) \cite{Pes77}, 
Anosov flows \cite{Rat74},  non-uniformly hyperbolic maps and flows (with singularities) \cite{CH96}.
Recently in \cite{Ka-Bern} it was shown that partially hyperbolic homogeneous systems are Bernoulli. The above list is not complete but it contains the main examples of smooth Bernoulli systems. 

We note that in the last 25 years there has been a significant progress in proving K property 
for partially hyperbolic systems. This study was initiated in \cite{Sin66, BrPes, GPS}. Currently the
strongest result is due to \cite{BW10} and says that a partially hyperbolic center bunched volume 
preserving diffeomorphism with essential accessibility property is K. 
Recall that $f$ is {\em partially hyperbolic} if there is a $Df$ invariant splitting 
$TM=E^u\oplus E^c\oplus E^s$ and positive functions $\nu(x), \hnu(x), \gamma(x),$, $\hgamma(x)$
such that
$$ \nu, \hnu<1, \quad \nu<\gamma<\hgamma^{-1}<\hnu^{-1},$$
and
$$ \|Df(v)\|< \nu\|v\| \text{ if } v\in E^s, \quad
\|Df(v)\|> \hnu^{-1} \|v\| \text{ if } v\in E^u, $$ 
$$ \gamma \|v\|< \|Df(v)\|< \hgamma^{-1} \|v\| \text{ if } v\in E^c. $$ 
A partially hyperbolic system is called {\em center bunched} if the above functions could be chosen 
so that $\nu<\gamma\hgamma$ and $\hnu<\gamma\hgamma.$ A key inspiration for our approach 
comes from the remark that any system with non zero Lyapunov exponents could be regarded as
a non-uniformly partially hyperbolic system (enjoying the center bunching). This allows one to extend
several tools from the theory of partially hyperbolic systems to the non-uniform setting and plays 
an important role in our proof.

Recall that for
a partially hyperbolic system one can define
an accessibility class of a point $x$ as the set of points which can be joined to $x$ by a piecewise smooth
curve such that each piece belongs to either one stable leaf or one unstable leaf. Essential accessibility
means that every measurable set which consists of whole accessibility classes has measure zero or one.
This is weaker than accessibility which means that there is only one accessibility class.
We note that the essential accessibility is insufficient for the Bernoulli property, see \cite{Ka80, Rud, KRHV}.
Thus a natural next step is to understand which additional features of the system are responsible for
Bernoullicity. The recent works \cite{Ka-Bern, DK19, DDKN2} seem to indicate that an important role is played by the competition
between the rate of mixing and the complexity of the system restricted to the subspace with zero exponents. In particular, in the
present paper we show that the exponential mixing implies the Bernoulli property as the growth
in the zero exponents directions is always sub-exponential. However, there are still many open questions
related to the Bernoulli property  for smooth systems. Below we mention a few which seem to play a central
role in the theory. 

\begin{question}
Is exponential mixing assumption in our main theorem optimal? In particular, does there exists 
a diffeomorphism which enjoys a stretched exponential mixing (i.e. at rate $e^{-n^\alpha}$ for $\alpha\in (0,1)$)  but is not Bernoulli?
\end{question}

We note that \cite{DDKN2} constructs non Bernoulli systems with arbitrary fast polynomial mixing rate.
However, in order to get the mixing rate of $n^{-\alpha}$ \cite{DDKN2} considers manifolds of
dimension growing quadratically
 with $\alpha.$ In fact, lowering the dimension of the phase space makes it more
difficult to construct non Bernoulli systems. In \cite{KRHV} the authors construct $K$ non-Bernoulli examples in dimension $4$. On the other hand it follows from the Pesin theory that $K$ implies Bernoulli in dimension $2$. Hence the following classical question is of a central importance:
\begin{question}
Does there exist a $K$ non Bernoulli diffeomorphism preserving a smooth measure  in dimension three?
\end{question} 

The next question is important for the theory of partially hyperbolic systems.

\begin{question}
Is every volume preserving partially hyperbolic system with accessibility property Bernoulli?
\end{question}

\subsection{Outline of the proof.}  
Since our approach requires  rather technical results  from the Pesin theory of $C^{1+\alpha}$ diffeomorphisms we will outline the main steps in the proof for convenience of the reader.

A standard approach for proving the Bernoulli property (developed in \cite{OrnsteinWeiss}) is to verify the {\em very weak Bernoulli (vwB) property} 
which means that for each $S\geq 1$ the itinerary of the orbit during 
the time interval $[0, S]$ is asymptotically independent of the remote past. 
More precisely, if $\cP$ is a finite partition then for large enough $N_2>N_1$,  the distribution of the
itineraries on time $[0, S]$ are almost the same 
for most atoms of $\DS \cP_{N_1, N_2}:=\bigvee_{i=N_1}^{N_2}f^i\cP$ 
(here the closeness of the distributions on itineraries is induced by the topology in which two 
itineraries are close if their Hamming distance is small). 
In this paper we also verify vwB property, but to do so we need to develop 
geometric structure of our system which follows from exponential mixing. We divide our argument into
several steps.

The first main step in the proof (conducted in Section \ref{sec:posent}) is to show that exponential mixing implies that some exponents are non-zero. For the proof of this we only require that $f\in C^{1+\alpha}(M)$ is exponentially mixing for a measure $\mu$ which is not supported on a single point. We first show that if $f$ is exponentially mixing, then there is $\eta_1>0$ and a set $B$ of positive measure
such that for all sufficiently large $n$ 
 the balls $O_n$ centered in $B$ with radius $e^{-\eta_1 n}$ satisfy that the diameter of  $f^n(O_n)$ becomes macroscopic. Next we show  that if exponents were zero then, for each
 $\eps>0$ the images of  balls of size $\sim e^{-\eps n}$ 
 centered at Oseledets typical points
 remain exponentially small after iterating $n$ times.
Taking $\eps<\eta_1$ leads to a contradiction
with the fact that such an image must become of order $1$ (macroscopic) in diameter. \medskip


The existence of non zero exponents makes the Pesin theory applicable to our problem.
In particular, almost every point has an unstable manifold of positive size and the unstable 
lamination is absolutely continuous. We note that all the necessary facts from Pesin theory needed in our paper
can be obtained by standard techniques. In particular our presentation 
relies heavily on Barreira-Pesin book \cite{BP}. On the other hand 
the statements in our paper are less restrictive than in most other references as we only assume
that the system has some non-zero exponents. In particular we extend the theory of fake center stable foliations
developed in the ergodic theory of partially hyperbolic systems to the non-uniform setting. 
One technical novelty in our argument is that we define center foliations that work for finitely many iterates,
and they have
good absolute continuity properties (see the descriptions of 
Sections \ref{sec:fake} and  \ref{ScAC} below).
We would like to emphasize that the results of Sections \ref{sec:u}--\ref{ScAC} (as well as Section \ref{sec:K}
and Appendix \ref{AppPesin})
are valid for any diffeomorphism preserving a smooth measure with some non-zero exponents
and as such are of independent interest.\medskip

Next, a standard backwards contraction argument going back to \cite{AnSin, OrnsteinWeiss, Sin66}
shows that the remote past partition $\cP_{N_1, N_2}$ is almost $u$-saturated,
meaning that if $\xi$ is  sufficiently small then for most points the unstable manifold of size $\xi$ around
$x$ belongs to the same atom as $x.$ We present this argument in Section \ref{sec:u}. 

Given almost u-saturation, the natural idea for verifying the vwB property is to show that
for any large $n$ and any two ``typical'' unstable pieces $\cW_1,\cW_2$ of size $\xi$ there exists an almost measure-preserving map $\theta:(\cW_1,m^u_{\cW_1})\to (\cW_2 ,m^u_{\cW_2})$ 
such that the points $x$ and $\theta x$ remain close for $1-\eps$ proportion of the first $n$ iterates. More precisely, the existence of such maps allows to control $S$--itineraries  for sufficiently large $S$,
while for small $S$ we can use the $K$-property, see  Corollary \ref{cor:VWE}  for details.
For many of the Bernoulli examples
mentioned in \S \ref{SSBroad} this map can be constructed  taking $n'\ll n$, subdividing 
$\DS f^{n'} \cW_s=\bigcup_j \cW_{s, j}$ so that $\cW_{1,j}$ is close to $\cW_{2,j}$ and defining 
$\theta: f^{-n'} \cW_{1,j}\to f^{-n'} \cW_{2,j}$ using the center-stable holonomy. 
In this approach almost measure preservation comes from the absolute continuity of the center stable
foliation, the closeness of $f^j x$ and $f^j (\theta x)$ comes from the fact that the center stable direction
is non-expanding and the possibility of subdividing
$f^{n'} \cW_1$ and $f^{n'} \cW_2$ into the pieces which are close to each
other comes from the minimality of the unstable lamination.
In our case none of the above properties is available. That is, we do not know if the center stable 
distribution is uniquely integrable and if so, it is not clear if the resulting lamination is absolutely continuous.
Moreover the vectors in the subspace corresponding to the zero exponent could grow albeit at a subexponential rate. Finally we do not know if the unstable lamination is minimal.
To overcome these difficulties we establish weakened analogues of the above properties which are
nonetheless sufficient for our purposes.

First of all,  in Section \ref{sec:fake} we introduce the crucial notion of {\em fake center-stable foliation} 
at time $n$. This (locally defined) foliation mimics the behavior of the center-stable foliation for $n$ iterates. 
Fake center-stable foliation was previously used in the study of ergodic properties of partially hyperbolic systems in \cite{BW10} where these foliations were constructed near an arbitrary orbit. In our setting this foliation
is defined only near Lyapunov regular orbits. The main result of Section \ref{sec:fake} is that we can define
this lamination outside of a set of an arbitrary small measure so that there is a unique leaf passing 
through each point. This statement is non trivial even for partially hyperbolic systems in case the
center stable distribution is not  uniquely integrable.

The properties of the fake foliation are studied in  Section \ref{ScAC}. 
For each fixed $n$ the center stable foliation is obtained by pulling back a smooth 
foliation, hence it is absolutely continuous. If the center stable distribution is uniquely integrable
then the fake foliations approach the real center stable foliation as $n$ tend to infinity.
Therefore if the center stable foliation is not absolutely continuous, then we could expect
the fake center stable jacobians to deteriorate on the unit scale.
However, we show in Proposition \ref{lem:abscont}  that if we take two submanifolds which are exponentially
close then the 
jacobian of the fake center stable holonomy is close to 1. Accordingly the fake center stable holonomy
between
nearby  typical unstable leaves is almost measure preserving (Proposition \ref{prop:3.2}). 
Another consequence of the local absolute continuity is local product structure of the measure 
on the small scale established in Corollary \ref{cor:prost'}. 
\medskip


The next key step in the analysis is the  exponential almost equidistribution of unstable leaves
established in Section \ref{sec:main}. The main result of that section is 
the Main Proposition (Proposition \ref{prop:expeq}) which says that given a typical unstable leaf $\cW$
and a partition of the phase space into cubes of size 
$r \geq e^{-\eps n}$ 
we can discard a small proportion of cubes so that the proportion of $f^n \cW$
inside the remaining cubes  is approximately equal to the measure of the cube.
We also point out that the Main Proposition also implies that for {\bf any} cube of size of order $1$, $f^n\cW$ is equdistributed in the cube.

A standard approach to proving  equidistribution of the unstable leaves is the following. Take 
exponentially narrow tube $\cT$ around $\cW$. By exponential mixing the image of $\cT$ is 
equidistributed in the phase space. Next every point $z$ in $\cT$ belongs to the center stable leaf
of a point $z'$ in $\cW$, and since the Lyapunov exponents in the center stable direction are non-positive,
we expect $f^n z$ and $f^n z'$ to be close. Unfortunately, this closeness only holds if
$z'$ is sufficiently regular point and while the contribution of non regular points is small it does not have
to be exponentially small. These necessitates discarding the cubes in our partition which attract an
anomalously large proportion of non regular points.
The Main Proposition is the crucial result in establishing both the $K$ and the Bernoulli properties. 
Namely, the $K$-property follows from the equidistribution of the image of unstable leaves on 
the unit scale, as we explain in Section \ref{sec:K}. 
To verify the very weak Bernoulli we follow the strategy
described above by constructing the coupling between the nearby pieces of 
$f^{n'} \cW_1$ and $f^{n'} \cW_2$ with $n'=\eps n.$  
Here the possibility of the subdivision 
$$f^{n'} W_s=\big(\bigcup_j \cW_{s,j}\big)\cup\{\text{small unmatched part}\}$$
 so that 
$\cW_{1,j}$ is close to $\cW_{2,j}$ comes from the exponential almost equidistribution,
the coupling map is almost measure preserving due to the local absolute continuity
of the fake center stable foliation and the closeness of $x$ and $\theta x$ comes from
the fact that $f^{n'} x$ and $f^{n'} \theta(x)$ are exponentially close at time $n'$ and the
divergence in the center stable direction is subexponential.
This argument is presented in Section \ref{ScVWB}.

\subsection{The choice of parameters.}
We finish our outline by specifying the dependence of parameters that appear in our proofs. Our proofs rely heavily on  the Pesin theory and the very weak Bernoulli property. They involve several scales of smallness (or largeness) of parameters. To make it easier for the reader we summarize the dependences that appear in the paper.

\begin{enumerate}
\item
We start with  $f\in C^{1+\alpha}(M)$ which is exponentially mixing  on $C^\br(M)$
with  exponent $\eta>0$. In the text we introduce the number
 $\heta$ (depending only on $\eta, \br$ and $\dim M$).
 Namely in Lemma \ref{cor:CE} we obtain as a consequence of exponential mixing, that if $\heta$ is small enough then
 the images of balls (centered at typical points) of size $e^{-\heta n}$ become macroscopic after $n$ iterates. Also 
  in Lemma~\ref{lem:expmixB} we show that exponential mixing for smooth functions implies mixing on  parallelograms of size $e^{-\eta_2}n$ provided that $\eta_2$ is small enough. 
  \footnote{According to \eqref{Heta} and \eqref{Eta2-Eta} we 
can take
$\DS \heta=\min\left(\frac{1}{10 D }, \frac{1}{8\br}\right), $
and $\DS
\eta_2=\frac{\eta}{1+2\br} \left(\frac{D}{10}-\frac{1}{15}\right),$  where $D=\dim(M)$, but the precise values of these constants are not important for our argument.}
We  use the notation $\alpha_i$ to denote  H\"older exponents of certain functions that we define. 
 All the $\alpha_i$ depend only  on  $\|f\|_{C^1}$, the Lyapunov exponents of $f$, $\alpha$, and $\dim M$.  In Proposition \ref{lem:abscont} we will also use the number $\beta>0$ which can be taken to be $1/2\min(\alpha,\min_i \alpha_i)$.

\item 
Our proof proceeds by verifying the Ornstein-Weiss criterion (see Corollary~\ref{cor:VWE}). 
Thus given $\eps>0$ we need to construct a map $\theta$ verifying \eqref{MetClose}.
 This $\eps$ defines the next level of smallness. 
We note that if the conditions of Corollary \ref{cor:VWE} are satisfied for some $\eps>0$ then they
also hold for all $\eps'>\eps.$ Therefore in the proof we assume that $\eps$ is sufficiently small
and we will use estimates like $\eps^2<\eps/100$ without additional explanations. 

\item 
Having fixed $\epsilon>0$ we take $\delta=\eps^{100}.$ 
We then apply Lemma  \ref{lyapchart} for this choice of $\delta$ to get the corresponding function 
$\mathfrak{r}(\cdot)$  describing the size of Pesin chart for $(\lambda, \delta)$-regular points.
 Next we pick the parameter $\tau>0$ for the  Pesin sets in  \eqref{eq:hpc} and \eqref{eq:hpk} respectively. For points in these sets the function  $\mathfrak{r}(\cdot)$ is larger than $\tau$. We again want $\tau=\tau(\epsilon)$ to be small enough in terms of $\epsilon$ so that $\mu(P_\tau)\geq 1-\eps^{10^{10}}$ (see \eqref{eq:pc}).
\item 

Having fixed all the above parameters we now choose a sufficiently large  $n_0$ 
(largeness depending on all the previously fixed parameters) and we conduct 
the proof for $n\geq n_0$.
\end{enumerate}

Throughout the paper we use the following abbreviations: 
$$\xi_n=e^{\eps^2 n-\eta_2\epsilon n}, 
\quad \tilde{r}_n=e^{-\eta_2\epsilon n-\eps^2n}. $$

The above parameters are the sizes of the parallelograms $B(\xi_n,\tilde{r}_n)$ that we will consider: the parameter $\xi_n$ is  the size in $u$ direction and $\tilde{r}_n$ is the size in the (fake) $cs$ direction. It is important that both are exponentially small, but the $u$ direction is longer than the $cs$ (this simplifies some arguments using H\"older continuity of the $u$ foliation).\\

\subsection{Notation.}
The following notation is used throughout the paper:

$\bb=10^{10}.$

$B(\xi, r)$--the parallelograms in the phase space (see \eqref{BXiR}). 

$\bar{B}_s$ -- disjoint parallelograms whose union covers most of the space (see family $\mathbb{B}_1$ after Lemma \ref{lem:setsB})

$\mathbb{B}_2$ -- the family of parallelograms defined after equation (\ref{HPrime}).
	
$B_i(r)$ -- parallelograms defined after equation (\ref{BI-R}).

$\cC^{u}_\fb, \cC^{cs}_\fb$
--unstable and center-stable cones of aperture $\fb$ 
(see \eqref{CuCones}, \eqref{CsCones}).

$\cC^{u}_x(y), \cC^{cs}_x(y)$--the cones of aperture $1/2$ of  $x$ shifted to $y$ (see \eqref{ShiftedCones}). 

$D=\dim(M),$ $d^u=\dim(E^u),$ $d^{cs}=d(E^{cs}).$

$E^u(x)$, $E^{cs}(x)$--the Oseledets subspaces at $x$ corresponding to positive and 
non-positive exponents respectively.

$\tilde E^{cs,n,\delta}_x(\bar y)=T_{\bar y}\tilde W^{cs,n,\delta}_{x}(\bar y)$. 

$\tf$--representation of $f$ in Lyapunov charts (see Lemma \ref{lyapchart}).

$\cF_{i,j}$--the leaves of fake center stable foliations (obtained by reindexing of $\cF_{i,j,s}$ from \eqref{eq:fak2},
see end of Section \ref{sec:fake}).

$h_{x, \delta}$--Lyapunov coordinate map (see Lemma \ref{lyapchart}).

$L_{x, \delta}$--the linear map rectifying the angles between $E^u(x)$ and $E^{cs}(x)$
(see Lemma \ref{lyapchart}).

LyapReg--the set of  Lyapunov regular points (see Definition \ref{DefLyapReg}).

$\cL_{n,\tau}=P_\tau\cap f^{-n}P_\tau\cap f^{-\epsilon n}P_\tau$ (see \eqref{eq:lnc}).

 $m^u_x$ -- Lebesgue measure on $W^u_x$.

$m^u_{\cW}$ -- the conditional of $\mu$ on the unstable leaf $\cW$ (see \eqref{MuCond}).
 
$P_\tau$ -- the Pesin set of points with $\mathfrak{r}\geq \tau$ (see \eqref{eq:pc}). 

$Q_x$ -- Lyapunov neighborhood of the point $x$ (see \eqref{QXN}).

$Q_x^n$ -- Bowen ball with respect to the Lyapunov neighborhood $x$ 
(see \eqref{QXN}).

$Q_x(r_u,r_{cs})$ -- images of the parallelograms in the tangent space (see \eqref{QXPar1}).

$r_u(x)$ -- size of unstable manifold of $x$ defined in Lemma \ref{LmLocMeas}.

$\mathfrak{r}_\delta(x)$ -- sizes of Lyapunov charts at $x$
 (see \eqref{PesRad}).

$\mathfrak{R}_\delta(x)$ -- a function measuring the nonuniformity in the Multiplicative
Ergodic Theorem for $x$ (see Definition \ref{DefLyapReg} and Theorem \ref{ThLRFull}).

$\tR_{i,j}(r)$ -- the domains of $\cF_{i,j}$ (see \eqref{eq:tilr}).

$\tilde{r}_n=e^{-\eta_2\epsilon n-\eps^2n}$ (see \eqref{DefRn}).

$\cR_n$ - union of $\cF-$saturations of $\tR_{i,j}(r)$ defined by \eqref{WholeLam}.

$\tW^{cs,n,\delta}_{x}(\bar y)$ -- the image of fake center stable foliation in the Lyapunov charts
(see Definition \ref{DefFakeCS}). 

$\tW^u(x)$ -- the image of (fake) unstable manifold in the exponential coordinates (see Lemma \ref{lem:unstman}).

$W^u_{x,R}$ -- the local stable manifold of $x$ of size $R$ (see \eqref{WuR}).

$W^u_{x}$ -- the local unstable manifold of $x$ (see \eqref{WuLoc}).

$\hcW_i$ -- reference unstable manifolds (see \eqref{HCW}).

$\alpha$ -- H\"older regularity of $Df$.

$\alpha_1$ -- the regularity of the Oseledets spaces on Pesin set (see Lemma \ref{LmIDHold}). 

$\alpha_2=\min(\alpha/2,\alpha_1)$ -- regularity of the rectifying maps $L_x$ (see \eqref{Alpha1}). 

$\alpha_3$ -- the (controlled)  regularity of center stable manifolds (see Lemma \ref{cs-foliation}).

$\alpha_4$ -- the regularity of unstable manifolds (see Lemma \ref{lem:unstman}).

$\alpha_5$ -- cone contraction stability exponent (see Lemma \ref{cones}).

$\alpha_6$ -- the (controlled) regularity of the tangent spaces of fake center stable manifolds
(see Lemma \ref{graphtransfcs}).

$\alpha_7$ -- the (controlled) regularity of the tangent spaces of  admissible manifolds
(see Lemma \ref{admisibleconv}).

$\beta$ -- regularity exponent of admissible manifolds in Definition \ref{apadm}.

$\eta^u_x:\R^u\to\R^{cs}$ -- a function whose graph defines $\tW^u(x)$ (see Lemma \ref{lem:unstman}).

$\tilde\eta^{cs,n,\delta}_{x,\bar y}:\R^{cs}\to\R^u$ -- the function defining the graph of $\tilde W^{cs,n,\delta}_{x}(\bar y)$
(see Lemma \ref{cs-foliation}). 

$\eta_2$ -- the smallness exponent sufficient to guarantee the exponential mixing on parallelograms
(see Lemma \ref{lem:expmixB}). 

$\lambda$ -- the smallest positive exponent (see Lemma \ref{lyapchart}).

$\xi_n=e^{\eps^2 n-\eta_2\epsilon n}$ (see \eqref{DefXiN}).

$\brho(x)$ -- density of $\mu$ in the product coordinates (see \eqref{MesMultiDisc}).

$\tau$ -- the parameter of the Pesin set so that $\mu(P_\tau)\geq 1-\eps^\bb.$

\subsection{Layout of the paper.}
The proof of Theorem \ref{thm:main}  is carried out in
Sections \ref{sec:posent}--\ref{ScVWB} following the outline given above.
Section \ref{ScDef} contains the necessary background. Technical estimates from
the Pesin theory needed in our proof are collected in the appendix.
Section \ref{ScSkew} describes several skew products with Anosov base and homogenous dynamics  
in the fiber
where the Bernoulli property follows from
our main result. These examples seem unaccessible by other methods. 
The results are sharp as the failure of exponential mixing also leads to the failure of the Bernoulli 
property.
\\

{\bf Acknowledgements}: The authors would like to thank Mariusz Lema\'nczyk, Yuri Lima and Yakov Pesin for suggestions improving the readability of the paper.

\section{Preliminaries}
\label{ScDef}

\subsection{The Pesin theory}
Let $f$ be a $C^{1+\alpha}$ diffeomorphism of $(M,\mu)$ with $\dim(M)=D$. We denote the derivative of $f$ by $Df_x:T_xM\to T_{fx}M$ and let $Df_x^{n}:=Df_{f^{n-1}x}\circ\ldots\circ Df_x$.
We assume throughout that $\mu$ is smooth and that $f$ is ergodic with respect to $\mu.$

\begin{definition}
\label{DefLyapReg}
\cite[Thm 2.4 \& Prop 2.6]{BRHZ}\cite[Thm 3.4.10, Thm 3.5.5, Prop 3.5.8.]{BP}
	Let $\lambda>0$. A point $x$ is said to be $(\lambda,\delta)$-Lyapunov regular if there is a splitting 
	$$T_{f^{k}x}M=E^{cs}(f^kx)\oplus E^u(f^kx)$$ for $k\in\Z$ and there are numbers
	$\mathfrak{R}_{\delta}(f^k x)>0$ such that
	\begin{enumerate}
		\item $\mathfrak{R}_{\delta}(f^{k+n}x)\leq e^{\delta|k|}\mathfrak{R}_{\delta}(f^nx)$ for every $k\in\Z$, $n\in\Z$;
		\item $Df^kE^\sigma(x)=E^\sigma(f^kx)$ for every $k\in\Z$, $\sigma=cs,u$;
		\item if $v\in E^{cs}(f^kx)$, $n\geq 0$, then $$\|Df^nv\|\leq \mathfrak{R}_{\delta}(f^{n+k}x)e^{\delta n}\|v\|;$$
		\item  if $v\in E^{u}(f^kx)$, $n\leq 0$, then $$\|Df^nv\|\leq \mathfrak{R}_{\delta}(f^{n+k}x)e^{n\lambda+|n|\delta}\|v\|;$$
		\item $\angle (E^{cs}(f^kx),E^u(f^kx))\geq \mathfrak{R}_{\delta}^{-1}(f^k(x))$.
		
	\end{enumerate}
\end{definition}

 Note that $(\lambda, \delta)$-Lyapunov regular point is 
$(\lambda, \delta')$-Lyapunov regular for any $\delta'>\delta$.

\begin{theorem}\label{ThLRFull}
\cite[Thm 3.4.10, Thm 3.5.5, Prop 3.5.8.]{BP}
	 Let $\lambda$ be the smallest positive Lyapunov exponent of $f$. Then
	for each $\delta>0$ the set of $(\lambda,\delta)$-Lyapunov regular points has full measure. 
	Moreover, 
	the function $x\to \mathfrak{R}_{\delta}(x)$ can be chosen Borel measurable. 
\end{theorem}

Fix $\lambda>0$ from the above theorem and let $\delta$ be a small number. Let
\be\label{eq:lyapreg}
LyapReg(\delta):=\{x\in M:\; x \text{ is } (\lambda, \delta/4)\text{-Lyapunov regular}\}.
 \ee

Notice that since $f$ is $C^{1+\alpha}$, there is a constant $C_1>0$ such that
denoting by $\exp_p$ the exponential map and letting $\hat f_p=\exp_{f_p^{-1}} \circ f\circ \exp_p$ 
we have that $\hat f_p$ is defined on the ball of radius $1/C_1$ in $T_p M$ and
 \begin{equation}\label{eq:fC1}\|(D_{\hat z_1}\hat f_p)^{-1}-(D_{\hat z_2}\hat f_p)^{-1}\|_{T_{fp}M\to T_pM}\leq C_1|\hat z_1-\hat z_2|^{\alpha}_{p}\end{equation}
  if $|\hat z_1|_p,|\hat z_2|_p\leq \frac{1}{C_1}$.
For $p\in LyapReg(\delta)$ let 
	\be	\label{PesRad}
	\mathfrak{r}_{\delta}(p):=\left(\delta^{-1}\frac{(\mathfrak{R}_{\delta/4}(p))^2}{\sqrt{1-e^{-\delta}}}(\sqrt{2})^{1+\alpha}C_1\right)^{-2/\alpha}.\ee 
	 The precise formula for $\mathfrak{r}_\delta(p)$ will not be important in our arguments. 
	We will just use that
	it is uniformly bounded from below on Pesin sets defined in \eqref{eq:hpc} below.
	We may assume that $\mathfrak{r}_{\delta}(p)\leq \frac{1}{C_1}$ (see e.g. the comment below Lemma \ref{lem:kx}).
	 Moreover  by property (1) in Definition \ref{DefLyapReg}, we have that 
	$$e^{-\frac{\delta}{\alpha}}\mathfrak{r}_{\delta}(p)\leq \mathfrak{r}_{\delta}(fp)\leq e^{\frac{\delta}{\alpha}}\mathfrak{r}_{\delta}(p).$$ 

 In what follows (see eg. Lemma \ref{lyapchart})
 we will work with a rescaling of $\mathfrak{r}(\cdot)$, $\mathfrak{r}'_\delta(\cdot)=\mathfrak{r}_{\delta\alpha}(\cdot)$. To simplify notation and since we will work with the rescaled function from now on, we will denote the rescaling also by $\mathfrak{r}_\delta(\cdot)$.

Let us write $\R^D=\R^u\times\R^{cs}.$ Sometimes we will write $\R^D_x$ to emphasize that it correspond to the tangent space of $x$. 


We now define the Lyapunov norm. Let $x\in LyapReg(\delta/4)$. Define for $u\in E^{u}(x)$ $$|u|'^2_{x,\delta}=\sum_{m\leq 0}|D_xf^mu|_{f^mx}^2e^{-2\lambda m-2\delta|m|}$$
	and for $v\in E^{cs}(x)$ $$|v|'^2_{x,\delta}=\sum_{m\geq 0}|D_xf^mv|_{f^mx}^2e^{-2\delta|m|}.$$ 
	These norms define inner products naturally and we extend it to an inner product on $T_xM=E^u(x)\oplus E^{cs}(x)$ by declaring these spaces orthogonal.

Next we define the Pesin sets.
Let 
\be\label{eq:hpc}
\hat{P}_\tau=\hat{P}^{\delta}_\tau:=\{x\in M\;:\; \mathfrak{r}_{\delta}(x)\geq \tau\}.
\ee 
Since $\mathfrak{r}_\delta(\cdot)$ is measurable it follows that 
$\mu(\hat{P}_\tau)\to 1$ when $\tau\to 0.$  (Note also that $\hat{P}_{\tau'}\subset \hat{P}_\tau$ for $\tau<\tau'$).

\begin{lemma}
	\label{LmIDHold}
	There exists $\alpha_1>0$ such that for every $\tau,\delta>0$ ,$\delta<\delta_0$ there is $K=K(\tau,\delta)>0$ such that the maps $x\mapsto E^u(x)$ and $x\mapsto E^{cs}(x)$ are $(K,\alpha_1)$-H\"older continuous on $\hP^\delta_\tau$ . 
\end{lemma}	
\begin{proof} The proof of this is the same as the proof of  Theorem 5.3.2 in \cite{BP} using the observation that 
the proof of this theorem  does not require that $E^{cs}$ is contracting. 
\end{proof}

\begin{lemma}[Lyapunov charts]\label{lyapchart}
[see 
\cite[Prop 5.1]{BRHZ},\cite[Thm 5.6.1]{BP} and \cite[Sec 11.2]{BRH} for one sided charts]
	There is $\alpha_2,\delta_0>0$ such that  for every  $\delta<\delta_0$ 
	there are linear maps $ L_{x,\delta}:\R^D\to T_xM$,  such that 
	\begin{enumerate}
	\item[$L1$.] $L_{x,\delta}$ is an isometry between the standard metric in $\R^D$ and  the  $|\cdot|'_{x,\delta}$ metric;
 	\item[$L2$.] for every sufficiently small $\tau>0$, the maps $x\mapsto L_{x,\delta}$ are $\alpha_2$- H\"older continuous on $\hat{P}_\tau$.
	\end{enumerate}

Moreover 
defining $$h_{x,\delta}=\exp_x\circ L_{x,\delta},\quad\mathrm{and}\quad\tilde f_{x,\delta}=h^{-1}_{fx,\delta}\circ f\circ h_{x,\delta},$$ 
we have
	\begin{enumerate}
		\item $h_{x,\delta}(0)=x$;
		\item $L_{x,\delta}(\R^u)= E^u(x)$ and $L_{x,\delta}(\R^{cs})=E^{cs}(x)$;
		\item $\max(\|L_{x,\delta}\|,\|L_{x,\delta}^{-1}\|)\leq \mathfrak{r}^{-1}_{\delta}(x)$;
		\item $\text{domain}(\tilde f_{x,\delta})\supset B_{\mathfrak{r}_{\delta}(x)}(0)$ and $\text{domain}(\tilde f^{-1}_{x,\delta})\supset B_{\mathfrak{r}_{\delta}(x)}(0)$;
		\item 
	$\DS e^{\lambda-\delta}|v|\leq |D_0\tilde{f}_{x,\delta}(v)| \text{ for }v\in\R^u,\quad |D_0\tilde f_{x,\delta}(v)|\leq e^{\delta}|v|\text{ for }v\in\R^{cs};$ 
		\item $\text{H\"ol}_{\alpha_2}(D\tilde f_{x,\delta})\leq\delta,$ 
		$\text{Lip}(\tilde f_{x,\delta}-D_0\tilde f_{x,\delta})\leq\delta$ and $\text{Lip}(\tilde f^{-1}_{x,\delta}-D_0\tilde f^{-1}_{x,\delta})\leq\delta;$
		\item $\max\Big(\text{Lip}(h_{x,\delta}), \text{Lip}(h_{x,\delta}^{-1})\Big)\leq \mathfrak{r}^{-1}_{\delta}(x)$.
	\end{enumerate}
\end{lemma}

 While this lemma is standard (see the references above) we recall the proof 
 (the statement of Lemma \ref{lyapchart} is obtained by combining  
Lemmas \ref{complyap}, \ref{basicbound},\ref{holLx} from Appendix \ref{AppPesin})
in Appendix \ref{AppPesin} since
the intermediate steps in its proof are
also important in the derivation of other estimates which will be used in the paper and are described below.

We extend $\tilde f_{x,\delta}$ to all $\R^D$ by making it linear outside of the ball of radius $2\mathfrak{r}_{\delta}(x)$ and with same bounds as in Lemma \ref{lyapchart} (note that this is possible by taking a smaller $\alpha_2$ if necessary, see the definition of $\mathfrak{r}_\delta(\cdot)$).

 \subsection{Hadamard-Perron, center-stable foliation}

\begin{definition}
\label{DefFakeCS}
	Given $0<\delta<\delta_0$, $x\in LyapReg(\delta)$ and $n\geq 0$, we define the foliation 
	 $\tW^{cs,n,\delta}_{x}$ 
on $\R^D_x$ by
		by pulling back the foliation by planes parallel to $\R^{cs}_x$ via $\tilde f_{x,\delta}^{(n)}$, i.e. $$\tW^{cs,n,\delta}_{x}(\bar y)=(\tilde f_{x,\delta}^{(n)})^{-1}\left(\R^{cs}+\tilde f_{x,\delta}^{(n)}(\bar y)\right).$$ 
	Let $\tilde E^{cs,n,\delta}_x(\bar y)=T_{\bar y}\tilde W^{cs,n,\delta}_{x}(\bar y)$. 
\end{definition}
Notice that for $k\in [0,n]$, $$\tilde f_{x,\delta}^{(k)}(\tilde W^{cs,n,\delta}_{x}(\bar y))=\tilde W^{cs,n-k,\delta}_{f^kx}(\tilde f_{x,\delta}^{(k)}(\bar y)).$$
The following lemma is crucial:
\begin{lemma}\label{cs-foliation}
	There exists $\alpha_3,\delta_f>0$ such that the following holds: for every $\delta\in (0,\delta_f)$, any $x\in LyapReg(\delta)$, every $\bar y\in\R^D$, there is $\tilde\eta^{cs,n,\delta}_{x,\bar y}:\R^{cs}\to\R^u$ such that 
	$$\tilde W^{cs,n,\delta}_{x}(\bar y)=\graph\left(\tilde\eta^{cs,n,\delta}_{x,\bar y}\right),$$ and  
	$$\|D\tilde\eta^{cs,n,\delta}_{x,\bar y}\|_{C^0}\leq \frac{3\delta}{1-e^{-\lambda+\sqrt\delta}}$$ and 
	\be \label{TSpHolder}
	 [D\tilde\eta^{cs,n,\delta}_{x,\bar y}]_{C^{\alpha_3}}\leq  
	\frac{12\delta}{1-e^{-\lambda+\sqrt\delta}}\ee 
\end{lemma}

We will prove the above lemma 
in \S \ref{sec:csfol} of the appendix.

We finish this subsection with the following straightforward lemma:
\begin{lemma}\label{jacgras}
	Let $\phi:GL(\R^D)\times Grass^u(\R^D)\to\R$ be $\phi(A,E)=\log|\det(A|E)|$. Given $C_0$ there is a constant $N(C_0)$ such that if $\|A\|,\|B\|, \|A^{-1}\|,\|B^{-1}\|\leq C_0$ 
 and
	$E$ and $F$ are graphs of the maps $L_E,L_F:\R^a\to\R^b$, respectively, where
	$\R^D=\R^a\times\R^b$,
	then $$|\phi(A,E)-\phi(B,F)|\leq N(C_0) (\|A-B\|+\|L_E-L_F\|).$$ 
\end{lemma}

In the proof of our main result (see the beginning of Section \ref{sec:fake}) we work with fixed  $\epsilon>0$. We always assume that   $\delta>0$ (from the above results) is small enough in terms of $\epsilon$ but fixed (e.g $\delta= \epsilon^{10}$ would do). Therefore, we will omit it in the notation below.

 Denote by $B^{cs}(z,R)$ the ball of radius $R>0$ centered at $z\in \R^{cs}$, with an analogous notation for $B^u(z,R)$. If $z=0$ we simply denote $B^{cs}(R)=B^{cs}(0,R)$ with an analogous notation for $B^{u}(R)$.

Let
\be
 \label{QXN}
 Q_x:=h_x(B_{\mathfrak{r}_{\delta}(x)}(0))\text{ and }Q^{(n)}_x:=\bigcap_{k=0}^{n-1}f^{-k}Q_{f^kx};
 \ee
 notice that $Q^{(n)}_x=h_x(\bigcap_{k=0}^{n-1}(\tilde f_x^{(k)})^{-1}(B_{\mathfrak{r}_{\delta}(f^kx)}))$. 
We also define  
\be \label{QXPar1}
Q_x(r_u,r_{cs}):=h_x(B^u(r_u)\times B^{cs}(r_{cs})), \ee
\be \label{QXPar2} Q_x(r)=Q_x(r,r),\quad Q_x(A,r)=h_x(A\times B^{cs}(r)). \ee 
Let $d^u=\dim E^u$ and $d^{cs}=\dim E^{cs}$.

In the paper it will be more convenient to work with cubes rather than balls (i.e. balls in the {\em maximum} norm $\|x\|_{\infty}=\max |x_i|$). Let $C^u(R)\subset \R^u$ be a cube centered at $0$ of side length $R$, i.e. a  ball centered at $0$ of radius $R/2$ in the metric $\|\cdot\|_\infty$.

\begin{lemma}\label{lem:unstman}
({Lemma \ref{fakeu-foliation}} and Corollary \ref{CrWuBC})
	 There are constants $C_0, \alpha_4>0$ such that for every $x\in LyapReg$,
$$
\tW^u_x:=\{y\in\R^D:\limsup\frac{1}{n}\log|\tilde{f}_x^{(n)}(y)|<0\}
$$ 
is the graph of a $C^{1+\alpha_4}$ function 
\be\label{eq:etx}
\eta=\eta^u_x:\R^u\to\R^{cs}
\ee with 
\be\label{eq:etau}
\|\eta\|_{C^{1+\alpha_4}}\leq C_0,\;\;\eta(0)=0,\;\; D_0\eta=0.
\ee
Moreover 
	if $\bar z_1,\bar z_2\in\tW_{x}^{u}$ then $$|(\tilde f_{f^{-1}x})^{-1}(\bar z_1)-(\tilde f_{f^{-1}x})^{-1}(\bar z_2)|\leq e^{(-\lambda+2\delta)}|\bar z_1-\bar z_2|;$$ and if
	 we define
\be \label{WuR}
\tW^u_{x,R}:=\graph(\eta,C^u(R)),\quad W^u_{x,R}=h_x(\tW^u_{x,R}) \ee
	then for every $x\in Lyapreg$ and $R>0$, 
	$$(\tilde f_x)^{-1}\tW^u_{fx,R} \subset \tW^u_{x,e^{(-\lambda+3\delta)}R}. $$
	\end{lemma}

\medskip
\noindent Set
\be\label{WuLoc}
W^u_x= W^u_{x,\mathfrak{r}_\delta(x)}.
\ee
Notice that $f^{-1}W^u_{x}\subset W^u_{f^{-1}x}$ since by Lemma \ref{lem:unstman} $f^{-1}(W^u_{x,\mathfrak{r}_\delta(x)})\subset W^u_{f^{-1}x,\mathfrak{r}_\delta(f^{-1} x)}(f^{-n} x) $. In some papers $W^u_x$ is called {\em local} unstable manifold of $x$
and is denoted $W^u_{x, loc}$. We do not use the subscript {\em loc} since we will not need
to consider global unstable manifolds of $x$.

By \cite[Theorem 7.1.1]{BP} (see also Corollary \ref{CrWuBC})
 there exists a measurable function $\fK(x)$ such that the size of the unstable
 manifold of $x$ is greater than $1/\fK(x)$ and moreover for each $y\in W^u_{x}$ 
\be\label{BackContrfK} 
  d(f^{-n} x, f^{-n} y)\leq \fK(x) e^{(\delta-\lambda) n} d(x, y) 
\ee  
   i.e. $W^u_{x}$ is exponentially contracted. Let 
 \be\label{eq:hpk}
 \tilde{P}_\tau=\{x: \fK(x)\leq \tau\}.
 \ee

In what follows for a given $\epsilon$ we will pick $\tau$ small enough  so that if we define (see \eqref{eq:hpc})
 \be\label{eq:pc}
 P_\tau:=\hat{P}_\tau\cap \tilde{P}_{\tau^{-1}},
 \ee
then $\mu(P_\tau)\geq 1-\eps^{\bb}$, where $\bb=10^{10}$. 

We start with the following observation:
\begin{lemma}\label{lem:boder} 
There exists a constant $K'$ such that for every $x\in LyapReg$, 
$$
\|D_0\tilde{f}_{x}\|, \|D_0\tilde{f}_{x}^{-1}\|<K'.
$$
\end{lemma}
\begin{proof}
This follows from (\ref{eq:compla}) in Lemma \ref{complyap}.
\end{proof}

\subsection{Cones}
Let 
\be \label{CuCones} \cC^{u}_\fb=\{v=v^u+v^{cs}\in\R^D: |v^{cs}|\leq \fb|v^u|\}, \ee
and
\be \label{CsCones} \cC^{cs}_\fb=\{v=v^u+v^{cs}\in\R^D: |v^{u}|\leq \fb|v^{cs}|\}\ee
 be the $\fb-$cones around $\R^u$ and $\R^{cs}$.  For $x\in LyapReg$ and $y\in Q_x$, define the 
 cones
  \be \label{ShiftedCones}
  \cC^u_x(y)=(D_yh_x^{-1})h_x \cC^u_{1/2}, \quad \cC^{cs}_x(y)=(D_yh_x^{-1})h_x \cC^{cs}_{1/2}. \ee 
 
 We have the following:

\begin{lemma}\label{LmCones}
For every $\delta,\tau>0$ there exists $n_{\delta,\tau}$ such that for every $n\geq n_{\delta,\tau}$, every $x\in \hP_\tau^\delta \cap f^{-n}\hP_\tau^\delta$, and every  $y\in Q^{(n)}_x$ we have
	\begin{enumerate}
		\item if $v\in  \cC^u_x(y)$ 
		$\|D_yf^nv\|\geq \tau^2e^{n(\lambda-3\delta)}\|v\|$;
		\item if $v\in  \cC^{cs}_x(y)$  then  for each $0\leq k\leq n$ it holds
		$\|D_yf^k v\|\leq \tau^{-2}e^{2\delta k}\|v\|$;
		\item $Df_x^n(\cC^u_x(y))\subset \cC^u_{f^n(x)}(f^n(y))$;
		\item $Df_x^n(\cC^{cs}_x(y))\supset \cC^{cs}_{f^n(x)}(f^n(y))$;
	\end{enumerate}
\end{lemma}
 The proof is an immediate consequence of the definition of $h_x$ and the following Lemma:

\begin{lemma}\label{cones}
There is $\alpha_5>0$ such that for every ${\fb} <1$ there is $C_0>0$  such that for every $\delta\in (0,\delta_0)$ there exists $r_\delta>0$ such that for every $r\in (0,r_{\delta})$ satisfying $C_0r^{\alpha_5}<{\fb}<1$ every $x\in LyapReg$ and $|w|<r$, 
	$$|D_w\tilde f_x(v)|\geq e^{\lambda-2\delta}\text{ for } v\in \cC^u_{\fb},\quad
	|D_w\tilde f_x (v)|\leq e^{2\delta}\text{ for } v\in \cC^{cs}_{\fb},$$ 
	$$D_y\tilde f_x\cC^{u}_{\fb}\subset \cC^u_{e^{-\lambda+4\delta}{\fb}}\,\,\,\,\text{and}\,\,\,\,\cC^{cs}_{\fb}\subset D_y\tilde f_x \cC^{cs}_{e^{\lambda+4\delta}{\fb}}.$$
\end{lemma}
\begin{proof}
This is a straightforward consequence of $f$ being $C^{1+\alpha}$ and Lemma \ref{lyapchart}.
\end{proof}






\subsection{Conditional measure along unstables}
For $x\in LyapReg$, we let $m^u_x$ to be Lebesgue measure on 
 $W^u_x$. 

\begin{lemma}\cite[Thm 8.6.8 and its proof and Thm 8.6.13 and Thm 9.3.4]{BP}
\label{lem:abscont00}
	For $x\in LyapReg$ and $z\in W^u_x$, we define  $\rho(z,\cdot):W^u_z\to\R$, $$\rho(z,y)=\prod_{j\geq 1}\frac{\det\left(Df_{f^{-n}(z)}|E^u(f^{-n}(z))\right)}{\det\left(Df_{f^{-n}(y)}|E^u(f^{-n}(y))\right)},$$ it is H\"older continuous. Moreover, defining $\rho_x:=\rho(x,\cdot)$ we have that for every $\tau$, $\hP_\tau\ni x\to \rho_x$ is H\"older continuous and $\rho(x,y)\rho(y,z)=\rho(x,z)$.
\end{lemma}
	 
For $\cW\subset W^u_x$, we define 
\be \label{MuCond} 
m^u_{\cW}(A)=\frac{\int_A\rho(x,y)dm^u_x(y)}{\int_{\cW}\rho(x,y)dm^u_x(y)}. \ee
Note that  if  $z\in W^u_x$ and we use $\rho_z$ to define $\mu_{\cW}$ we will get the same formula.


For $x\in \hP_\tau$, we take a transversal $T_x=h_x(\graph(\psi))$, where $\psi:\R^{cs}_x\to\R^u_x$ is  $C^1$ function with $\|\eta\|_{C^1}<C_1$,  and we define $\tilde T^\tau_x=\bigcup_{y\in \hP_\tau\cap Q_x}T_x\cap W^u_{y,\tau}$  and $\tilde Q^\tau_x=\bigcup_{y\in \hP_\tau\cap Q_x}W^u_{y,\tau}$. We have the following lemma:

\begin{lemma}\label{lem:abscont0} (\cite[\S 3.3]{Pes76})
	Given $\tau>0$ and $x\in \hP_\tau$, if $A\subset Q_x$ then $$\mu(A\cap \tilde Q_x^\tau)=\int_{\tilde T_x^\tau}d\nu^\tau_x(z)\int_{W^u_{z,\tau}\cap A}\rho(z,y)dm^u_z(y)$$ 
	where $\nu_x^\tau$ is absolutely continuous w.r.t. Lebesgue on $\tilde T_x^\tau\subset T_x$. Also we have that $$\frac{m^{cs}_{T_x}(T_x\setminus \tilde T^\tau_x)}{m^{cs}_{T_x}(T_x)}\to 0$$ as $\tau\to 0$ where $m^{cs}_{T_x}$ is Lebesgue measure on the transversal $T_x$.
\end{lemma}




\subsection{Measure theory}
\begin{definition}
A  map $\theta:(X_1,\nu_1)\to (X_2,\nu_2)$ between two measure spaces
is called {\em $\epsilon$-measure preserving} if there exists a set $E_1\subset  X_1$, 
$\nu_1(E_1)<\epsilon$ and such that for every $A\in X_1\setminus E_1$, we have
$$
\left|\frac{\nu_2(\theta(A))}{\nu_1(A)}-1\right|<\epsilon.
$$
\end{definition}

 The following fact will be useful in constructing $\epsilon$-measure preserving maps. By {\em Lebesgue
space} we mean a probability measure defined on a Borel $\sigma$-algebra in a Polish space.

\begin{lemma}\label{LmUniqueLeb}
(\cite{Rokhlin})
Any two atomless Lebesgue spaces are isomorphic.
\end{lemma}

\subsection{$K$ and Bernoulli properties}\label{sec:VWB}
Let $(X,\cB,\mu)$ be a standard probability Borel space and let $\cP=(P_1,\ldots P_k)$ and $\mathcal{Q}=(Q_1,\ldots, Q_k)$ be two finite measurable partitions of $X$. Let $\cP\vee\mathcal{Q}$ be the partition into sets of the form $P_i\cap Q_j$, $i,j\in \{1,\ldots, k\}$. Let $T$ be an automorphism of $(X,\cB,\mu)$. We say that $\cP$ is {\em generating} if $\bigvee_{-\infty }^{+\infty}T^i\cP=\cB$.

We say that a property holds for $\epsilon$ a.e.\ atom of a partition $\mathcal{Q}$ if the union of all atoms for which the property does not hold has measure $\leq \epsilon$.

We recall the definition of $K$-property:
\begin{definition}\label{def:K2} Let $T$ be an automorphism of $(X,\cB,\mu)$ and let $\cP$ be a finite partition of $X$. We say that $\cP$ is a $K$-partition if for every $D\in\bigvee_{-\infty}^{+\infty}T^i\cP$ and every $\epsilon>0$ there exists $N_0=N_0(\epsilon,D)$ such that for every $N'\geq N\geq N_0$, $\epsilon$ a.e.\ atom $A\in \bigvee_N^{N'}T^i\cP$ satisfies
$$
\left|\frac{\mu(A\cap D)}{\mu(A)}-\mu(D)\right|<\epsilon.
$$
We say that $T$ has the {\em Kolmogorov property} ($K$ property) if there exists a generating $K$-partition. It then follows that {\em every} partition is a $K$-partition.
\end{definition}
We will need the following simple modification of the original definition of the $K$-property:
\begin{lemma}\label{def:K}
$T$ has the   $K$ property if there exists a {\em generating partition} $\cP$ such that for every $D\in \cB$ and every $\epsilon>0$ there exists $n\in \N$ and $N_0=N_0(\epsilon,D)$ such that for every $N'\geq N\geq N_0$, $\epsilon$ a.e.\ atom $A\in \bigvee_N^{N'}T^i\cP$ satisfies
\begin{equation}\label{eq:spi}
\left|\frac{\mu(A\cap f^{-n}D)}{\mu(A)}-\mu(D)\right|<\epsilon.
\end{equation}
\end{lemma}
\begin{proof} The proof is an immediate consequence of the following observation: if $A\in\bigvee_N^{N'}T^i\cP$, then $A=f^n(A')$ where $A'\in \bigvee_{N-n'}^{N'-n}T^i\cP$. Since $\mu(A\cap D)=\mu(A'\cap f^{-n}D)$ it easily follows that the statement of the lemma is equivalent to the original definition of the $K$ property.
\end{proof}

By \cite{RS61} 
the $K$ property is equivalent to {\em completely positive entropy}: every {\em factor} of $T\in Aut(X,\cB,\mu)$ has positive entropy.

\begin{definition}\label{def:Bernoulli} $T\in Aut(X,\cB,\mu)$ is Bernoulli if it is measure theoretically isomorphic to the {\em Bernoulli shift}, i.e.\ the shift map on the space $(\{1,\ldots,\ell\}^\Z, {\bf p}^\Z)$, where ${\bf p}=(p_1,\ldots,p_\ell)$ is a probability vector.
\end{definition}

For $A\subset X$, $\cP_{|A}$ denotes the induced partition of  the space $(A, \mu_{|A})$, 
i.e. 
$$\cP_{|A}:=(P_1\cap A,\ldots, P_k\cap A) \text{ and }
\mu_{|A}(B)=\frac{\mu(A\cap B)}{\mu(A)}.$$
We introduce the following distance on the space of partitions of $(X,\mu)$:
$$
\bar{d}(\cP,\mathcal{Q}):=\sum_{i=1}^{k}\mu(P_i\triangle Q_i).
$$
Now let $\cP^{s}=(\cP^{s}_1,\ldots,\cP^{s}_k)$, $s=1,\ldots, S$ be a sequence of finite partitions of $(X,\mu)$ and $\mathcal{Q}^s=(\mathcal{Q}^s_1,\ldots,\mathcal{Q}^s_k)$, $s=1,\ldots, S$ be a sequence of finite partitions of $(Y,\nu)$.  If  additionally $(X,\mu)=(Y,\nu)$, then
$$
\bar{d}\left((\cP^s)_{s=1}^S,(\mathcal{Q}^s)_{s=1}^S\right):=\frac{1}{S}\sum_{s=1}^S\bar{d}(\cP^s,\mathcal{Q}^s).
$$
More generally, if $(\cP^s)_{s=1}^{S}$ and $(\mathcal{Q}^s)_{s=1}^S$ are partitions of different spaces,  we say that $\cP^s\sim \mathcal{Q}^s$ for $s=1,\ldots, S$ if
$\mu(P_i^s)=\nu(Q^s_i)$ for $i=1,\ldots,k$ and $s=1,\ldots, S$. We can then compare the distance between $(\cP^s)_{s=1}^S$ and $(\mathcal{Q}^s)_{s=1}^S$ by setting
$$
\bar{d}\left((\cP^s)_{s=1}^S,(\mathcal{Q}^s)_{s=1}^S\right)=\inf_{\mathcal{\bar{Q}}^s\sim \mathcal{Q}^s,\; s=1,\ldots,S}\bar{d}\left((\cP^s)_{s=1}^S,(\mathcal{\bar{Q}}^s)_{s=1}^S\right),
$$
where the infimum is taken over sequences of partitions $\mathcal{\bar{Q}}^s$ of $(X,\mu)$.
We denote by $T^n\cP$ the partition given by $(T^nP_1,\ldots, T^nP_k)$.
\begin{definition}[Very weak Bernoulli,vwB]\label{def:VWB} Let $T\in Aut(X,\cB,\mu)$ and let $\cP$ be a finite partition of $X$. Then $\cP$ is a very weak Bernoulli partition (vwB partition) if for every $\epsilon>0$ there exists $N_0\in \N$ such that  for every $N'\geq N\geq N_0$ every $S\geq 0$ and $\epsilon$ a.e.\ atom $A$
  of $\bigvee_{N}^{N'}T^{i}(\cP)$, we have
$$
\bar{d}\left(\{T^{-i}\cP\}_{s=0}^S,\{T^{-i}\cP_{|A}\}_{s=0}^S\right)<\epsilon.
$$
\end{definition}
The following classical theorem is a crucial tool in establishing Bernoullicity of a system (see e.g.\ \cite{OrnsteinWeiss}). Recal that a sequence of partitions $(\cP_k)_{k=1}^{+\infty}$ of $(X,\cB,\mu)$ {\em converges to partition into points} if the smallest $\sigma$-algebra with respect to which all $\cP_k$  are measurable, is $\cB$.
\begin{theorem}\label{conga} If $(\cP_k)_{k=1}^{+\infty}$ is a sequence of partitions of $(X,\cB,\mu)$ converging to partition into points and, for every $k\geq 1$, $\cP_k$ is VWB partition for $T\in Aut(X,\cB,\mu)$, then $T$ is a Bernoulli system.
\end{theorem}
We will now recall the main method of establishing VWB property, \cite{OrnsteinWeiss}. For a partition $\cP=(P_1,\ldots P_k)$ of $(X,\mu)$, an integer $S\geq 1$ and $x\in X$ the $S,\cP$-name of $x$ is a sequence $(x^\cP_i)_{i=0}^S\in \{1,\ldots,k\}^{S+1}$  given by the condition $T^{i}(x)\in P_{x^\cP_i}$. Let $e:\Z\to \Z$ be given by $e(0)=0$ and $e(n)=1$ for $n\neq 0$.

We have the following lemma
\begin{lemma}[Lemma 1.3. in \cite{OrnsteinWeiss}]
\label{lem:VWE} 
Let $T\in Aut(X,\cB,\mu)$ and $\cP$ be a finite partition of $X$. If for every $\epsilon>0$ there exists $N\in \N$ such that for every $N'\geq N$, $\epsilon$ a.e.\ atom $A\in \bigvee_{N}^{N'}T^i\cP$ and every $S\geq 1$ there exists an $\epsilon$-measure preserving map $\theta=\theta(N,S,A):(A,\mu_{|A})\to (X,\mu)$ such that
\be\label{eq:as123}
\bar{d}_S(x,\theta(x)):=\frac{1}{S}\sum_{i=0}^{S-1}e\left(x_i^\cP-(\theta(x))^{\cP}_i\right) <\epsilon.
\ee
then $\cP$ is a VWB partition.
\end{lemma}

A finite partition $\cP$ of $(X,\cB,\mu,d)$ is called {\em regular} if for every $\epsilon>0$ there exists $\delta>0$ such that 
$$
\mu(V_\delta(\partial \cP))<\epsilon,
$$
where for $A\subset X$, $V_\delta(A)$ denotes the $\delta$ neighborhood of $A$ (in the metric $d$). For existence results of regular partitions we refer the reader to e.g. \cite[Lemma 4.1]{OW2}. In what follows we will always consider only regular partitions.

We shall use the following form of Lemma \ref{lem:VWE}.

\begin{corollary}
\label{cor:VWE} 
Let $f:(X,\cB,\mu,d)\to (X,\cB,\mu,d)$ satisfy the $K$-property
and $\cP$ be a regular partition of $X$. If for every $\epsilon>0$ there exists $N,\tilde{N}\in \N$ such that for every $N'\geq N$, $\epsilon$ a.e.\ atom $A\in \bigvee_{N}^{N'}T^i\cP$ and every $S\geq \tilde{N}$ there exists an $\epsilon$-measure preserving map $\theta=\theta(N,S,A):(A,\mu_{|A})\to (X,\mu)$ such that for everu $x\in A$,
\be\label{MetClose}
\frac{1}{S}\Card\left(\{i\in \{0,\ldots, S-1\}\;:\; d(f^ix,f^i(\theta x))<\epsilon\}\right)>1-\epsilon,
\ee
then $\cP$ is a VWB partition.
\end{corollary}
\begin{proof}
Fix $\epsilon>0$.  We shall assume that \eqref{MetClose} holds with $\eps'=\eps^{10}$ instead
of $\eps$ and verify that the assumptions of Lemma \ref{lem:VWE} hold.
Since $\cP$ is regular, there exists  $\delta>0$ such that $\mu(V_\delta(\cP))\leq \epsilon^8$. By ergodicity of $f$ (recall that $f$ is $K$), there exists $m\in \N$ and a set $E\subset X$, $\mu(E)\geq 1-\epsilon^4$ such that for every $x\in E$ and every $n\geq m$,
\be\label{eq:vx}
\Big|\{0\leq i\leq n-1\;:\; f^i(x)\notin V_\delta(\cP)\}\Big|\geq (1-\epsilon^2)n.
\ee
Let $N$ and $\tilde{N}$ come from the assumptions of the statement of the corollary for $\epsilon'=\min(\delta,\epsilon^{10})$. We assume WLOG that $\tilde{N}\geq m$. We will first show that for $S\leq \tilde{N}$, the assumptions of Lemma \ref{lem:VWE} are satisfied by the $K$-property. Let $\{D_i\}_{j=1}^{\tilde{J}}$ be all the atoms of
 $$\bigvee_{-\tilde{N}-1}^{0}f^i\cP.$$
 Notice that by definition, if $x,y\in D_i$, then $f^jx$ and $f^jy$ are in the same atom of $\cP$ for ever $0\leq j\leq \tilde{N}$. By the $K$-property (see Definition \ref{def:K2}) for $D=D_i\in \bigvee_{-\infty}^{+\infty}f^i\cP$, $i=1,\ldots, \tilde{J}$ (simultaneously), it follows that there exists $N_0$, such that for every $N'\geq N_0$, $\epsilon^{2}$ a.e. atom $A\in \bigvee_{N_0}^{N'}f^i\cP$ satisfies: for every $i\leq \tilde{J}$, we have
 \be\label{eq:as12}
 \left|\frac{\mu(A\cap D_i)}{\mu(A)}-\mu(D_i)\right|\leq \epsilon^2.
 \ee
 For $S\leq \tilde{N}$ and $i\leq \tilde{J}$, let $\theta_i=\theta_i(N,S,A):(A\cap D_i, \mu_{|A})\to (D_i,\mu_{|D_i})$ be any $\epsilon$-measure preserving map (such map exists by \eqref{eq:as12} and Lemma \ref{LmUniqueLeb}). We then naturally define for $x\in A$, $\theta(x)=\theta_i(x)$, where $x\in A\cap D_i$. By definition, $\theta$ is $\epsilon$ measure preserving and moreover for $x\in A\cap D_i$, $\theta(x)\in D_i$. This however implies that $\bar{d}_S(x,\theta(x))=0$, by the definition of $D_i$ and since $S\leq \tilde{N}$. This implies that Lemma \ref{lem:VWE} holds for $S\leq \tilde{N}$.

 Consider now the case $S>\tilde{N}$. Let $N\geq N'$. We say that an atom $A\in \bigvee_{N}^{N'}f^i\cP$ is {\em good} if it satisfies  \eqref{MetClose}
 and moreover
 $$
 \mu(A\cap E)\geq (1-\epsilon^2)\mu(A).
 $$
 Since $\mu(E)\geq 1-\epsilon^4$ 
 it follows that $\epsilon$ a.e. atom $A\in  \bigvee_{N}^{N'}f^i\cP$ is good.
 For a good $A$, let $\bar{\theta}=\theta(N,S,A):(A,\mu_{|A})\to (X,\mu)$ be the $\epsilon'$- measure preserving map from the assertion of Corollary \ref{cor:VWE}.
 We define $\theta(x):=\bar{\theta}(x)$ if $x\in A\cap E$ and $\theta(x)=x$ otherwise.
 Since $\mu(A\cap E)\geq (1-\epsilon^2)\mu(A)$ it follows that $\theta:(A,\mu_{|A})\to (X,\mu)$ is $\epsilon$ measure preserving.
 Now, if $x\notin A\cap E$, then $\theta(x)=x$ and so \eqref{eq:as123} trivially holds. If $x\in A\cap E$ then  \eqref{MetClose} holds for $x$ and $\theta(x)$ and $\epsilon=\epsilon'$.
 But since $x\in E$, we get that \eqref{eq:vx} holds for $x$ and $S\geq \tilde{N}\geq m$. Therefore, if $i\leq S-1$ is a time satisfying \eqref{MetClose} and \eqref{eq:vx},
 then $f^ix$ and $f^i(\theta(x))$ are in the same atom of $\cP$ (since $\epsilon'\leq \delta$).
 It remains to notice that the total cardinality of such $i\leq S$ is at least $(1-\epsilon)S$. This shows \eqref{eq:as123} and hence finishes the proof of Corollary \ref{cor:VWE}.
\end{proof}

\section{Exponential mixing implies non-zero exponents}\label{sec:posent}
{ In this section we will prove Proposition \ref{PrLyapNZ}. Let $f:(M,\mu)\to (M,\mu)$ be exponentially mixing. We will show that $f$ has at least one non-zero Lyapunov exponent. In this section we don't need to assume that the measure $\mu$ is smooth; it is enough that  $\supp(\mu)$ is not a single point.



For $z\in M$ and for $r>0$, let $O_{r}(z)$ be the ball of radius $r$ centered at $z$. We have the following lemma:
\begin{lemma}\label{cor:CE} If $f$ is exponentially mixing for some non-atomic measure 
then there exist $c,{ \heta}>0$ such that for every $B\subset M$ with $\mu(B)\geq 1-c$, 

\begin{equation}\label{eq:stre}
\inf_{n\in \N}\max_{z\in B}\diam\Big(f^n\Big(O_{e^{-\heta n}}(z)\Big)\Big)\geq c.
\end{equation}
\end{lemma}

Note that \eqref{eq:stre} is equivalent to saying that for each $n$
$$ \mu\left(x: \diam\left(f^n\left(O_{e^{-\hat{\eta} n}}(z)\right)\right)\geq c\right)\geq c. $$
In other words, the lemma says that the image of balls of radius $e^{-\heta n}$ become macroscopic
for a set of centers of a sizable measure.
We will prove Lemma \ref{cor:CE} in Appendix~\ref{AppEM}.

}

\begin{proof}[Proof of Proposition \ref{PrLyapNZ}:]
{Assume by contradiction that all exponents of $f$ are non positive and let $\zeta=\heta/1000$. 
Let $\zeta'>0$ be such that $[\sup \|Df_x\|]^{\zeta'}<e^{\zeta}$ and $\zeta'<c/2$  where $c$ is from Lemma \ref{cor:CE}.

By Oseledets theorem 
there exists $C>0$, a set $A\subset M$, $\mu(A)\geq 1-\zeta'/2$ and $n_1\in \N$ such that for every $x\in A$ and every $n\geq n_1$,
$$
\|Df_x^n\|<Ce^{\zeta n}
$$
By enlarging $C$ if necessary, we may assume that the above also holds for $n\leq n_1$. Using ergodic theorem for the function $\chi_A$ it follows that there exists $n_2$ and a set $B$, $\mu(B)\geq 1-c$ such that for every $x\in B$ and every $n\geq n_2$,
$$
\Big|\{0\leq m\leq n\;:\; f^mx\in A\}\Big|\geq (1-\zeta')n.
$$
Then it follows that for  every $x\in B$, $n\geq n_2$, and every $\ell \leq n$, 
\begin{equation}\label{eq:erg}
\|Df_{f^\ell x}^{n-\ell}\|\leq Ce^{2\zeta n}.
\end{equation}
Indeed, for $\ell \leq n$, let $m\geq \ell$ be the smallest  number
such that $f^mx\in A$. Since $x\in B$ it follows that $|m-\ell|\leq \zeta' n$. Using that $f^mx\in A$ and the definition of $\zeta'$, we get 
$$
\|Df_{f^\ell x}^{n-\ell}\|\leq \|Df_{f^{\ell}x}^{m-\ell}\|\cdot \|Df_{f^mx}^{n-m}\|\leq [\sup \|Df_x\|]^{\zeta' n}\cdot Ce^{\zeta n-m}\leq 
Ce^{2\zeta n}.
$$
Moreover, by enlarging $C$, we can assume that \eqref{eq:erg} holds also for $n\leq n_2$.

We will show that for every $n$ sufficiently large, \eqref{eq:stre} does not hold. For $x\in B$ let $k=k(x,n)$ be the smallest number such that 
$$
\diam\Big(f^k\Big(O_{e^{-\heta n}}(x)\Big)\Big)\geq e^{-\frac{1}{2}\heta n}.
$$
 By Lemma \ref{eq:stre} there is $x\in B$ such that $k(x,n)\leq n.$
Then by the definition of $k$ it follows that for every $y\in O_{e^{-\heta}n}(x)$ and every $0\leq i< k\leq n$,
\begin{equation}\label{eq:dera1}
d(f^ix,f^iy)\leq e^{-\frac{1}{2}\heta n}.
\end{equation}
We will show that  there is a constant $C'$ such that for every $z\in O_{e^{-\heta}n}(x)$ and every $0\leq i\leq k\leq n$,
\begin{equation}\label{eq:sd1}
\|Df_z^i\|\leq C'e^{3\zeta k}.
\end{equation}
Before we prove \eqref{eq:sd1}, let us show how it implies the proposition.}

Notice that if $k\leq n$ then using \eqref{eq:sd1} and the mean value theorem, for every $i\leq k$ and every $z\in O_{e^{-\heta}n}(x)$
$$
d(f^kx,f^kz)\leq C^*e^{3\zeta k}d(x,z)\leq  C^* e^{3\zeta k-\heta n}
\leq e^{-\frac{2}{3}\heta n}<e^{-\frac{1}{2}\heta n},
$$
which contradicts the definition of $k$. 
 This contradiction shows that $f$ in fact has non-zero exponent proving the proposition.

It remains to establish \eqref{eq:sd1}. We will proceed by induction on $i$. The case $i=1$ just follows by taking $C'\geq \sup_{x\in M} \|Df_x\|$. For $j\in \N$, let $A_j:=Df_{f^jx}$ and $B_j=B_j(z)=Df_{f^jz}$ and assume that \eqref{eq:sd1} holds for all  $\ell\leq i-1$. This implies that for every $\ell \leq i-1$, 
$$
\left\|\prod_{j=0}^{\ell-1}B_j\right\|\leq C'e^{3\zeta k}.
$$
Then 
\begin{eqnarray*}
\|Df_x^i-Df_z^i\|=\left\|\prod_{j=0}^{i-1}A_j-\prod_{j=0}^{i-1}B_j\right\|
&=& \left\|\sum_{\ell=0}^{i-1}\left(\prod_{j=\ell+1}^{i-1}A_j\right)(A_\ell-B_\ell)\prod_{j=0}^{\ell-1}B_j\right\|\\
&\leq&\sum_{\ell=0}^{i-1}\left\|\prod_{j=\ell+1}^{i-1}A_j\right\|\|A_\ell-B_\ell\|\left\|\prod_{j=0}^{\ell-1}B_j\right\|.
\end{eqnarray*}
Note that since $\ell\leq i-1\leq k-1$, it follows by \eqref{eq:dera1} that
$$
\|A_\ell-B_\ell\|\leq C''e^{-\frac{1}{2}\heta n}.
$$
Moreover, by \eqref{eq:erg} (using that $i\leq k$),
$$
\left\|\prod_{j=\ell+1}^{i-1}A_j\right\|\leq C\cdot e^{2\zeta k}
$$
Using the inductive assumption and the above bound (and $i\leq k$) we get that 
$$
\|Df_x^i-Df_z^i\|\leq CC'C''i\cdot e^{2\zeta k- \frac{1}{2}\heta n+ 3\zeta k}\leq C' e^{3\zeta k},  
$$
since $i\leq k\leq n$ and $\zeta=\heta/1000$.
This proves \eqref{eq:sd1} and finishes the proof of the proposition.
\end{proof}

 Combining Proposition \ref{PrLyapNZ}, with Pesin entropy formula we see that
if $f$ is exponentially mixing for a smooth measure $\mu$ then it has positive entropy.
Our main result, Theorem \ref{thm:main}  gives a much stronger conclusion, namely $(f, \mu)$ is Bernoulli.
In particular, the exponentially mixing diffeos with the same entropy are 
isomorphic.

\section{Almost u-saturation}\label{sec:u}
 In Sections \ref{sec:u}--\ref{ScAC} we assume that 
$f$ is $C^{1+\alpha}$ diffeomorphism
of a compact manifold $M$ preserving a smooth measure $\mu$, such that at least one
Lyapunov exponent of $f$ is non-zero. In particular, the assumption that $f$ is exponentially 
mixing will not be used until Section \ref{sec:main} . 

\begin{definition}\label{def:sat}
A set $A\subset M$ is called $(\epsilon, \xi)$ u-saturated 
if there exists a set $E\subset A$ 
 with $\mu|_A(E)\geq 1-\epsilon$ such that if $x\in E$ and $\cW$ is a unstable
 box of size $\xi$ containing $x$ then $\cW\subset A.$

A partition $\Pi$  is called $(\epsilon, \xi)$ u-saturated 
if there exists a set $E\subset M$ 
 with $\mu(E)\geq 1-\epsilon$ and such that if $x\in E$ and $\cW$ is a unstable
 box of size $\xi$ containing $x$ then $\cW\subset \Pi(x).$
\end{definition}

We note that if $\Pi$ is $(\xi, \epsilon^2)$ u-saturated then $\epsilon$ almost every atom of $\Pi$ is
$(\xi, \eps)$ u-saturated. We will need the following  lemma which uses exponential contraction of the unstable foliation. 
Let $\cP$ be any  partition of $M$ with smooth boundaries.

\begin{lemma}\label{lem:satu}
For every  $\epsilon>0$ there exists $\xi=\xi(\epsilon)>0$ 
  there exists  $\brN=N_{\xi,\epsilon}\in \N$
  such that for every
  $N_1\geq N_2\geq \brN,$
the partition  
$\DS \fP=\bigvee_{i=N_1}^{N_2}f^i(\cP)$ 
is $(\xi, \epsilon)$ u-saturated.
\end{lemma}
\begin{proof} The proof is similar to the proof of Lemma 2.1. in \cite{OrnsteinWeiss}. 
 
Let $\tilde{P}_{\tau^{-1}}$ be the set from \eqref{eq:hpk} where $\tau$ such that $\mu(\tilde{P}_{\tau^{-1}})\geq 1-\eps^4$ and let $\xi=\tau.$

 Notice that if for some $k\geq N_1$ and some atom $A\in f^k(\cP)$ we have $x\in \tilde{P}_\tau\cap A$ but $W^{u}_{x,\xi}\notin A$, then by exponential contraction,  
 $d(f^{-k}x,\partial \cP)\leq c^{-k}$ where $c=e^{\delta-\lambda}<1$. 
Since $\cP$ is piecewise smooth,
 the measure of points
 satisfying the last condition for some $k\geq N_1$ is at most
$\DS C \sum_{k=N_1}^{+\infty}c^{-k}$ for some $C>0.$ 
The last expression can be made smaller than $\eps^4$  by taking $N_1$ large enough.
In summary, the set of points $x$ such that $\fP(x)\not\supset W^u_{x, \xi}$ has measure which is 
smaller than $2\eps^4.$ By Markov inequality, $\fP$ is  $(\xi, \eps)$ u-saturated.
\end{proof}

\begin{lemma}
\label{LmUDeco}
  If $A$ is $(\xi, \eps)$ u-saturated and 
\be\label{eq:atomL}\mu_{|A}(P_\tau)\geq 1-\epsilon^2,
\ee
 then we can decompose
$ \mu|_A=\mu_g+\mu_r$ so that $\mu_r(M)\leq \epsilon$ and 
there is a family $\{\cW_t\}_{t\in \fT}$ of  unstable boxes of size $\xi$ 
and a measure $\nu$ of $\fT$
such that
\be\label{UDeco} \mu_g(B)=\int_{\fT} m^u_{\cW_t}(B) d\nu(t) \ee
\end{lemma}

\begin{proof}
Let $\bar{E}_1=A\setminus P_\tau$ and $\bar{E}_2=E^c$, and $\bar{E}= \bar{E}_1\cup \bar{E}_2$
where $E$ is from Definition \ref{def:sat} for $A$ and $E^c$ is the complement of $E.$ 
Since $A$ is $(\xi,\epsilon)$ saturated and \eqref{eq:atomL} holds for $A$, we have 
that $\mu_{|A}(\bar E)\leq \epsilon$.
It then follows that $A\setminus \bar{E}=\bigcup_{x_i}\bigcup_{z\in T_{x_i}\cap P_\tau}W^u_{z,\xi}\cap P_\tau$.


Now absolute continuity of unstable foliation on $P_\tau\subset \hP_\tau$ (see Lemma \ref{lem:abscont0}) gives
$$ \mu(B)=
\mu|_A\left(B\cap \bar E \right)+
\sum_i
\int_{T_{x_i}\cap P_\tau} d\nu_i(z)m^u(W^{u}_{z,\xi}\cap A\cap P_\tau) 
\mu_{\cW}(B). $$
By the definition of $\bar E$, the first term is smaller than $\eps$. This gives \eqref{UDeco} with 
\\
$\DS \fT=\bigcup_i \bigcup_{z\in T_{x_i}\cap P_\tau} z\quad\text{and}\quad
d\nu=\sum_i d\nu_i.
$
\end{proof}

\begin{lemma}
\label{LmLocMeas}
Let $\fB$ be as set with $\mu(\fB)\leq \hat{\eps}^4.$ Then there is $\xi_0$ such that for $\xi\leq \xi_0$
the set 
$$\cK=\{x\in M\;:\; r_u(x)\geq \xi\text{ and } m^u_{W^u_{x,\xi}} (\fB)\leq \hat{\eps})\}$$ 
has measure greater than $1-4\hat{\eps}$ where $r_u(x)$ is the size of the unstable manifold of $x.$
\end{lemma}

\begin{proof}
Take a small $\brxi$. Then Lemma \ref{LmUDeco} gives a decomposition
$\mu=\mu_g+\mu_r$ where $\mu_r(M)\leq \hat{\eps}$ and 
$$ \mu_g(A)=\int_{\fT} m^u_{\cW_t}(A) d\nu(t) $$
where $\cW_t$ are unstable boxes of size $\brxi.$ By Markov inequality we may assume, 
possibly increasing $\mu_r(M)$ by $\hat{\eps}$,
that
for each $t\in \fT$ we have $m^u_{\cW_t}(\fB^c)>1-\hat{\eps}^3.$
Next take $\xi_0$ so small that
$m_{\cW_t}^u(\partial_{\xi_0} \cW_t)\leq \hat{\eps}. $ 
Now for each $t\in \fT$ we apply Lemma \ref{LmUniCover} with 
$$\nu=m^u_{\cW_t}, \quad
D=\{x\in \cW_t:\;\; d(x, \partial \cW_t)\geq \xi\}, \quad 
 B=\fB $$
and conclude that for each $t$ we have 
$$\mu^u_{\cW_t}(\cK^c)\leq \mu^u_{\cW_t}(\cK^c\cap D)+\mu^u_{\cW_t}(D^c)\leq 2\hat{\eps}.$$
Integrating over $t\in \fT$ and remembering that $\mu_r(M)<2\hat{\eps}$ we obtain the result. 
\end{proof}

\section{Construction of fake center-stable foliation.}\label{sec:fake}
 The results of this section are obtained under the assumption that
$f$ is $C^{1+\alpha}$ diffeomorphism
of a compact manifold $M$ preserving a smooth measure $\mu$
with some non-zero exponents. The reader may suppose in the arguments below
that $\eta_2$ is a small constant and $\eps\ll \eta_2.$ Later we apply the
results of this section in the case where $\eta_2$ satisfies the conditions of Lemma 
\ref{lem:expmixB}. Before we go to the details of the construction we provide an outline.\\

{\bf Outline of the construction of the fake cs-foliations:} The constructed foliation depends on the time parameter $n$. The idea is to choose a regular reference point $x$
and to bring back  the foliation  $\exp_x( \R^{cs})$
from time $n$ to time $0$. The problem with this approach is that 
the leaves constructed using two different reference points $x'$ and $x''$
may intersect. To overcome this problem we will create small buffers between the regions
where we take different foliations. This will resolve the problem of intersections at the 
price that our foliation will be defined not on all of $M$ but on a set whose measure could be
made arbitrary close to 1.

More precisely, we fix a family $\{\bar{B}_s\}=\{B_s(\xi_n,r)\}$ (see \eqref{BXiR})  of disjoint parallelograms centered at points in $\cL_{n,\tau}$ (see \eqref{eq:lnc})  and with sizes $\bar{\xi}_n=e^{-\eps^3n}$  in the $\R^u$ direction and $r_n=\bar{\xi}_n^{1+\omega}$ (where $\omega$ is sufficiently small)
in the $\R^{cs}$ direction. The exact choice of the size of $\bar{\xi}_n$ is not so important as long as it is exponentially small with $n$ (with the exponent much smaller than the exponential mixing exponent). We want the $r_n$ to be smaller than $\xi_n$ to be able to use absolute continuity of the unstable foliation. We also want that $\{\bar{B}_s\}$ cover most of the space (i.e.\footnote{ Recall that $\bb=10^{10}.$}
 $1-30\epsilon^{\bb/4}$). Existence of such a family is established in Lemma \ref{lem:setsB}. We now take $\xi_n:= e^{\eps^2n-\eta_2 \eps n}$ and $\eps'=\eps^{100\bb}$. We say that an unstable manifold $\cW$ of size $\xi_n$ is {\em good} if 
 $$
 m^u_{\cW}\Big(\bigcup_{s}f^{-n}\Big((1-\eps')\cdot \bar{B}_s\Big)\cap \cL_{n,\tau}\Big)\geq 1-100\eps^{\bb/16}.
 $$ Using again Lemma \ref{lem:setsB} we show that there exists a family $\{\hcW_i=W_i(\xi_n)\}$ of good unstable manifolds such that the corresponding parallelograms $\{B_i=B_i(\hcW_i,r)\}$ with $r\leq e^{-\eta_2\eps n}$ cover most of the space.  In fact in the applications in Sections \ref{ScAC} and \ref{sec:main} we will only work with two sizes of $r=e^{-\eta_2\eps n}$ and $\tilde{r}=e^{-\eta_2\eps n-\eps^2n}$. The extra $-\eps^2n$ term in $\tilde{r}$ is responsible for the fact that if we start with two points on the same fake center-stable leaf which are at distance $\tilde{r}$ then after  $\eps \cdot n$ iterates their distance will be less than $r$. This will be used in Proposition \ref{prop:3.5}.
Next, for each $i$ and $s$ we look at the maximal connected components $\{R_{i,s,j}\}$ of the set $\hcW_i\cap f^{-n}((1-\eps')\bar{B}_s)$. 
 For each such connected component $R_{i,s,j}$ we construct a fake cs-foliation by pulling the $\R^{cs}$ foliation for the set $f^{n}(\hcW_i)\cap \bar{B}_s$ (see \eqref{eq:fak1}, \eqref{eq:fak2}). 
Here the fact that we  hit $\bar{B}_s$ not too close to the boundary will create aforementioned buffers, while the fact that $\eps'$ is small, will ensure that our foliations are defined on
 a set of a measure close to 1. Also we want $\bar{B}_s$ to be much larger than $B_i$ to ensure
 that the unstable leaves we construct will fully cross $B_i$.
  The next step is to glue the foliations for different $s,j$ to get a foliation of (most of) the sets $B_i(\hcW_i,r)$. The crucial properties of the foliation $\{\cF_{i,s,j}\}$ are established
  in Lemma \ref{lem:ff}. Namely, we show that indeed one can glue the foliations over different $s,j$ as they don't intersect (property {\bf F5.}). We also show that the growth on the leaves of the foliation $\{\cF_{i,s,j}\}$ is sub-exponential (property {\bf F4.}). In property {\bf F3} we show that most of the set $\hcW_i$ 
intersects a leaf of $\cF_{i,s, j}$.
The crucial property of the constructed foliation is that it is locally absolutely continuous (on exponentially small scale), see Proposition \ref{lem:abscont}.  The bounded distortion property  (property {\bf F2} of Lemma \ref{lem:ff})
plays a key role in proving Proposition~\ref{lem:abscont}. \\

We now move to the details of the construction.


 Recall that for $x\in \R^D$, $\DS \|x\|_{\infty}=\max_{i}|x_i|$.   For $\xi<\tau$, $x\in P_\tau$ let $\cW(\xi)\subset W^u_{z,\tau}$ be an  unstable cube of $z$ of size $\xi$. Let
\begin{equation}
\label{BXiR}
B(\xi,r)=B(\cW(\xi),r):=h_z\left( 
\{(a, b)\in \R^u\times \R^{cs}: \|a\|_{\infty}\leq \xi, \|b-\eta_z(a)\|_{\infty}\leq r\}\right)
\end{equation}
 where the unstable manifold of $z$ is given by
the graph of $\eta_z$ in the exponential coordinates (see \eqref{eq:etx}).
 By Lemma \ref{lyapchart} 
 there exists a measurable function $\brho:LyapReg\to \R$ and $\alpha_{\brho}>0$ such that for every $\tau>0$ the restriction of $\brho$ to $P_\tau$ is H\"older continuous with exponent $\alpha_{\brho}$ (independent of $\tau$) and, by Fubini theorem
for each $\heps$ and $\tau>0$ and any $\xi$ and $r$ small enough (in terms of $\heps$ and $\tau$) so that
$$ \mu(B(\xi,r))\in (1-\heps, 1+\heps)\brho(x)\cdot\int_{\|a\|_{\infty}\leq \xi} \mes(B^{cs} (\eta_z(a), r)) da
$$
\be\label{MesMultiDisc}
=(1-\heps, 1+\heps)\brho(x)\cdot m^u(\cW(\xi)) \mes(B^{cs} (0, r)),
\ee
for arbitrary $x\in B(\xi,r)\cap P_\tau$. 
In fact it is enough to take $\brho(x):=\det(L_x)$ (see Lemma~\ref{lyapchart} and Lemma \ref{LmIDHold}).


Fix $\epsilon$ and let $\tau$ be as in \eqref{eq:pc}. For $n\in \N$ we define 
\be\label{eq:lnc}
\cL_{n,\tau}=P_\tau\cap f^{-n}P_\tau\cap f^{-\epsilon n}P_\tau.
\ee
Note that by \eqref{eq:pc}, $\mu(\cL_{n,\tau})\geq 1-3\eps^\bb$. 
\begin{lemma}
\label{lem:setsB}
There exists $\omega>0$ and  $n_{\epsilon}\in \N$ such that for every  $n\geq n_{\epsilon}$  and for any
$\xi\in (0,e^{-\eps^3n})$,
$r\in [\xi^{1+\omega},\xi]$
 and for any set $H\subset M$ there exists a finite family of pairwise disjoint sets $\mathbb{B}=\{B_i(\xi,r)\}_{i\in J_r}=\{B(\cW_i(\xi),r)\}$ such that for every $i\in J_r$, 
$\cW_i(\xi)$ satisfies 
\be\label{eq:largeL}
m^u_{\cW_i(\xi)}(H)\leq \mu(H)^{1/4}.
\ee
and 
\be \label{BigFamily}
\mu\left(\bigcup_{i\in J_r}B_i(\xi,r)\right)\geq 1-10\mu(H)^{1/4}.
\ee
\end{lemma}
\begin{proof}
 Let $\cK=\{x\in P_\tau: m^u_{\cW_{x, \xi}}(H)\leq \mu(H)^{1/4}\}.$
By Lemma \ref{LmLocMeas}, $\mu(\cK)>1-4\mu(H)^{1/4}.$
Divide $M$ into cubes $Q_j$ of size $R=\xi^{1-10\omega}$
so that each cube belongs to a single coordinate chart.
In each cube which intersects $\cL_{n,\tau}$, we choose a point $z_j\in \cL_{n,\tau}.$ 
Then divide each cube into bricks $Z_{ij}$ which are products of cubes of size $(1+\heps^3) \xi$
in $E^u(z_j)$ direction and cubes of size $(1+2\heps) r$ in $E^{cs}(z_j)$ direction. 
The centers $z_{ij}$ of $Z_{ij}$ lie on a lattice. Shifting that lattice if necessary we may assume that
the proportion of $z_{ij}$ which is not in $\cK$ is at most 
$\DS 2\frac{\mu(\cK^c\cap Q_i)}{\mu(Q_i)}.$ Now for each $z_{ij}\in \cK$ we add 
$B_{ij}=B(W_{z_{ij}}(\xi), r)$
into our collection $\mathbb{B}.$ 
 By \eqref{eq:etau} it follows  that for $\|a\|\leq \xi$ we have
$\|\eta_z(a)\|\leq C \xi^{1+\alpha_4}.$ Hence choosing $\omega<\alpha_4$ ensures that
$B_{ij}\subset Z_{ij}$ and $\mu(B_{ij})>(1-\heps^2) \mu(Z_{ij}).$

Since by construction 
$$\mu\left(\bigcup_{ij} Z_{ij}\right)\geq 1-2\mu(\cK^c)\geq 1-8\mu(H)^{1/4}$$
we obtain \eqref{BigFamily} concluding the proof of the lemma.
\end{proof}

We apply the above lemma with $\bar{\xi}_n:=e^{-\eps^3n}$, $\bar{r}:=\bar{\xi}_n^{1+\omega}$ and 
$H=\cL_{n,\tau}^c$, to get  a family $\mathbb{B}_1=\{\bar{B}_s\}=\{B_s(\bar{\xi}_n,\bar{r})\}_{s\in J_r}$ as in the above lemma.   Since $\mu(\cL_{n,\tau})\geq 1-3\eps^\bb$, 
\eqref{BigFamily} implies that
$$
\mu(\bigcup_{s}\bar{B}_s)\geq 1-10(3\eps^\bb)^{1/4}\geq 1-30\eps^{\bb/4}.
$$
We apply Lemma \ref{lem:setsB} again with
\be \label{DefXiN}
\xi_n=e^{\eps^2 n-\eta_2\epsilon n},\;\; r\leq e^{-\eta_2\eps n}, \ee 
and 
\be \label{HPrime} H'=\Big(\cL_{n,\tau}\cap \bigcup_{s}f^{-n}((1-\eps')\cdot \bar{B}_s)\Big)^c, 
\ee
where $\eps'=\epsilon^{100\bb}$ to get a family  $\mathbb{B}_2(r)=\{B_i(\xi_n,r)\}_{i\in J_r}=\{B(\cW_i(\xi_n),r)\}$. Notice that $\mu(H')\leq 40\eps^{\bb/4}$ and so by the above lemma,
\be\label{eq:mesb2}
\mu\Big(\bigcup_{i\in J_r} B_i(\xi_n,r)\Big)\geq 1-10\cdot (40\eps^{\bb/4})^{1/4}\geq 1-100\eps^{\bb/16}.
\ee
Moreover, by \eqref{eq:largeL} it follows that for every $i\in J_r$, 
\be\label{eq:lint}
m^u_{W_i(\xi_n)}(\cL_{n,\tau}\cap \bigcup_{s}f^{-n}\Big((1-\eps')\bar{B}_s\Big)\geq 1-\mu(H')^{1/4}\geq 1-100\eps^{\bb/16}.
\ee

We will now construct the fake center stable partitions of elements in $\mathbb{B}_2(r)$ by pulling back the $\R^{cs}$ partition from element in $\mathbb{B}_1$ (with time $n$). Since $\xi_n=e^{\eps^2n-\eta_2\eps n }$ is now fixed throughout the paper we denote 
\be \label{HCW}
\hcW_i=\cW_i(\xi_n)
\ee
\be\label{BI-R}
B_i(r)=B_i(\xi_n,r)
\ee 
where $r\leq e^{-\eta_2\eps n}$. 
For any  $s$ consider the maximal connected components $\{R_{i,s,j}\}_{j}$ of the set $\hcW_i\cap f^{-n}\Big((1-\eps')\bar{B}_s\Big)$. Notice that by \eqref{eq:lint} the union of $R_{i,s,j}$ (over $s,j$) for which $R_{i,s,j}\cap \cL_{n,\tau}=\emptyset$ has measure $\leq 100\eps^{\bb/16}$ and so from now on we will restrict our attention to those $R_{i,s,j}$ for which $R_{i,s,j}\cap \cL_{n,\tau}\neq \emptyset$.  Let $z=z_{i,s,j}\in R_{i,s,j}\cap \cL_{n,\tau}$  
and let for $y\in R_{i,s,j}$
\be\label{eq:fak1}
\tilde W^{cs,n}_{i,s,j}(y)=(\tilde f_{z}^{(n)})^{-1}(\R^{cs}+\tilde f_{z}^{(n)}(h^{-1}_{z}y))\ee
where $\R^{cs}+y$ is the parallel to $\R^{cs}$ through $y$ and  let
\be\label{eq:fak2}
\cF_{i,s,j}(y):=W^{cs,n}_{i,s,j}(y)=h_{z}(\tilde W^{cs,n}_{i,s,j}(y)).\ee

The following  lemma summarizes the properties of the foliations $\{\cF_{i,s,j}\}$ that will be important 
for our purposes.

\begin{lemma}\label{lem:ff}
For any $\hat{\eps}>0$ there exists $n_{\hat{\eps}}$ such that for $n\geq n_{\hat{\eps}}$,  we have the following:
\begin{enumerate}
\item[{\bf F1.}] for any $i,s,j$ and any $k\leq n$, 
$$
f^k(R_{i,s,j})\subset h_{f^kz_{i,s,j}}(B^u( e^{\lambda(-n+k)} e^{-\eps^5 n}));
$$
\item[{\bf F2.}] for any $i,s,j$ and for any $z,z'\in R_{i,s,j}$, we have 
\be\label{eq:jacsm}
\left|\frac{\det (Df^{n}_z|E^u(z))}{\det( Df^{n}_{z'}|E^{u}(z'))}\right|\in (1-\hat{\epsilon},1+\hat{\epsilon});
\ee
\item[{\bf F3.}] for any $i$, $m^u_{\hcW_i}(\bigcup_{s,j}R_{i,s,j}\cap \cL_{n,\tau})\geq 1-100\eps^{\bb/16}$;
\item[{\bf F4.}] for any $i,s,j$, any $y' \in \cF_{i,s,j}(y)\cap B_i(r)$ and any $0\leq k \leq n$
$$
d(f^k y,f^k y')\leq e^{\epsilon^{10} n} \cdot r.
$$
\item[{\bf F5.}] for any $i$, if 
$$
B_i(r)\cap \cF_{i,s,j}(y)\cap \cF_{i,s',j'}(y')\neq \emptyset,
$$
for some $y\in R_{i,s,j}, y'\in R_{i,s',j'}$, then $s=s'$, $j=j'$ and $y=y'$.
\end{enumerate}
\end{lemma}
\begin{proof} We start by showing {\bf F1}. By definition, $f^{n}(R_{i,s,j})\subset (1-\hat{\eps})\bar{B}_s$ and it is a connected component  of $f^n(\hcW_i)\cap (1-\hat{\eps})\bar{B}_s$. Moreover $z_n:=f^n(z_{i,s,j})\in f^n(R_{i,s,j})\cap f^n(\L_{n,\tau})$. In particular, $z_n\in P_\tau\subset  \tP_{\tau^{-1}}$. Therefore, $f^n(R_{i,s,j})\subset W^u_{z_n,\tau}$ and in particular $f^n(R_{i,s,j})$ 
 its backward iterations 
satisfy \eqref{BackContrfK}.
 Since $f^n(R_{i,s,j})\subset (1-\hat{\eps})\bar{B}_s$ and $f^n(R_{i,s,j})\cap P_\tau\neq\emptyset$ by construction,
the 
H\"older continuity of $z\mapsto \eta_z$ on $P_\tau$ (see Lemma \ref{fakeu-foliation})
implies that that the unstable size of $f^n(R_{i,s,j})$ is at most 
 $$ 2 (1+\hat{\eps})\bar{\xi}_n=2 (1+\hat{\eps})e^{-\eps^3n}.$$  
 This gives {\bf F1.} by using $f^{k}(R_{i,s,j})=f^{k-n}(f^n(R_{i,s,j}))$  
and applying \eqref{BackContrfK}.
\\

For {\bf F2.} notice that 
$$
 |\log \det (Df^{n}_z|E^u(z))-\log \det (Df^{n}_{z'}|E^u(z'))|\leq 
$$
$$C_1 \sum_{k=0}^{n}\left[d(f^{k}z,f^{k}z')+d(E^u(f^{k}z), E^u(f^{k}z'))\right]
\leq C_2 n e^{-\eps^5 n},$$
 where in the last inequality we use 
 {\bf F1} to estimate the first term and Lemma \ref{admisibleconv} with $L_1=E^u(z), L_2=E^u(z')$
 to estimate the second term.\\

 {\bf F3} is immediate from \eqref{eq:lint}. \\

We now show {\bf F4.} In the coordinates $(a, b)$ introduced in \eqref{BXiR}, $\cF_{i,s,j}(y)$ is given by the graph
$a=A(b).$ Let $y=(A(b_0), b_0),$ $y'=(A(b_1), b_1)$. Consider a curve 
$\gamma(t)=\{h_z(A(b(t), b(t))\}$ where $b(t)=b_0(1-t)+b_1(t).$  
Recall that each component $R_{i, j, s}$ contains a point $z=z_{i,j,s}\in P_\tau.$
Hence combining Lemma \ref{cs-foliation} with already established property {\bf F1.} we conclude that\footnote
{Lemma \ref{cs-foliation} tells us that the fake center stable leaves have tangent spaces in $\cC^{cs}_z$,
while {\bf F1.} ensures that $\cF(y)$ belongs to some fake center-stable leaf.}
$\dot\gamma(t)\in \cC_z^s$ for all $t.$ 
Hence Lemma \ref{LmCones}(2) tells us that $f^k$ increases distances
by at most the factor $\tau^{-2} e^{2k\delta}$ proving {\bf F4.} (since $\delta<\eps^{100}$).
\\

It remains to show {\bf F5}. Assume first that $s=s'$. 
Then $y=y'$ since otherwise $\cF(y$) intersects $\hcW$
in two points, which is impossible by transversality. Namely
if two points $x_1, x_2$ belong to $\cF(y)$ then the segment (in $(a,b)$ coordinates
introduced in \eqref{BXiR}) joining $x_1$ and $x_2$ belongs to  $\cC^{cs}_y$. If the two points belong
$\hcW$ then the segment joining them belongs to $\cC^u.$ Since  $\cC^u_y\cap \cC^{cs}_y=\{0\}$,
$\cF(y)\cap \hcW$ is a single point.
Since $R_{i, s,j}\cap R_{i, s, j'}$
contains $y$ we must have $j=j'$ completing the proof in the case $s=s'.$


Assume now that $s\neq s'$. We claim that for any $s$,
\be\label{eq:subb}
B_i(r)\cap \cF_{i,s,j}(y)\subset f^{-n} (\bar{B}_s)
\ee
The above claim immediately gives {\bf F5.} as the sets $\{\bar{B}_s\}$ are pairwise disjoint. So it remains to show \eqref{eq:subb}. By definition we have that $f^{n}(R_{i,s,j})\subset (1-\hat{\eps})\bar{B}_s$. In particular, $d(f^ny,\partial \bar{B}_s)\geq \hat{\eps}\min(\bar{\xi}_n,\bar{r})= \hat{\eps}\cdot e^{-\eps^3n(1+\omega)}$. On the other hand by  {\bf F4} it follows that for every $y'\in B_i(r)\cap \cF_{i,s,j}(y)$, $d(f^ny,f^ny')\leq e^{\eps^{100 n}}r\leq e^{\eps^{100}n-\eta_2\eps n}$. Since for $n$ large enough, $e^{\eps^{100}n-\eta_2\eps n}\ll \hat{\eps}\cdot e^{-\eps^3n(1+\omega)}$, $f^ny'\in \bar{B}_s$. This finishes the proof.
\end{proof}

By enumerating (replacing the indices $(s,j)$ by $j$), 
we will denote the sets $\{R_{i,s,j}\}$ and the foliations $\cF_{i,s,j}$ simply by $\{R_{i,j}\}$ and $\{\cF_{i,j}\}.$
Define 
\be\label{eq:tilr}
\tilde{R}_{i,j}(r):=\bigcup_{y\in R_{i,j}}\cF_{i,j}(y)\cap B_i(r).
\ee

\section{Absolute continuity of fake center foliation.}
\label{ScAC}
 In this section we continue assuming that
$f$ is $C^{1+\alpha}$ diffeomorphism
of a compact manifold $M$ preserving a smooth measure $\mu$
with some non-zero exponents. 
The goal of the section is to establish 
an absolute continuity at exponential scale of the foliations $\{\cF_{i,j}\}$
 which plays a crucial role in the proof of the main theorem.
  Let $\pi_{i,j}$ denote the holonomy along the foliation $\cF_{i,j}$. Recall that $\hcW_i$ is the reference unstable manifold in $B_i(r)$ and $R_{i,j}\subset \hcW_i$.

\begin{definition}\label{def:compl}
We say that an unstable manifold  $\cW$ of a point $z\in \cL_{n,\tau}$, {\em crosses 
$\tilde R_{i,j}$ completely} if
$\DS \;\cW\cap \pi_{i,j}(y)\neq \emptyset \text{ for  every }y\in R_{i,j}.$
$\cW$ {\em crosses
$B_i(r)$ completely} if  for every $j$ it crosses $\tilde R_{i,j}$ completely.
\end{definition}

\begin{definition}\label{apadm}
	 Given $x$ a Lyapunov regular point, we say that $\cL$ is an $(\beta,A)$-admissible manifold if 
	\begin{enumerate}
		\item $\cL\subset Q_x$;
		\item $h_x^{-1}\cL \subset \graph(\eta_{\cL})$, where $\eta_{\cL}:\R^u\to\R^{cs}$ satisfies 
		$\DS 
		 \|\eta_{\cL}\|_{C^{1+\beta}}\leq A.
		 $
			\end{enumerate}
\end{definition}

\begin{remark}
	Note that $(\beta,A)$-admissibility depends on the reference point $x$. However, 
	as long as  the reference point is in $P_\tau$ we have a uniform control on the parameter $A$ when we change of charts. That is, there is a constant $K(\tau)$ so that if $x,y\in P_\tau$, $y\in Q_x$ and $\cL$ is $(\beta,A)$-admissible for $x$ then it is $(\beta, K(\tau)\cdot A)-$admissible for $y$.
\end{remark}

 Let $\beta>0$ be a number that is small enough, one can take 
 $\DS \beta:=\frac{1}{2}\Big[\min\Big(\alpha, \min_{i\leq 7}\alpha_i\Big)\Big]$, see Lemmas \ref{lyapchart},\ref{cs-foliation},\ref{lem:unstman},\ref{LmIDHold}, \ref{cones},\ref{admisibleconv} for the definitions of the $\alpha_i$.

\begin{proposition}\label{lem:abscont} 
	For every 
	$\hat{\epsilon}>0$ there exists $n_{\hat\epsilon}$ such that for every $n\geq n_{\hat \epsilon}$, every $r\leq e^{-\eta_2\epsilon n}$ the following holds: let $x\in \cL_{n,\tau}$ and let $\cL$ be an 
	$(\beta,1)$-admissible manifold.
	Assume further that
	\begin{enumerate}
		\item $\cL$ crosses $\tilde R_{ij}(r)$ completely;
		\item there are points $x\in\hat\cW_i$ and $x'\in\cL$ such that if $\bar x'=h_x^{-1}(x')$, 
		then $|\bar x'|\leq e^{-\eta_2\epsilon n}$ and $\|D_0\eta_{\cL}\|\leq e^{-\beta\eta_2 n}$.  
	\end{enumerate} 
	Then the jacobian $J(\pi_{i,j})$ satisfies
	\begin{equation}
		\label{EqJac1}
		J(\pi_{i,j})(z)\in (1-\hat\epsilon,1+\hat\epsilon)\;\mbox{for each}\; z\in R_{i,j}.
	\end{equation}
	In particular,
	\begin{equation}
	\label{EqArea1}
		m^u_{\cW}(\pi_{i,j}(R_{i,j}))\in (1-\hat\epsilon,1+\hat\epsilon)m^u_{\hat\cW_i}(R_{i,j}).
	\end{equation}
\end{proposition}

\begin{proof}
We will drop the index $i$ in the proof. Let $\bar\cL=h_x^{-1}(\cL)$, $\overline{\hat\cW}=h_x^{-1}(\hat\cW)$. 
We let $\bar\pi_j:\overline{\hat\cW}\to \bar\cL$ be $\bar\pi_j=h_x^{-1}\circ\pi_j\circ h_x$. 
For $\bar z\in \overline{\hat\cW}$ let $\bar z'=\bar\pi_{j}(\bar z)\in \bar\cL$.
Set $\bar\pi^{(n)}_j:\tilde f_x^{(n)}\left(\overline{\hat\cW}\right)\to \tilde f_x^{(n)}(\bar\cL)$, the holonomy 
 along
the $\R^{cs}$ spaces.  Then $\bar\pi_{j}=(\tilde f_x^{(n)})^{-1}\circ\bar\pi^{(n)}_j\circ \tilde f_x^{(n)}$. 
 For a linear map $T:E\to F$ where
$ E, F\subset \R^D$, $\dim E=\dim F$, we denote $\bar J(T)=|\det(T)|$ where determinant is taken with respect to natural volume coming from euclidean distance in $\R^D$. We denote as well $\bar J(\bar\pi_j)(\bar z)=\bar J(D_{\bar z}\bar\pi_j)$. 

So we have that $$\bar J(\bar\pi_j)(\bar z)=\frac{\bar J(D_{\bar z}\tilde f_x^{(n)}|{T_{\bar z}\overline{\hat\cW}}) }{\bar J(D_{\bar z'}\tilde f_x^{(n)}|{T_{\bar z'}{\bar\cL}}) }\bar J(\bar\pi_j^{(n)})(f^{(n)}_x(\bar z)).$$ 

Thus the proposition is a consequence of the following two estimates:

\begin{equation}\label{jac}
	\forall z,z'\in \tilde R_{i,j} \;\;\frac{\bar J(D_{\bar z}\tilde f_x^{(n)}|{T_{\bar z}\overline{\hat\cW}}) }{\bar J(D_{\bar z'}\tilde f_x^{(n)}|{T_{\bar z'}{\bar\cL}}) }\in (1-\hat\epsilon^4,1+\hat\epsilon^4);
\end{equation}

\begin{equation}\label{jaccons}
	\forall z^*\in f^n(R_{i,j})\;\;\;\bar J(\bar\pi_j^{(n)})(\bar z^*)\in (1-\hat\epsilon^4,1+\hat\epsilon^4)
\end{equation}

We have 
$$\frac{\bar J(D_{\bar z}\tilde f_x^{(n)}|{T_{\bar z}\overline{\hat\cW}}) }{\bar J(D_{\bar z'}\tilde f_x^{(n)}|{T_{\bar z'}{\bar\cL}}) }=\prod_{k=0}^{n-1}\frac{\bar J\left(D_{\tilde f_x^{(k)}(\bar z)}\tilde f_{f^kx}|T_{\tilde f_x^{(k)}(\bar z)}\tilde f_x^{(k)}(\overline{\hat\cW})\right)}{\bar J\left(D_{\tilde f_x^{(k)}(\bar z')}\tilde f_{f^kx}|T_{\tilde f_x^{(k)}(\bar z')}\tilde f_x^{(k)}(\bar\cL)\right)}.$$

Taking logarithms we have to estimate: $$\sum_{k=0}^{n-1}\left|\log\bar J\left(D_{\tilde f_x^{(k)}(\bar z)}\tilde f_{f^kx}|T_{\tilde f_x^{(k)}(\bar z)}\tilde f_x^{(k)}(\overline{\hat\cW})\right)-\log\bar J\left(D_{\tilde f_x^{(k)}(\bar z')}\tilde f_{f^kx}|T_{\tilde f_x^{(k)}(\bar z')}\tilde f_x^{(k)}(\bar\cL)\right)\right|.$$

By Lemma \ref{lyapchart} we have that that for every $x\in LyapReg$,
$Hol_{\alpha_2}(D\tilde f_x)\leq\delta$ and by Lemma \ref{lem:boder}, $ \|D\tilde f_x\|,\| D\tilde{f}^{-1}_x\|\leq K'$.  We can hence bound  the above sum using Lemma~\ref{jacgras} by
  $$N(K') \sum_{k=0}^{n-1}\delta|\tilde f_x^{(k)}(\bar z)-\tilde f_x^{(k)}(\bar z')|^{\alpha_2}+d\left(T_{\tilde f_x^{(k)}(\bar z)}\tilde f_x^{(k)}(\overline{\hat\cW}), T_{\tilde f_x^{(k)}(\bar z')}\tilde f_x^{(k)}(\bar\cL)\right)
 $$

By assumption (2) in the proposition, the fact that $\cL$ is $(\beta,1)$ admissible and 
Lemma~\ref{fakeu-foliation}, we get 
 using the transversality of $\cF_{ij}$ to $\cL$ and $\hcW$
that 
$$|\bar z-\bar z'|\leq Ce^{-\eta_2\epsilon n}\quad{and}\quad d\left(T_{\bar z}\overline{\hat\cW}, T_{\bar z'}\bar\cL\right)\leq e^{-\beta\eta_2 n}.$$ 
Since $\bar z'\in W^{cs,n}_x(\bar z)$, we get by Corollary \ref{expcs} that $|\tilde f_x^{{(k)}}(\bar z)-\tilde f_x^{{(k)}}(\bar z')|\leq e^{3k\delta}|\bar z-\bar z'|$ for every $k\in [0,n]$ and hence $|\tilde f_x^{{(k)}}(\bar z)-\tilde f_x^{{(k)}}(\bar z')|\leq Ce^{-\eta_2\epsilon n}e^{3k\delta}$. 
 We now apply Lemma~\ref{admisibleconv}  with
$L_1= T\cL$ and $L_2= T\hat{\cW}.$
 Using Lemma \ref{fakeu-foliation}, 
the  admissibility 
of $\cL$ and  the assumption $\|D_0\eta_\cL\|\leq e^{-\beta \eta_2 n}$ 
 we conclude from \eqref{eq:adml} that

\begin{eqnarray*}
	d\left(T_{\tilde f_x^{(k)}(\bar z)}\tilde f_x^{(k)}(\overline{\hat\cW}),
	 T_{\tilde f_x^{(k)}(\bar z')}\tilde f_x^{(k)}(\bar\cL)\right)&\leq& e^{k(-\lambda+\sqrt{\delta})}e^{-\beta \eta_2 n}
	+6\delta Ce^{-\alpha_7\eta_2\epsilon n}e^{3k\alpha_7\delta}
	\end{eqnarray*}
 proving \eqref{jac}  if $n$ is sufficiently large in terms of $C,\delta$ and $\hat{\eps}$.

To prove
\eqref{jaccons} we notice that $\bar\pi_j^{(n)}$ is just holonomy along planes parallel to $\R^{cs}$ so there is a constant $C_3$ only depending on dimension so that 
$$
|\bar{J}(\bar\pi_j^{(n)})(\bar z^*)|\leq C_3d\left(T_{\bar z^*}\tilde f_x^{(n)}(\overline{\hat\cW}),
T_{\bar\pi^{(n)}_j(\bar z^*)}\tilde f_x^{(n)}(\hat\cL)\right)\leq C_3\left( e^{n(-\lambda+4\delta-\beta\eta_2)}
+6\delta Ce^{n(\alpha_7\delta-\alpha_7\eta_2\epsilon)}\right).
$$
Since $\delta=\eps^{100}$, we get the result.
\end{proof}

\begin{corollary}\label{cor:uns}
If for some $x\in \cL_{n,\tau}$ and $r\leq e^{-\eta_2\epsilon n}$, $\cW=\cW_{x, \xi}$ crosses $\tilde{R}_{i,j}(r)$ completely and $d(\cW, \hcW_i)\leq e^{-\eta_2 \eps n}$
then \eqref{EqJac1} and \eqref{EqArea1}  are satisfied.
\end{corollary}

\begin{proof}
 We need to verify the conditions of Proposition \ref{lem:abscont} with $\cL=\cW$. Note that (1) and (2) follow from the assumptions of the corollary and the fact that $\|D_0\eta_\cL\|<e^{-\alpha_1 \eta_2 n}$ follows from  Lemma~\ref{LmIDHold}. Moreover,  $\cL=\cW$ satisfies the assumptions of Definition  \ref{apadm}
 by Lemma~\ref{fakeu-foliation}
 (since $\beta<\alpha_4$).
\end{proof}

Proposition \ref{lem:abscont}  
allows to establish the local product structure on the sets $\tilde{R}_{i,j}(r)$. 
Let $\brho$ be the function defined in \eqref{MesMultiDisc}.

\begin{corollary}\label{cor:prost'}
For every $\hat{\epsilon},\tau>0$ there exists $n_{\hat{\epsilon},\tau}\in \N$ such that for every $n\geq n_{\hat{\epsilon},\tau}$ every $r\leq e^{-\eta_2 \epsilon n}$, we have 
$$
\mu(\tilde{R}_{i,j}(r)))\in (1-\hat{\epsilon},1+\hat{\epsilon}) \brho(x)  m^u(R_{i,j})\cdot \mes\left(B^{cs} (0, r)\right)
$$
where $x$ is an arbitrary point in $\tilde R_{i,j}(r) \cap P_\tau.$ 
\end{corollary}
\begin{proof}
Recall that 
$$\tilde{R}_{i,j}(r)=\bigcup_{y\in R_{i,j}}\left(\cF_{i,j}(y)\cap B_i(r)\right).$$ 
In particular, we have coordinates $(a,b)$ on $B_i(r)$ given by \eqref{BXiR}. 
Let 
$$ \fW^\fb=\{x\in B_i(r): b(x)=\eta_z(a(x))+\fb\}.$$ 
 Using uniform smoothness of $\hcW_i$ on $P_\tau$,
the fact that $T\hcW_i(0)=\R^u\times \{0\}$ and
 the H\"older continuity of  $\brho(\cdot)$ 
on  $P_\tau$,
we conclude that the ratio of $\mu|_{B_i}$ to the the measure 
$\hmu$ given in our coordinates by $d\hmu=\brho(x) 
da db$ is between $1-\hat{\eps}^{2}$ and $1+\hat{\eps}^{2}.$
Therefore 
$$ \mu(\tilde{R}_{i,j}(r)))\in  \brho(x)  \left(1-\hat{\eps}^{2}, 1+\hat{\eps}^{2}\right) 
\int_{B_r^{cs}(0)} \mes(\fR^\fb) d\fb $$
where $\fR^\fb$ is the image of $R_{i,j}$ under the $\cF_{i,j}$--holonomy inside
$\fW^\fb.$ Since
$R_{i,j}\cap \cL_{n,\tau}\neq\emptyset$ (because we only consider the rectangles satisfying this condition),
Proposition \ref{lem:abscont}  
applied to $\cL=\fW^b$ (the assumptions of the proposition are verified 
since $\fW^b$ are just translations of
$\hcW$ in the appropriate coordiantes)
tells us that  $$\mes(\fR^\fb)\in \left(1-\hat{\eps}^{2}, 1+\hat{\eps}^{2}\right) m^u(R_{i,j})$$ for $n$ large enough
and the result follows.
\end{proof}

The above corollary implies also the following:
\begin{lemma}\label{lem:new1} 
There exists $n_{\epsilon}\in \N$ such that for every $n\geq n_{\epsilon}$, for every $i$ we have for $r\leq e^{-\eta_2\epsilon n}$
\be\label{eq:rij}
\mu\left(\bigcup_{j}\tilde{R}_{i,j}(r)\right)\geq (1- 200\epsilon^{\bb/16})\mu(B_i(r)).
\ee
In particular,
$$
\mu\left(\bigcup_{i,j}\tilde{R}_{i,j}(r)\right)\geq 1- 300\epsilon^{\bb/16}.
$$
\end{lemma}
\begin{proof} The first part is an immediate consequence of  Corollary \ref{cor:prost'} and \eqref{MesMultiDisc}
(with $\hat{\eps}=\eps^{100b}$) by summing over $j$ and using {\bf F3}.
The second part follows from the first and \eqref{eq:mesb2}.
\end{proof}

Recall that for each $i$ we have the reference manifold $\hcW_i\subset B_i(r)$ of size 
$\DS  \xi_n\!=\!e^{\epsilon^2 n-\eta_2\epsilon n} .$

\begin{proposition} \label{prop:3.2}
 For every $\hat{\epsilon}>0$ there exists $n_{\hat{\epsilon},\epsilon}\in \N$ such that for every $n\geq n_{\hat{\epsilon},\epsilon}$ the following holds. 
 Let $z\in \cL_{n,\tau}$. Suppose that $W^u_{z,\tau}$ crosses $B_i(r)$ completely and let 
$\cW_z\subset B_i(r)$, $\cW_z\subset W^u_{z,\tau}$ be a piece of the unstable manifold containing $z$ of size $\xi\geq (1-\hat{\epsilon}^2)\xi_n$.
Let $\pi_{i,z}: \hcW_i\to \cW_z$ be given by $\pi_z(y):=\pi_{i,j}(y)$ if $y\in R_{i,j}$. Then  $\pi_z$ is $\hat{\epsilon}$--measure preserving. In particular,
$$
m^u_{\cW_{z}}\Big(\pi_z(\bigcup_{j} R_{i,j})\Big)\in (1-\hat{\epsilon},1+\hat{\epsilon}) m^u_{\hcW_i}(\bigcup_{j}R_{i,j})
$$
 \end{proposition}
\begin{proof} Note that by Corollary \ref{cor:uns}
we have (since $\xi\geq (1-\hat{\eps}^2)\xi_n$)

$\DS
m^u_{\cW_z}\Big(\pi_z(\bigcup_{j} R_{i,j})\Big)=\sum_{j}m^u_{\cW_z}(\pi_{i,j}(R_{i,j}))\in (1-\hat{\eps}, 1+\hat{\eps}) \sum_{j}m^u_{\hcW}(R_{i,j}).
$
\end{proof}

\section{Equidistribution of unstable leaves.}
\label{sec:main} 
\subsection{Equidistribution criterion.}
\label{SSEquiCriterion}
Let $O_z(r)$ denote a ball of radius $r>0$ centered at $z\in M$. 
 Denote \be \label{DefRn} \tilde{r}_n=e^{-\eta_2\epsilon n-\eps^2n}. \ee
 Recall Definition \ref{def:compl} and \eqref{BI-R}.

\begin{proposition}\label{prop:3.5} For every $\epsilon>0$ there exists $n_{\epsilon}\in \N$ such that for every $n\geq n_{\epsilon}$ the following holds. Fix any
$$B\in \Big\{B_i(e^{-\eta_2 \epsilon n})\Big\}\cup \Big\{(1-\epsilon^{4000})B_i(e^{-\eta_2\epsilon n})\Big\}\cup\Big\{O_z(r)\Big\}_{z\in M},$$
with $1>r\geq e^{-\eta_2\epsilon n}$ and let   $B_s\in \{B_i(\tilde{r}_n)\}$. 
Assume that  $\cW=\cW^u_{x,\xi}$, with $e^{-\epsilon^3 n}\leq \xi\leq \tau$   and $x\in \cL_{n,\tau}$ 
crosses $B_s$ completely. Then 
$$
m^u_\cW\left(\left(\bigcup_{j}\tilde{R}_{s,j}(\tilde{r}_n)\right)\cap f^{-\epsilon n}(B)\right)\leq(1+\epsilon^{20})\mu(B)m^u_{\cW}\left(\bigcup_{j}\tilde{R}_{s,j}(\tilde{r}_n)\right).
$$

If moreover, for some $\zeta>0$
\be\label{eq:drp3}
\mu\left(\left(\bigcup_{j}\tilde{R}_{s,j}(\tilde{r}_n)\right)\cap  f^{-\epsilon n}((1- \zeta)B))\right)\geq 
(1- 2\epsilon^{40})\mu(B_s)\mu((1-\zeta)B),
\ee
then we have
 $$
m^u_\cW\left(\left(\bigcup_{j}\tilde{R}_{s,j}(\tilde{r}_n)\right)\cap f^{-\epsilon n}(B)\right)\geq 
(1-\epsilon^{20})\mu((1-\zeta)B)m^u_{\cW}\left(\bigcup_{j}\tilde{R}_{s,j}(\tilde{r}_n)\right).
$$
\end{proposition}
\begin{proof} Denote  $\tilde{B}=B_s$ and let $C_j=\tilde{R}_{s,j}(\tilde{r}_n)\cap \cW$, 
$\tilde{C}_j=\tilde{R}_{s,j}(\tilde{r}_n)$, $ \DS \cC=\bigcup_j \tC_j.$

We will first prove the upper bound. Call an index $j$ {\em nice} if $C_j\cap  f^{-\epsilon n}(B)\neq \emptyset$. Let 
$\DS A^+=\bigcup_{j \text{ nice }}C_j$. Since
$$
m^u_\cW\left(\cC\cap f^{-\epsilon n} B\right)\leq m^u_\cW(A^+)
$$
It is enough to show that
\be\label{eq:cla3'} 
m^u(A^+) \leq (1+\eps^{1000})m^u\left(\cW\cap \cC \right)\mu(B).
\ee
 Let  $\kappa=\kappa(\epsilon)$ be small enough  so that, we have 
\be \label{MesSq} 
\mu((1+\kappa)\cdot B)\leq (1+\eps^{2000})\mu(B)\;\;\text{ and }\;\;
\mu((1-\kappa)\cdot  B)\geq (1-\eps^{2000})\mu(B).
\ee
We claim that 
\be\label{eq:cla2'}
 \cC \cap f^{-\epsilon n}((1+\kappa)B)\supset \bigcup_{j \text{ nice }}\tilde{C}_{j}.
\ee
Indeed, let $z\in\tilde{C}_{j}$ with $j$ nice  and let $z'=\pi_{s,j}(z)\in C_j$ and let $z''\in C_j\cap f^{-\eps n}(B)$. 
Then  by {\bf F4} and {\bf F1} in Lemma \ref{lem:ff} it follows that 
$$
d(f^{\epsilon n}z'',f^{\epsilon n}z)< d(f^{\epsilon n}z'',f^{\epsilon n}z')+d(f^{\epsilon n}z',f^{\epsilon n}z)\leq  2 
e^{\epsilon^3 n}\cdot e^{-\eta_2\epsilon n-\eps^2n}.
$$
By definition  $f^{\epsilon n}z''\in B$. We claim that $f^{\epsilon n}z\in (1+\eps^{1000})B$ if $n$ is large enough. Indeed, if $B=O_z(r)$ with $r\geq e^{-\eta_2\eps n}$ is a ball, then the claim follows by triangle inequality as $2e^{-\eta_2\epsilon n-\eps^2n +\eps^3n}+r\leq (1+\kappa)r$. If 
$B=B_i(e^{-\eta_2 \eps n})$ (or $B=(1-\eps^{4000})\cdot B_i(e^{-\eta_2\eps n})),$ 
then the claim again follows by triangle inequality, the definition of $B_i(\cdot, \cdot)$ (see \eqref{BXiR})  and by  again using that $2 
e^{\epsilon^3 n}\cdot e^{-\eta_2\epsilon n-\eps^2n}$ is much less than  $\kappa e^{-\eta_2 \eps n}$ for $n$ large enough.

 Since $z\in \tilde{C}_j$ is arbitrary, \eqref{eq:cla2'} follows. 
 
By exponential mixing and  \eqref{MesSq}, \eqref{eq:cla2'}
\begin{equation}
\label{510-1}
\mu\left(\bigcup_{j\text{ nice }}\tilde{C}_{j}\right)\leq 
\mu\left(\cC \cap  f^{-\epsilon n}((1+\kappa)B)\right)\leq
\end{equation}
$$ \mu(\tilde{B}\cap f^{-\epsilon n}((1+\kappa)B)) \leq (1+\eps^{2000})\mu(\tilde{B})\mu((1+\kappa)B)\leq 
(1+3\eps^{2000})\mu(\tilde{B})\mu(B).
$$
By Corollary \ref{cor:prost'}  it follows  that if $\brho_{\tilde{B}}=\brho(x)$, for an arbitrary $x\in \tilde{B}\cap P_\tau$, then 
\begin{equation}
\label{510-2}
 (1-\epsilon^{2000})\brho_{\tilde{B}}\cdot m^u(A^+)\mes\Big(B^{cs}(0,e^{-\eta_2\epsilon n-\eps^2n})\Big) \leq 
 \mu\left(\bigcup_{j \text{ nice}}\tilde{C}_{j}\right).
\end{equation}
Also by  Lemma \ref{lem:new1} and Corollary \ref{cor:prost'}
\begin{equation}
\label{510-3}
\mu(\tilde{B})\leq (1+\eps^{2000})
\mu(\cC)
\leq (1+4\eps^{2000})\brho_{\tilde{B}}\cdot 
m^u\left(\cW\cap\cC \right)\cdot 
\mes\Big(B^{cs}(0,e^{-\eta_2\epsilon n-\eps^2n})\Big).
\end{equation}
Combining  \eqref{510-1}--\eqref{510-3} gives \eqref{eq:cla3'} finishing the proof of the upper bound. 
\\

We will now show the lower bound assuming additionally that \eqref{eq:drp3} holds.

An index $j$ is called {\em good} if $C_j\subset f^{-\epsilon n}(B)$. Let $\DS A^-=\bigcup_{j \text{ good }}C_j$. Then by definition, 
$$
m^u_\cW\left(f^{-\epsilon n}B\cap \cC \right)\geq m^u_\cW(A^-).
$$
It is therefore enough to show that 
\be\label{eq:cla3} 
m^u(A^-) \geq (1-\eps^{20})m^u\left(\cW\cap \cC \right)\mu((1-\zeta)B).
\ee

We claim that 
\be\label{eq:cla2}
\cC \cap f^{-\epsilon n}((1-\zeta)B))\subset \bigcup_{j \text{ good }}\tilde{C}_{j}
\ee
The proof of the above claim is analogous to the proof of \eqref{eq:cla2'}. By \eqref{eq:cla2} and
 \eqref{eq:drp3},
$$
\mu\Big(\bigcup_{j\text{ good }}\tilde{C}_{j}\Big)\geq 
\mu\left(\cC \cap f^{-\epsilon n}((1-\zeta)B\right)\geq (1-2{\eps^{40}})\mu(\tilde{B})\mu((1-\zeta)B)\geq 
$$
\be\label{510-4}
(1-3\epsilon^{40})\mu(\tilde{B})\mu((1-\zeta)B).
\ee
 By Corollary \ref{cor:prost'}  it follows that 
\be\label{510-5}
 (1+\epsilon^{2000})\brho_{\tilde{B}}\cdot m^u(A^-) \mes\Big(B^{cs}(0,e^{-\eta_2\epsilon n-\eps^2n})\Big) \geq 
 \mu\left(\bigcup_{j\;\; good} \tC_j\right)
\ee
Moreover by Lemma \ref{lem:new1} and Corollary \ref{cor:prost'}, 
\be\label{510-6}
\mu(\tilde{B})\geq (1-\epsilon^{2000})\mu(\cC)\geq (1-\epsilon^{2000})\brho_{\tilde{B}}\cdot 
m^u(\cW\cap \cC)\mes\Big(B^{cs}(0,e^{-\eta_2\epsilon n-\eps^2n})\Big) 
\ee
Combining  \eqref{510-4}--\eqref{510-6} we get
$$
m^u_{\cW}(A^-)\geq (1-4\eps^{40})m^u_{\cW}(\cC)\mu((1-\zeta)B)
$$
finishing the proof of the lower bound.
\end{proof}


For $r>0$ let $\{B_{i}(r)\}_{i\in J_r}=\{B_i\left(r\right)\}_{i\in J_r}$ be the family of parallelograms constructed in Lemma \ref{lem:setsB}. The next result, proven in \S \ref{sub:hf}
will help us to verify the fully crossing assumption 
 in Proposition \ref{prop:3.5}.

\begin{lemma}\label{lem:compl}
For every $\hat{\epsilon}>0$ there exists $n_{\hat{\epsilon},\epsilon}\in \N$ such that for every 
$n\geq n_{\hat{\epsilon},\epsilon}$  any  $\cW\subset W^u_{\brz,\tau}$, for some $\brz\in \cL_{n,\tau}$  any $i\in J_r$, 
for $r=e^{-\eta_2\eps n}$ 
if $ \partial{\cW}\bigcap B_i(r)=\emptyset$ 
and $\cW\cap (1-\hat{\eps})\cdot B_i(r)\neq \emptyset$, then $\cW$ crosses $B_i(r)$ completely.
\end{lemma}


\subsection{The Main Proposition.}
Recall that $\tilde{r}_n=e^{-\eps^2 n-\eps \eta_2 n}$ and let 
\be
\label{WholeLam}
\mathcal{R}_n:=\bigcup_{i,j}\tilde{R}_{i,j}( \tilde{r}_n)
\ee
be the lamination defined by \eqref{eq:tilr}.


Let $\{B_i\left(e^{-\eta_2 \epsilon n}\right)\}_{i\in J}$
be the family given by Lemma~\ref{lem:setsB}.

\begin{proposition}[{\bf Main Proposition}]\label{prop:expeq} 
Assume that $f$ is exponentially mixing. For every $\epsilon>0$ 
 there exists $\xi_0<\tau$ and such that for each $\xi<\xi_0$ there is 
 $\bar{n}=n(\epsilon,\xi)\in \N$, such that for every $n\geq \bar{n}$,
there exists a set $K_{n}\subset \cL_{n,\tau}$, $\mu(K_{n})\geq 1-\epsilon^{10}$ satisfying the following: let $\{B_i\}_{i\in J}$ be either $\{B_i(e^{-\eta_2\epsilon n})\}_{i\in J_{e^{-\eta_2\epsilon n}}}$ OR a family of disjoint balls $\{O_{i}(r)\}_{i\in P}$ with $r\in [e^{-\eta_2\epsilon n},1)$ and  $\mu(\bigcup_{i\in P}O_i(r))\geq 1-\epsilon^{10^5}$. Then for every $x\in K_{n},$
 every unstable box $\cW$ of size $\xi$ containing $x$ 
there is a subset $J'(x)\subset J$ such that $\mu(\bigcup_{i\in J'} B_i)>1-\epsilon^{10}$ and we have 
\be\label{eq:LLn}
m^u_\cW(f^{-\eps n}(\cL_{n,\tau}))\geq 1-\epsilon^{10},
\ee
and for every $i\in J'(x)$,
  \be
\label{ExpEqui}
m^u_\cW\left(\mathcal{R}_n\cap f^{-\epsilon n}(B_i)\right)\in (1-\epsilon^{10},1+\epsilon^{10}) \mu(B_i),
\ee
and for $\tilde{\eps}=\eps^{4000}$,
  \be
\label{ExpEqui2}
m^u_\cW\left(\mathcal{R}_n\cap f^{-\epsilon n}((1-\tilde{\epsilon})B_i)\right)\in (1-\epsilon^{10},1+\epsilon^{10}) \mu((1-\tilde{\epsilon})B_i).
\ee

\end{proposition}

\subsection{Proof of Main Proposition.}
Fix $\epsilon,\xi>0$. Notice that it is enough to prove \eqref{ExpEqui} for $\{B_i\}$ and then repeat the proof for $\{(1-\tilde{\eps})B_i\}$ and take intersections of the corresponding sets $K_n$. 
Let
 $ \hat{\epsilon}=\epsilon^{10^4}$ and let $\bar{n}\geq n_{\hat{\eps},\eps}$, where $n_{\hat{\eps},\eps}$ is the maximum of the corresponding $n_{\hat{\eps},\eps}$ (or $n_{\eps}$) coming from  Proposition \ref{prop:3.5}, Lemma \ref{lem:new1}, Corollary \ref{cor:prost'}, Proposition \ref{lem:abscont},  Lemma \ref{lem:expmixB}, Lemma \ref{lem:setsB} and Lemma \ref{lem:compl} are satisfied.

 Let   $\bar{\epsilon}=\bar{\epsilon}(\hat{\epsilon})<\epsilon^{1000}$ 
 be small enough  so that for each $i\in J$, we have 
\be
\label{eq:absc1'}
\mu((1+\bar{\epsilon})\cdot B_i)\leq (1+\hat{\epsilon}^2)\mu(B_i)\;\;\text{ and }\;\;\mu((1-\bar{\epsilon})\cdot B_i)\geq (1-\hat{\epsilon}^2)\mu(B_i).
\ee

We will now define the set $K_n$. Let $\{B_i(\tilde{r}_n)\}_{i\in J_{\tilde{r}_n}}$ be the family $\mathbb{B}_2(r)$ for $r=\tilde{r}_n$. To make the notation simpler
we will denote the family $\{B_i(\tilde{r}_n)\}_{i\in J_{\tilde{r}_n}}$ by
\footnote{
The family $\{\tB_s\}_{s\in \brJ}$ should not be confused with 
$\{B_i(e^{-\eta_2\eps n})\}_{i\in J_{e^{-\eta_2\eps n}}}$ 
which is one of the possible targets in Proposition \ref{prop:expeq}.} 
$\{\tilde{B}_s\}_{s\in \bar{J}}$.

Let $\mathcal{B}_n(\hat{\eps}):=\Big(\bigcup_{s\in \bar{J}}(1-\hat{\eps})\cdot \tilde{B}_s\Big)$ and 
$$
K^\flat:=\left\{x\in M\;:\; m^u_{W^u_{x,2\xi}}
\left(\mathcal{B}_n(\hat{\eps}) \cap f^{-\epsilon n}(\cL_{n,\tau})\cap \cL_{n,\tau}\right)
\geq 1-\epsilon^{400}\right\}.
$$
Define
$$
K_n:=K^\flat\cap \cL_{n,\tau}.
$$
Below we summarize the properties of $K_n$ needed in the proof. 
Let $x\in K_n$ and let $\cW$ be an unstable box of size $\xi$ containing $x$, then
\begin{enumerate}
\item[{\bf A}.] $\mu(K_n)>1-\epsilon^{10}$;
\item[{\bf B}.] 
$m^u_{\cW}\Big(f^{-\epsilon n}(\cL_{n,\tau})\cap \cL_{n,\tau}\cap \mathcal{B}_n(\hat{\epsilon})\Big)
\geq 1-\epsilon^{300}$;

\item[{\bf C}.] There exists a set $J'\subset J$ such that
$\DS
\mu\left(\bigcup_{i\in J\setminus J'}B_i\right)<\epsilon^{10}
$
and for every $i\in J'$
\be\label{eq:newts}
m^u_\cW\left(\mathcal{R}_n\cap f^{-\epsilon n}(B_i)\right)\in (1-\epsilon^{10},1+\epsilon^{10}) \mu(B_i).
\ee

\end{enumerate}
Notice that {\bf A}-- {\bf C} immediately imply the Main Proposition.
Therefore, it remains to prove properties {\bf A}-- {\bf C}.\\

{\it Proof of {\bf A}:} Notice that {\bf A} follows by showing that  
$\mu(K^\flat)\geq 1-\epsilon^{20}$.  
But that estimate follows by applying Lemma \ref{LmLocMeas} with $\fB=\cL_{n,\tau}^c\cup f^{-\eps n} \cL_{n,\tau}^c \cup \cB_n(\heps)^c $ and 
 $\eps^{400}$  instead of $\eps$ (note that $\mu(\fB)\leq \eps^{4\times 400}$ by the definition of $\heps$ and \eqref{eq:mesb2}).
\\

{\it Proof of {\bf B}:}
Since $\cW\subset W^u_{x,2\xi}$ and  $x\in K_1$
$$m^u\left(\cW\setminus (f^{-\epsilon n}\cL_{n,\tau}\cap \cL_{n,\tau})\cap \mathcal{B}_n(\hat{\epsilon})\right)\leq $$$$
m^u\left(W^u_{x, 2\xi}\setminus (f^{-\epsilon n}L_{n,\tau}\cap \cL_{n,\tau})\cap \mathcal{B}_n(\hat{\epsilon})\right)
\leq \eps^{400} m^u(W^u_{x, 2\xi})\leq \eps^{300} m^u(\cW). $$
This finishes the proof of {\bf B}. \\

{\it Proof of {\bf C}:} By {\bf B} it follows that $m^u_{\cW}(\mathcal{B}_n(\hat{\eps}))\geq 1-\eps^{300}$. Let $S\subset \bar{J}$ be such that for $s\in S$, $\cW\cap (1-\eps^{40})\cdot\tilde{B}_s \neq \emptyset$ and $\partial \cW\cap \tilde{B}_s=\emptyset$. Note that by Lemma \ref{lem:compl}, $\cW$ crosses $\tilde{B}_s$ completely for every $s\in S$. Denote $\tilde{B}_\cW=\bigcup_{s\in S} \tilde{B}_s$.

 We claim that 
\be\label{eq:bme}
m^u_\cW(\tilde{B}_\cW)\geq 1-2\epsilon^{300}.
\ee
Indeed,  if for some $s'\in \bar{J}$, $\partial\cW\cap \tilde{B}_{s'}\neq \emptyset$, then $\tilde{B}_{s'}\subset 
V_{n^{-2}}(\partial \cW)$
where $V_{n^{-2}}(\partial \cW)$ denotes
$n^{-2}$ neighborhood of $\partial \cW$. Therefore using {\bf B} we get
$$
m^u_\cW(\tilde{B}_\cW)\geq m^u_{\cW}(\bigcup_{s\in \bar{J}}\tilde{B}_s)-m^u_{\cW}(V_{n^{-2}}(\partial \cW))\geq m^u_{\cW}(\mathcal{B}_n(\eps))-C(\epsilon,\xi)\cdot n^{-2}\geq 1-2\epsilon^{300},
$$
if $n$ is large enough. 

Let $J_0\subset J$ be such that for $i\in J_0$,
\be\label{eq:J0}
m^u_\cW(\mathcal{R}_n\cap \tilde{B}^c_\cW\cap f^{-\eps n}B_i)\geq \epsilon^{150}\mu(B_i).
\ee
Then 
$$
2\epsilon^{300}>m^u_\cW(\tilde{B}^c_\cW)\geq \sum_{i\in J_0}m^u_\cW(\mathcal{R}_n\cap \tilde{B}^c_\cW\cap f^{-\eps n}B_i)\geq \epsilon^{150}\mu(\bigcup_{i\in J_0}B_i),
$$
and hence 
$$
\mu(\bigcup_{i\in J\setminus J_0}B_i)\geq 1-3\epsilon^{150}.
$$
We will therefore discard the set $J_0$ and, with a slight abuse of notation, still denote $J=J\setminus J_0$.

 By Lemma \ref{lem:expmixB} applied to each $\tilde{B}_s$ and $(1-\bar{\epsilon})B_i$, $i \in J$, we get 
\be\label{mixx'}
\mu(\tilde{B}_\cW\cap f^{-n}((1-\bar{\epsilon})B_i)\in (1-\hat{\epsilon},1+\hat{\epsilon})\mu(\tilde{B}_\cW)\mu((1-\bar{\epsilon})B_i).
\ee
By the right inequality in \eqref{eq:absc1'} and Lemma \ref{lem:setsB}

$$
 \mu(\bigcup_{i\in J}(1-\bar{\epsilon})B_i))\geq (1-\hat{\eps}^2)\mu(\bigcup_{i\in J}B_i)\geq (1-\hat{\eps}^2)(1-3\eps^{150})\geq 1-4\eps^{150}.
$$
Therefore, 
$$
\mu\Big(\tilde{B}_\cW \cap \Big(\bigcup_{i\in J}f^{-n}(1-\bar{\epsilon})B_i)\Big)\Big)
\geq (1-\hat{\eps})(1-4\eps^{150})\mu(\tilde{B}_\cW)\geq (1-5\eps^{150})\mu(\tilde{B}_\cW)
$$
Note that  by Lemma \ref{lem:new1} applied to each $s\in S$, 
$$
\mu(\mathcal{R}_n\cap \tilde{B}_\cW)=\sum_{s\in S}\mu\Big(\bigcup_{j}\tilde{R}_{s,j}(e^{-\eta_2\epsilon n-\eps^2n})\Big)\geq (1-\hat{\eps})\sum_{s\in S}\mu(\tilde{B}_s)=(1-\hat{\eps})\mu(\tilde{B}_\cW).
$$
The two inequalities above give
\be\label{eq:lily'}
\mu\Big(\mathcal{R}_n\cap \tilde{B}_\cW\cap  \Big(\bigcup_{i\in J}f^{-n}(1-\bar{\epsilon})B_i)\Big)\Big)\geq (1-6\eps^{150})\mu(\tilde{B}_\cW).
\ee
Let $J'\subset J$ be such that for every 
 $i\in J'$,
\be\label{eq:la22}
\mu\Big(\mathcal{R}_n\cap \tilde{B}_\cW\cap f^{-n}\Big((1-\bar{\epsilon})B_i\Big)\Big)\geq (1-\epsilon^{100})\mu\Big(\tilde{B}_\cW\cap f^{-n}((1-\bar{\epsilon})B_i\Big)\Big).
\ee
By \eqref{eq:lily'}, the definition of $J'$, and \eqref{mixx'},
$$
(1-6\eps^{150})\mu(\tilde{B}_\cW)\leq \mu\Big(\mathcal{R}_n\cap \tilde{B}_\cW\cap  \Big(\bigcup_{i\in J}f^{-n}((1-\bar{\epsilon})B_i)\Big)\Big)=
$$$$
\sum_{i\in J}\mu\Big(\mathcal{R}_n\cap \tilde{B}_\cW \cap f^{-n}((1-\bar{\epsilon})B_i)\Big)\leq $$$$
\sum_{i\in J'}\mu(\tilde{B}_\cW\cap f^{-n}((1-\bar{\epsilon})B_i))+ (1-\epsilon^{100})\sum_{i\in J\setminus J'}\mu(\tilde{B}_\cW\cap f^{-n}((1-\bar{\epsilon})B_i))\leq
$$$$
(1+\hat{\eps})\mu(\tilde{B}_\cW)\mu(\bigcup_{i\in J'}(1-\bar{\eps})B_i)+ (1+\hat{\eps})(1-\epsilon^{100})\mu(\tilde{B}_\cW)\mu(\bigcup_{i\in J\setminus J'}(1-\bar{\eps})B_i)\leq $$$$
(1+\hat{\eps})^2\mu(\tilde{B}_\cW)-
\epsilon^{100}(1+\hat{\eps})\mu(\tilde{B}_\cW)\mu(\bigcup_{i\in J\setminus J'}(1-\bar{\eps})B_i).
$$

From the above it follows that
$$
\mu(\bigcup_{i\in J\setminus J'}B_i)\leq 2 \mu(\bigcup_{i\in J\setminus J'}(1-\bar{\eps})B_i)\leq \epsilon^{10}.
$$
It remains to show that \eqref{eq:newts} holds for $i\in J'$. Fix $i\in J'$. We will first show the upper bound, which is slightly easier. Notice that since $\cW$ crosses $\tilde{B}_s$ completely, we have by the upper bound in Proposition \ref{prop:3.5} and \eqref{eq:J0},
 $$
m^u_\cW(\mathcal{R}_n\cap f^{-\eps n}B_i)\leq \epsilon^{150}\mu(B_i)+m^u_\cW\left(\mathcal{R}_n\cap \tilde{B}_\cW\cap f^{-\epsilon n}(B_i)\right)=$$$$\epsilon^{150}\mu(B_i)+\sum_{s\in S}m^u_\cW\left(\mathcal{R}_n\cap \tilde{B}_s\cap f^{-\epsilon n}(B_i)\right)\leq$$$$\epsilon^{ 150}\mu(B_i)+ (1+\eps^{ 20})
\mu(B_i)\sum_{s\in S}m^u_\cW(\mathcal{R}_n\cap \tilde{B}_s)\leq (1+\eps^{10})\mu(B_i).
$$
 This finishes the proof of the upper bound.

We now proceed to show the lower bound in \eqref{eq:newts}. 
Let $S'=S'(i)\subset S$ be such that for $s\in S'$, we have 
\be\label{eq:s111}
\mu(\mathcal{R}_n\cap \tilde{B}_s\cap f^{-n}(1-\bar{\epsilon})B_i)\geq (1-\epsilon^{40})\mu(\tilde{B}_s\cap f^{-n}(1-\bar{\epsilon})B_i).
\ee
Then by \eqref{eq:la22} and exponential mixing, 
$$
(1-\epsilon^{100})\mu(\tilde{B}_\cW)\mu(B_i)\leq \sum_{s\in S}\mu\Big(\mathcal{R}_n\cap \tilde{B}_s \cap f^{-n}((1-\bar{\epsilon})B_i)\Big)\leq $$$$(1+2\hat{\epsilon})\sum_{s\in S'}\mu(\tilde{B}_s)\mu(B_i)+(1+2\hat{\eps})(1-\epsilon^{40})\sum_{s\in S\setminus S'}\mu(\tilde{B}_s)\mu(B_i)\leq $$$$(1+2\hat{\eps})\mu(B_i) \mu(\tilde{B}_\cW) -\epsilon^{40} \mu(B_i)\sum_{s\in S\setminus S'}\mu(\tilde{B}_s).
$$
From this it follows that 
\be\label{eq:cda}
\sum_{s\in S'}\mu(\tilde{B}_s)\geq (1-{ 2\epsilon^{60}})\mu(\tilde{B}_\cW).
\ee
 Take $s\in S'$. Then \eqref{eq:s111} holds. Now exponential mixing implies that \eqref{eq:drp3} holds 
 with $\zeta=\bar{\eps}$.
 Since $\cW$ crosses $\tilde{B}_s$ completely it follows from Proposition \ref{prop:3.5} that
\be\label{eq:fineq}
m^u_{\cW}(\mathcal{R}_n\cap \tilde{B}_s\cap f^{-\epsilon n}B_i)\geq (1-2\epsilon^{20})\mu(B_i)m^u_{\cW}(\mathcal{R}_n\cap \tilde{B}_s).
\ee
Summing the above over $s\in S'$, we get that 
\be\label{eq:sums'}
m^u_{\cW}(\mathcal{R}_n\cap \bigcup_{s\in S'}\tilde{B}_s\cap f^{-\epsilon n}B_i)\geq (1-2\epsilon^{20})\mu(B_i)m^u_{\cW}(\mathcal{R}_n\cap \bigcup_{s\in S'}\tilde{B}_s).
\ee

Let $s\in S$. Since $\cW$ crosses $\tilde{B}_s$ completely, we have by Proposition \ref{lem:abscont}
$$
m^u_{\cW}(\mathcal{R}_n\cap \tilde{B}_s)\in (1-\hat{\eps},1+\hat{\eps}) m^u_{\hcW_s}(\mathcal{R}_n\cap \tilde{B}_s)
$$
and 
$$
m^u_{\hcW_s}(\mathcal{R}_n\cap \tilde{B}_s)=m^u_{\hcW_s}(\bigcup_{j}R_{s,j})\geq 1-\hat{\eps}.
$$
From this it follows that for every $s,s'\in S$, 
$$
\frac{m^u_{\cW}(\mathcal{R}_n\cap \tilde{B}_s)}{m^u_{\cW}(\mathcal{R}_n\cap \tilde{B}_{s'})}\in (1-4\hat{\epsilon},1+4\hat{\epsilon}).
$$
Combining this with \eqref{eq:bme} and Lemma \ref{lem:new1} we get that
$$
m^u_{\cW}(\mathcal{R}_n\cap \bigcup_{s\in S'}\tilde{B}_s)\geq (1-2\epsilon^{20})\frac{|S'|}{|S|}m^u_{\cW}(\tilde{B}_\cW\cap \mathcal{R}_n)\geq (1-4\epsilon^{20})m^u_{\cW}(\tilde{B}_\cW)\geq 1-8\epsilon^{20}.
$$
Using \eqref{eq:sums'}, we get $m^u_\cW(\mathcal{R}_n\cap f^{-\eps n}(B_i))\geq (1-\epsilon^{10})\mu(B_i)$. This finishes the proof of the lower bound and hence also the proof of {\bf C}.
The proof of Main Proposition is thus finished. \qed

\section{K-property}\label{sec:K}
In this section we will use Proposition \ref{prop:expeq} to show the $K$-property.  In fact, the full strength of Proposition \ref{prop:expeq} is not needed for the $K$-property, i.e. we don't need a 
quantitative  equidistribution. Since this result
is of independent interest, we will present the proof in the more general case. For this we introduce the notion of equidistributed leaves.
\begin{definition}\label{def:unsm}Let $f\in C^{1+\alpha}(M,\mu)$. We say that
{\em the images of most unstable leaves become equidistributed
under $f$} 
if 
for each ball $\mathbf{B}\subset M$ and $\eps>0$ there is $\xi>0$ such that for each sufficiently
large $n$ there is
a set $K=K_n$ with $\mu(K)>1-\eps$
such that for each $x\in K$ the size of unstable manifold of $x$ is greater than $\xi$ and 
moreover  
\be\label{FNEqui1} 
m^u_{W^u_{x, \xi}} \left( f^{- n}(\mathbf{B})\right)\in (1-\eps,1+\eps) \mu(\mathbf{B}) .
\ee
\end{definition}

We note that this definition is equivalent to a stronger property, namely that \eqref{FNEqui1} 
holds for each unstable box containing $x$ (rather than centered at $x$):

\begin{lemma}
\label{LmFNEqui}
If the images of most unstable leaves become equidistributed
under $f$ then for each ball $\mathbf{B}\subset M$ and $\eps>0$ there is $\xi>0$ such that for each sufficiently
large $n$ there is
a set $\hK=\hK_n$ with $\mu(\hK)>1-\eps$
such that for each $x\in \hK$ for each $y$ such that $x\in W^u_{y,\xi}$ we have
\be\label{FNEqui2} 
m^u_{W^u_{y, \xi}} \left( f^{- n}(\mathbf{B})\right)\in (1-\eps,1+\eps) \mu(\mathbf{B}) .
\ee
\end{lemma}

\begin{proof}
Take $\heps\ll \eps$. Let $n$ be large and let $\tK_n$ be the set of points satisfying
$$ m^u_{W^u_{x, \xi}} \left( f^{- n}(\mathbf{B})\right)\in (1-\eps,1+\eps) \mu(\mathbf{B}) . $$
Since the images of most unstable leaves become equidistributed
under $f$, we have that $K_n\subset \tK_n$ and  so $\mu(\tK_n)>1-\heps$ if $n$ is large enough.
Let 
$$ \hK=\hK_n=\{x: r_u(x)\geq 3\xi\text{ and } m_{W^u_{x,3\xi}}(\tK)>1-\heps^{1/4}\}. $$
By Lemma \ref{LmLocMeas}, $\mu(\hK)\geq 1-4\heps^{1/4}. $
Next, take $x\in \hK$ and $y\in M$ s.t. $x\in W^u_{y,\xi}.$ Since $x\in \hK_n$ we can find
$z\in W^u_{x,3\xi}\cap \tK_n$ s.t. $d(y, z)\leq 3 \heps^{(1/4 D)} \xi.$ 
Then 
$$m^u_{W^u_{x,3\xi}}(W^u_{y,\xi}\cap W^u_{z,\xi})\geq 
(1-4d  \heps^{ (1/4 D)})m^u_{W^u_{x,3\xi}}(W^u_{z,\xi}).$$
Hence $y$ satisfies \eqref{FNEqui2} provided that 
$4d  \heps^{(1/4 D)}+\heps<\eps.$
\end{proof}

\begin{proposition}
\label{PrWuDense}  
Let $f$ be a $C^{1+\alpha}$ diffeomorphism of a compact manifold $M$ preserving a smooth
measure $\mu.$ If  the images of most unstable leaves become equidistributed
under $f$, then $f$ has K property.
\end{proposition}

\begin{proof}  
We will use Lemma \ref{def:K}.
Fix $\cD\in\cB$, $\epsilon>$ and $\tau>\xi>0$. Let $\{\mathbf{B}_i\}_{i\in J}$,  be a family of balls with $r=r(\epsilon,\xi)$ such that $\mu(\cD\triangle \bigcup_{i\in J}\mathbf{B}_i)<\epsilon^2$.  It is enough show that \eqref{eq:spi} holds
for $\mathbf{B}_i$ (and arbitrary small $\epsilon$). 
  Take $n$ sufficiently large and let $\hK=\hK_n$ be from  Lemma \ref{LmFNEqui}.
 By Lemma \ref{lem:satu} and Markov inequality,
$\eps$ almost every atom $A$ of $\DS  \bigvee_{i=N_1}^{N_2}f^i(\cP)$ 
is $\sqrt{\eps}$ saturated and
has the property that
\be \label{KDense} \mu|_A(\hK)\geq 1-\sqrt{\eps}. \ee
We claim that every atom with these two properties satisfies \eqref{eq:spi}. Indeed by Lemma \ref{LmUDeco} we have
$$ \mu_A(f^{- n} B)=\int_{\fT} \mu_{\cW_t}^u(f^{- n} B) d\nu_A(t)\pm \sqrt\eps.
$$
Next, by Definition \ref{def:unsm}, if $\cW_t\cap \hK\neq \emptyset$ then
$m^u_{\cW_t}(f^{- n} B)=\mu(B)\pm \eps$ while by our assumption
\eqref{KDense}, $\nu_A(t: \; \cW_t\cap K=\emptyset)\leq \sqrt \eps.$
Combining the above estimates we obtain \eqref{eq:spi} proving  the proposition.
\end{proof}

As a consequence of the above result and the Main Proposition, we deduce:
\begin{corollary}
\label{prop:K} 
If $f$ has exponential decay of correlation then $f$ is a $K$-automorphism.
\end{corollary}
\begin{proof} It is enough to show that Main Proposition implies that he images of most unstable leaves become equidistributed
under $f$. For given $\epsilon$, and $\mathbf{B}$, let  $\{O_i\}$
be any family of balls containing $\mathbf{B}$.
Let $\DS \breps=\min\left(\epsilon^{100}, \min_{i\in \{1, \dots, m\}} \frac{\mu(O_i)}{100}\right).$ Let  $n(\breps, \xi)$ come from the Main Proposition. Take $n\geq n(\breps,\xi)$ and let $K=K_n$ be as in
the Main Proposition (for the family $\{O_i\}_{i\in J}$). Note that due to our choice of
$\eps,$ $J'(x)=J$ for every $x\in K$ since removing even
one ball from $J$ will decrease the measure by at least $100 \breps.$ In particular, the assertion of Main Proposition 
 applies to the ball $\mathbf{B}$. Note that $\breps n $ is large if  $n$ is large.
\end{proof}


\section{VWB-property}
\label{ScVWB}
\subsection{The main reduction.}
The following lemma is the main step towards the proof of VWB--property.

\begin{proposition}\label{lem:umatch} 
For every $\epsilon$ and $\xi>0$ there exists $n_0=n_0(\epsilon,\xi)$ such that for every $n\geq n_0$, there exists a set $K=K_n\subset M$, $\mu(K)\geq 1-\epsilon$ so that for every $x,x'\in K$,  and every unstable boxes $\cW_1$, $\cW_2$  of size $\xi$  containing $x$ and $x'$ respectively 
there exists an $\epsilon$ measure preserving map $\theta:(\cW_1,m^u_{\cW_1})\to (\cW_2,m^u_{\cW_2})$ such that 
\be\label{eq:ham}
\frac{1}{n}\Card\left(\{i\in \{0,\ldots, n-1\}\;:\; d(f^ix,f^i(\theta x))<\epsilon\}\right)>1-\epsilon.
\ee
\end{proposition}
Before proving the proposition let us
show how it implies our main result.

\begin{proof}[Proof of Theorem \ref{thm:main} .]
 Let $\cP$ be a sufficiently fine  partition with piecewise smooth boundary.
We will verify the conditions of Corollary \ref{cor:VWE}
 with $\tilde{N}:=n_0(\epsilon,\xi)$ where $n_0$ comes from Proposition \ref{lem:umatch}. 
 Since $f$ enjoys the $K$-property by Corollary \ref{prop:K}, 
we just need to verify \eqref{MetClose}. 
Moreover the partition $\cP$ is regular since it has piecewise smooth boundary. 
Fix an arbitrary $\eps>0.$ 
 By Lemma \ref{lem:satu}
we can take a small $\xi$ and choose $N$ so large that for
any $N'\geq N$, $\eps^2$ almost every atom of the partition $\DS \bigvee_{N}^{N'}T^i\cP$
is $(\xi, \eps^2)$ u-saturated. 

Let $A$ be such an atom satisfying additionally 
$\mu_{|A}( K_n)\geq 1-\eps^2$ (where $K_n$ is from Proposition \ref{lem:umatch}). Since $M$ is also $(\xi, \eps^2)$ u-saturated Lemma \ref{LmUDeco} gives
$$ \mu{|A}=\mu_{A,r}+\int_{\fT_A} \mu_{\cW_t} d\nu_A(t), \quad
\mu=\mu_{M,r}+\int_{\fT_M} \mu_{\cW_t} d\nu_M(t). 
$$
By possibly  removing from $\fT_A$ and $\fT_M$ sets of measure
less than $\eps^2$ we may assume that for each $t\in \fT_A\cup \fT_M,$
$\cW_t\cap K_n\neq\emptyset.$ 
By Lemma \ref{LmUniqueLeb} there is a measure preserving map
$$\Theta: \left(\fT_A, \frac{\nu_A}{\nu_A(\fT_A)}\right)
\mapsto \left(\fT_M, \frac{\nu_M}{\nu_M(\fT_M)}\right). $$
Since both $\nu_A(\fT_A)$ and $\nu_M(\fT_M)$ are between $1-\eps$ and $1$, 
$\Theta$ is also an $\eps$ measure preserving map 
$(\fT_A, \nu_A)\mapsto (\fT_M, \nu_M).$ By Proposition \ref{lem:umatch}  for each $t\in \fT_A$ 
there exists a $\eps$ measure preserving map $\theta_{t, S}:\; \cW_t\mapsto \cW_{\Theta(t)}$
satisfying \eqref{eq:ham} with $n=S$ (notice that $S\geq \tilde{N}=n_0$).
Defining 
$ \theta(x)=\theta_{t, S}(x)$ if $x\in \cW_t$ for some $t\in \fT_A$ and defining it in an arbitrary way
if $\DS x\not\in \bigcup_{t\in \fT_A} \cW_t$ gives an $3\eps$ preserving map between 
$(A, \mu|_A)$ and $(M, \mu)$ satisfying \eqref{MetClose}. 
Since $\eps$ is arbitrary, $f$ is Bernoulli.
\end{proof}

\subsection{Plan of the proof of Proposition \ref{lem:umatch}.}

 For $\epsilon$ and $\xi$ let $$n_0(\epsilon,\xi):=\epsilon^{-2}\max(n(\epsilon^2,\xi),n_{\epsilon^2})$$ 
 where $n(\epsilon,\xi)$ and $n_\epsilon$ come from the Main Proposition and Proposition \ref{prop:3.2} respectively. 
Let $n\geq n_0$, then $n':=\epsilon n\geq n(\epsilon^2,\xi)$. 
Let $\{B_i\}_{i\in J}$ be the collection of parallelograms  and $K_n$ be the set from the Main Proposition. For $x,x'\in K_n$ let $J''=J'(x)\cap J'(x')$ where $J'(x), J'(x')\subset J$ are the subcollections given by the 
Main Proposition. Then 
$\DS \mu\left(\bigcup_{i\in J''} B_i\right)\geq 1-2 \epsilon^{10},$ so  
it follows  by \eqref{ExpEqui} that for $\cW\in \{\cW_1,\cW_2\}$
$$
m^u_{\cW}\left(\mathcal{R}_n\setminus \bigcup_{i\in J''} f^{-n'} B_i\right)\leq  3\epsilon^{10}
$$
 From this and \eqref{eq:LLn} it follows that 
 $$
m^u_{\cW}\left(\left[\mathcal{R}_n\cap f^{-n'}(\cL_{n,\tau})\right]\setminus \bigcup_{i\in J''} f^{-n'} B_i\right)\leq  4\epsilon^{10}.
$$

 The proof of Proposition \ref{lem:umatch} consists of three steps. 
On the first step we remove a small proportion of parallelograms from 
$J''$ so that for the remaining parallelograms most of the intersections of $f^{n'} \cW_s \cap B_i$
happen inside the regular set and not too close to the boundary of $B_i.$ 
On the second step we show that for good parallelograms constructed on the first step 
most of the components of $f^{n'} \cW_s\cap B_i$ are Markov, i.e. they are fullly
crossing $B_i.$ On the third step we construct the coupling between the Markov components using fake
center stable holonomy. Steps 1--3 described above are curried out in \S\S \ref{SSPrune}--\ref{SSCouple} 
respectively.

\subsection{Prunning.}
\label{SSPrune} 
Let $J'''\subset J''$ be the set of $i$ such that for $s=1,2$, we have 
\be\label{eq:lclarge}
(1+\epsilon^{4})m^u_{\cW_s}(\mathcal{R}_n\cap f^{-n'}(\cL_{n,\tau})\cap f^{-n'}((1-\tilde{\eps})B_i))\geq m^u_{\cW_s}(\mathcal{R}_n\cap f^{-n'}((1-\tilde{\eps})B_i)).
\ee
\begin{lemma}\label{JJ-JJJ}
$\DS
m^u_{\cW_s}\Big(\bigcup_{i\in J''\setminus J'''}f^{-n'}B_i\Big)\leq \epsilon^2.
$
\end{lemma}
\begin{proof}
By definition  $J''\setminus J'''=\hJ_1\cup \hJ_2$
where $\hJ_s$ is the set of  indices such that \eqref{eq:lclarge} fails for $\cW_s.$
Note that for $i\in \hJ_s$ 
$$ m^u_{\cW_s}(f^{-n'}(\cL_{n,\tau}^c)\cap \cR_n \cap f^{-n'}((1-\teps)B_i))=$$
$$m^u_{\cW_s}(\cR_n \cap f^{-n'}((1-\teps)B_i))- m^u_{\cW_s}(f^{-n'}(\cL_{n,\tau})\cap\cR_n \cap f^{-n'}((1-\teps)B_i))\geq$$
$$ \frac{\eps^4}{2}m^u_{\cW_s}(\cR_n \cap f^{-n'}((1-\teps)B_i))\geq \frac{\eps^4}{3}\cdot \mu(B_i),
$$
the last inequality by Main Proposition and since $\teps=\eps^{4000}$.
Therefore
$$ \eps^{10} \stackrel{\eqref{eq:LLn}}{\geq} m^u_{W^s} (f^{-n'}(\cL_{n, \tau}^c))\geq $$$$
m^u_{\cW_s}(f^{-n'}(\cL_{n,\tau}^c)\cap \cR_n)\geq 
\sum_{i\in \hJ_s} m^u_{\cW_s}(f^{-n'}(\cL_{n,\tau}^c)\cap \cR_n \cap f^{-n'}((1-\teps) B_i))\geq 
\frac{\eps^4}{3} \sum_{i\in \hJ_s} \mu(B_i). 
 $$
Hence 
$\DS \sum_{i\in \hJ_s} \mu(B_i)\leq 3 \eps^6 $. 
Since $J''\setminus J'''=\hJ_1\cup \hJ_2$ it follows that 
$$  \sum_{i\in J''\setminus J'''} \mu(B_i)\leq 6 \eps^6 . $$ 
Now the Main Proposition gives
$$  \sum_{i\in J''\setminus J'''}m^u_{\cW_s} (\cR_n \cap f^{-n'} B_i) \leq 8 \eps^6 . $$ 
Since 
$$ m^u_{\cW_s} \left(\bigcup_{i\in J''\setminus J'''} f^{-n'} B_i \right)\leq
m^u_{\cW_s} \left(\cR_n \bigcap \left[\bigcup_{i\in J''\setminus J'''} f^{-n'} B_i\right] \right)+
m^u_{\cW_s}(\cR_n^c) $$
the result follows.
\end{proof}

\subsection{Abundance of Markov returns.}
  By Lemma \ref{JJ-JJJ}
it is  enough to construct for every $i\in J'''$, an $\epsilon/10$- measure preserving map 
\begin{multline}\label{eq:defth}
\theta_i: \left(\mathcal{R}_n\cap f^{-n'}(\cL_{n,\tau})\cap  f^{-n'}(B_i), m^u_{\cW_1\cap f^{-n'}(B_i)}\right)\to 
\\
\left(\mathcal{R}_n\cap f^{-n'}(\cL_{n,\tau})\cap  f^{-n'}(B_i), m^u_{\cW_2\cap f^{-n'}(B_i)}\right)
\end{multline}
such that \eqref{eq:ham} holds for $x$ and $\theta_ix$. From now, we fix $i\in J'''$ and drop it from the notation. Recall that $\tilde{\epsilon}=\epsilon^{4000}$. Fix $s=1,2$. For any $\bar{z}\in P_\tau\cap f^{n'}\cW_s\cap B$ let $\cW(\bar{z})=W^u_{\bar{z},\tau}\cap B$. Notice that if $\bar{z'}\in \cW(\bar{z})\cap P_\tau \cap f^{n'}(\cW_s)$, then $\cW(\bar{z'})=\cW(\bar{z})$. Moreover, if 
$\bar{z}\in P_\tau\cap f^{n'}\cW_s\cap (1-\tilde{\eps})B$ and 
$d(f^{-n'}\bar{z},\partial \cW_s)\geq c^{-n'}$ (for some small $c>0$), then 
 the size 
 \footnote{ This means that for some point $\tilde{z}\in\cW(\bar{z})$, $\cW(\bar{z})$ is inside an unstable cube of $\tilde{z}$ of side length $(1+\epsilon^{10})\xi_n$ and $\cW(\bar{z})$ contains the unstable cube of $\tilde{z}$ of side length $(1-\epsilon^{10})\xi_n$. } of $\cW(\bar{z})$
 is $\in (1-\epsilon^{10},1+\epsilon^{10})\xi_n$ (recall that $\hcW_i\subset B_i$  is an unstable cube of size $\xi_n$).
 From the above it follows that for some $\{z_{s,\ell}\}_{\ell=1}^{ m_s}\subset \cL_{n,\tau}$,
$$
\{x\in \cW_s\cap f^{-n'}(\cL_{n,\tau})\cap f^{-n'}((1-\tilde{\eps})B)\;:\; d(x,\partial \cW_s)\geq c^{-n'}\}\subset \bigcup_{\ell=1}^{m_s}f^{-n'}(\cW(z_{s,\ell}))
$$
and
\be\label{eq:toadsub}
\{x\in \cW_s: d(x,\partial \cW_s)\geq c^{-n'}\}\cap\bigcup_{\ell=1}^{m_s}f^{-n'}(\cW(z_{s,\ell}))\subset \cW_s\cap f^{-n'}(B),
\ee
where 
\be\label{xij}
\text{the size of }\;\;\; \cW(z_{s,\ell})\;\;\;\text{ is }\geq (1-\epsilon^{10})\xi_n
\ee
Let $W_{s,\ell}:=\cW(z_{s,\ell})$, $\tW_{s, \ell}=W_{s,\ell}\cap\cR_n.$

\begin{lemma}
\label{AlmostFill2}
$$ \frac{m^u_{\cW_s} \left(\mathcal{R}_n\bigcap 
\left(\bigcup_{\ell=1}^{m_s} f^{-n'} \tW_{s,\ell}\right)\right)}{\mu(B)}
\in \left[1-10\eps^2, 1+10 \eps^2\right]. $$
\end{lemma}

\begin{proof}
By \eqref{eq:lclarge} and Main Proposition for
$(1-\tilde{\eps})B$ it follows that
$$
m^u_{\cW_s}\Big(\mathcal{R}_n\cap f^{-n'}(\cL_{n,\tau})\cap f^{-n'}((1-\tilde{\eps})B)\Big)\geq  $$$$(1-2\epsilon^4)m^u_{\cW_s}\Big(\mathcal{R}_n\cap f^{-n'}((1-\tilde{\eps})B)\Big)\geq$$$$
(1- 3\epsilon^4)\mu((1-\tilde{\eps})B)\geq (1-\epsilon^3)\mu(B).
$$
Moreover, $m^u_{\cW_s}(\{x\in \cW_s\;:\; d(x,\partial \cW_s)\leq c^{-n'}\})\leq C(\epsilon)\cdot c^{-n'}\leq \epsilon^{10} \mu(B)$,
for large enough $n$. 
This together with Main Proposition shows that

\begin{equation}
\label{AlmostFill}
\frac{m^u_{\cW_s} \left(\mathcal{R}_n\bigcap 
\left(\bigcup_{\ell=1}^{m_s} f^{-n'} W_{s,\ell}\right)\right)}{\mu(B)}
\in \left[1-\eps^2, 1+\eps^2\right].
\end{equation}

Next let $z', z''\in W_{s,\ell}.$ Then
\be \label{SRBProd}
\frac{J(f^{-n} (z'))}{J(f^{-n} (z''))}=
\prod_{j=0}^{n'-1} \left[ \frac{J(f^{-1} (f^{-j} z'))}{J(f^{-1} (f^{-j} z''))}\right]\in (1-\eps^2, 1+\eps^2) 
\ee
where the last step uses Corollary \ref{CrWuBC} and Lemma \ref{admisibleconv}  (applied to  $L^{(k)}=TW_{s,\ell}\circ f^{k-n'}$).

Let 
$\DS \tB=B\cap \cR_n,$ $\DS \brcW=\hcW\cap \cR_n. $
We  apply Proposition \ref{prop:3.2} for $\hcW$ and $\cW_z=W_{s,\ell}\subset B$, with $z=z_{s,\ell}$. Notice that $z_{s,\ell}\in \cL_{n,\tau}$ and by \eqref{xij}, $\zeta\geq (1-\epsilon^{10})\xi_n$.
For $s=1,2$ let $\pi_{s,\ell}:\brcW\to \tW_{s,\ell}$ be given by Proposition \ref{prop:3.2}. 
By construction $m^u_{\hcW}(\cR_n^c)\leq 100 \eps^{\bb/16}$ and since, by Proposition \ref{prop:3.2},
$\pi_{s, \ell}$ can be considered as an $\eps^2$ measure preserving map $\hcW\to W_{s,\ell}$ we
conclude that $m^u_{W_{s, \ell}} (\cR_n^c)\leq \eps^2.$
Combining this with \eqref{AlmostFill} and \eqref{SRBProd} we obtain the result.
\end{proof}

\subsection{Construction of coupling between unstable leaves.}
\label{SSCouple}
\begin{proof}[Proof of Proposition \ref{lem:umatch}]

By  Lemma \ref{AlmostFill2}
it is  enough to construct an $\epsilon/100$ -measure preserving map

\begin{multline}\label{eq:th'}
\theta':\left(\mathcal{R}_n\bigcap \left(\bigcup_{\ell=1}^{m_1} f^{-n'}\left(\tW_{1,\ell}\right)\right), \;m^u_{\bigcup_{\ell=1}^{m_1} f^{-n'}(\tW_{1,\ell})}\right)\to \\
\left(\mathcal{R}_n\bigcap
\left(\bigcup_{\ell=1}^{m_2} f^{-n'}\left(\tW_{2,\ell}\right)\right), 
\;m^u_{\bigcup_{\ell=1}^{m_2} f^{-n'}(\tW_{2,\ell})}\right)
\end{multline}
so that \eqref{eq:ham} holds for $x$ and $\theta' x$.

Let $\tau_{s,\ell}= \pi_{s,\ell}^{-1}\circ f^{n'}$. 
Let $z', z''$ be two points in $\hcW.$ Then
$$ \frac{J(\tau^{-1}_{s, \ell}(z'))}{J(\tau^{-1}_{s, \ell}(z''))}
=\frac{J(\pi _{s, \ell}(z'))}{J(\pi_{s, \ell}(z''))} \times
\frac{J(f^{-n'} (\pi_{s, \ell}(z')))}{J(f^{-n}(\pi_{s, \ell}(z'')))}
\in \left(1-\eps^{3/2}, 1+\eps^{3/2}\right)$$
where the first factor is estimated by by Proposition \ref{prop:3.2} and the second is estimated
by \eqref{SRBProd}.

Thus for any subset $Q\subset \brcW$ we have
$$ m^u_{\cW_s} \left(\bigcup_\ell  \tau^{-1}_{s, \ell}(Q)\right)=
\sum_\ell \int_Q J(\tau_{s, \ell}^{-1} z) dz $$
$$\in 
(1-\eps^{3/2}, 1+\eps^{3/2}) 
\frac{m^u_{\hcW}(Q)}{m^u_{\hcW}(\brcW \cap f^{n'} \cR_n)}
\sum_{\ell} \int_{\brcW\cap f^{n'}\cR_n} J(\tau_{s, \ell}^{-1} z) dz
 $$
\be\label{PreQ}
=(1-\eps^{3/2}, 1+\eps^{3/2}) 
\frac{m^u_{\hcW}(Q)}{m^u_{\hcW}(\brcW\cap f^{n'} \cR_n)}
m^u_{\cW_s} \left(\cR_n \cap \left[\bigcup_{\ell} f^{-n'} \tW_{s, \ell}\right]\right). 
\ee
Now divide $\hcW$ into cubes $Q_k$ of size $\lambda^{-2n}$ with $\lambda=\|f\|_{C^1}.$
By  Lemma \ref{AlmostFill2} and \eqref{PreQ} (applied to $Q_k\cap f^{n'} \cR_n$),
for each $k$
\be\label{QkRatio}
 \frac{m^u_{\cW_1} \left(\cR_n\cap\left[ \bigcup_\ell  \tau^{-1}_{1, \ell}(Q_k)\right]\right)}
{m^u_{\cW_2} \left(\cR_n\cap \left[ \bigcup_\ell  \tau^{-1}_{2, \ell}(Q_k)\right]\right)}
\in \left(1-10 \eps^{3/2}, 1+10 \eps^{3/2}\right).  \ee

Let $\theta$ be any $10\eps^{3/2}$ measure preserving map 
$$
\theta:\mathcal{R}_n\cap\bigcup_{\ell=1}^{m_1}f^{-n'}(\tW_{1, \ell}) \to
\mathcal{R}_n\cap\bigcup_{\ell=1}^{m_2}f^{-n'} (\tW_{2, \ell} )
$$
which, for each $k$, maps
$$ \cR_n\cap\left[ \bigcup_\ell  \tau^{-1}_{1, \ell}(Q_k)\right]\to 
\cR_n\cap\left[ \bigcup_\ell  \tau^{-1}_{2, \ell}(Q_k)\right] $$
(such a map exists by \eqref{QkRatio}). 
We will show that for $n'\leq t\leq n$
\be\label{eq:clx'}
d(f^t x,f^t(\theta(x)))<\epsilon.
\ee
This will give \eqref{eq:ham} and finish the proof of the proposition.
By construction there is a point $y$ such that $d(f^{n'}x, f^{n'} y)\leq \lambda^{-2n}$
and
$f^{n'}y$ and $f^{n'}(\theta(x))$ are connected by $\pi_{s,\ell}$ and are in $B$. 
Since $\pi_{s,\ell}$ is the holonomy for the $\cF_{i,j}$ foliation, 
 {\bf F4} in Lemma~\ref{lem:ff} gives that for $n'\leq t\leq n.$
 $$ d(f^t y, f^t \theta(x))\leq \frac{\epsilon}{2} .$$
On the other hand 
$$ d(f^t x f^t y)\leq \lambda^{t-n'} d(f^{n'}x, f^{n'} y)\leq \lambda^{-n} .$$
Now \eqref{eq:clx'} follows by the triangle inequality.
\end{proof}

\section{Bernoulli property of skew products.}
\label{ScSkew}
We consider a skew product $F$ acting on $X\times Y$ by the formula
$$ F(x,y)=(fx, \tau(x)y) $$
where $f$ is a diffeo of $X$ preserving a measure $\mu_X,$ 
$Y$ admits an action of a Lie group $G$ preserving a measure $\mu_Y$
and $\tau: X\to G$ is a smooth function. $F$ preserves the measure 
$\mu=\mu_X\times \mu_Y.$

\begin{theorem}
\label{ThEMSkew}
\cite[Theorem 4.1(a) and Remark 4.2]{DDKN1}
If $f$ and $G$ are exponentially mixing and if there are positive constants $C, \eps_1, \eps_2$ such that
$$  \mu_X(x: \|\tau_N(x)\|\leq \eps_1 N)\leq C e^{- \eps_2 N}$$ 
then $F$ is exponentially mixing.
\end{theorem}

Here we consider the case where $f$ is a volume preserving Anosov diffeo,
$G=SL_d(\reals)$ and $Y=G/\Gamma$ where $\Gamma$ is a cocompact lattice. We consider two types 
skewing functions. \medskip

(a) {\em Generalized $T, T^{-1}$ transformations.} Let
$\DS \tau(x)=\text{diag}\left(e^{\alpha_1(x)},  e^{\alpha_2(x)}, \dots ,e^{\alpha_d(x)}  \right)$
where $\DS \sum_{k=1}^d \alpha_k=0.$  In this case the conditions of Theorem \ref{ThEMSkew}
are satisfied provided there is $k$ such that $\mu(\alpha_k(x))\neq 0$ 
(see \cite[\S 8.2]{DDKN1}). Thus if $\tau$ has non-zero drift then $F$ is Bernoulli.
By contrast, in the zero drift case $F$ is {\bf not} Bernoulli by \cite[Example 4.3]{DDKN2}.
\medskip

(b) $\tau$ is close to identity and it is s pinching and
twisting in the sense of \cite{AV07}. In this case the conditions of Theorem \ref{ThEMSkew}
are satisfied due to \cite[Theorem 1.5]{GS19}. In particular, if we assume that $d=2$ then
the analysis of \cite[\S 3.7]{GS19} shows that if the conditions of Theorem \ref{ThEMSkew}
fail than either $\tau$ preserves a Riemannian metric on $Y$ or there is $N$ such that
$F^N$ preserves a line field. In the first case $F$ is not ergodic, while in the second case
(a power of)
$F$ is conjugated to the map
$$ F(x,y)=\left(f(x), \left(\begin{array}{cc} \alpha(x) & \beta(x) \\ 0 & \alpha^{-1}(x) \end{array}\right)
\right). $$
In case $\mu(\ln \alpha)\neq 0$ 
the conditions of Theorem \ref{ThEMSkew} are still satisfied,
so that $F$ is Bernoulli. On the other hand, in case 
$\mu(\ln \alpha)=0$ it seems likely that combining the methods of 
\cite{DDKN2} and \cite{KRHV} one can show that $F$ is {\bf not} Bernoulli.

We see that in both cases (a) and (b) our main theorem gives optimal results.

\appendix
\section{The Pesin Theory}
\label{AppPesin}
\subsection{Estimates of norms.}

We start with discussing L1. and L2. in Lemma \ref{lyapchart}:
\begin{lemma}\label{holLx}
For every $\delta<\delta_0$, $L_{x,\delta}$ in Lemma \ref{lyapchart} can be chosen so that $L_x:\R^D\to T_xM$ are isometries between $|\cdot|'_ {x,\delta}$ and euclidean metric in $\R^D$ and so that for every $\tau>0$,  on $P_\tau$, $x\to L_{x,\delta}$ is $\alpha_2$- H\"older. 
\end{lemma}
\begin{proof} By Lemma
\ref{LmIDHold}, 
 $E^{cs}(x)$ and $E^u(x)$ are $\alpha_1$-H\"older continuous
 on $P_\tau$ (and hence also $\alpha_2$- H\"older continuous since $\alpha_2\leq \alpha_1$).
Taking a H\"older continuous varying basis for $E^{cs}(x)$ and $E^u(x)$ and applying Gram-Schmidt orthogonalization gives the H\"older continuity of $L_x$.
\end{proof}

	\begin{lemma}\label{lem:kx}
		Let $K(x)>0$ be a measurable function. If $\sum_{n\in\Z}K(f^n(x))e^{-\epsilon |n|}<\infty$ $x$ a.e. then there is positive measurable function $K(x; \epsilon)$  such that $K(x)\leq K(x; \epsilon)$ and $e^{-\epsilon}K_{\epsilon}(x)\leq K_{\epsilon}(f(x))\leq e^\epsilon K_{\epsilon}(x)$ for a.e. $x$.
	\end{lemma}
	\begin{proof} Define $\DS K_{\epsilon}(x):=\sum_{n\in\Z}K(f^n(x))e^{-\epsilon |n|}$. \end{proof}

We shall use this lemma to assume that the function $\fR_\delta(x)$ from Lemma \ref{ThLRFull}
can be taken sufficiently large. 
Namely given $\delta>0, M>0$ and recalling that by Theorem \ref{ThLRFull} 
a.e. $x$ is $(\lambda, \delta/2)$-regular, we set
 $\widetilde{\fR}_\delta(x)=\max(M, \fR_{\delta/2}(x))$ and $\fR_\delta'=\widetilde{\fR}_\delta(x; \delta).$
 Then $\fR_\delta'$ satisfies all the properties of Definition \ref{DefLyapReg} and additionally $\fR_\delta'(x)>M$.
Thus we will assume in the arguments below that $\fR$ is sufficiently large (this will entail that the
size $\mathfrak{r}_\delta$ of Pesin charts defined by \eqref{PesRad} is sufficiently small).

In the considerations below we will omit $\delta$ in the notation for $|\cdot|'_{x,\delta}$	.
	\begin{lemma}\label{complyap}
		Let $x$  be a $(\lambda, \delta/4)$-Lyapunov regular point, where
		$\delta$ is sufficiently small. For $w\in T_{x}M$ we have 
		$$\frac{1}{\sqrt 2}|w|_x\leq|w|'_x\leq\frac{\sqrt{2}(\mathfrak{R}_{\delta/4}(x))^2}{\sqrt{1-e^{-\delta}}}|w|_x,$$ 
		\be\label{eq:compla}
		|D_xfu|'_{fx}\geq e^{\lambda-\delta}|u|'_x\;\mbox{for}\; u\in E^{u}(x),\;\;\;\;\;\;|D_xfv|'_{fx}\leq e^{\delta}|v|'_x \;\mbox{for}\; v\in E^{cs}(x)\ee
		and
		 $$
		 \|D_xf\|'_{x\to fx}\leq \sqrt{2} \|Df\|_{C^0},\;\;\;\;\;\;\;\;\|(D_xf)^{-1}\|'_{fx\to x}\leq \sqrt{2} \|Df^{-1}\|_{C^0}.$$
		\end{lemma}
	\begin{proof}
Clearly $|v|_x\leq |v|_x$ for $v\in E^{cs}$ and $v\in E^u$. Now, if $w=w^{cs}+w^u$, then $$|w|_x\leq |w^{cs}|_x+|w^{cs}|_x\leq|w^{cs}|'_x+|w^{cs}|'_x\leq\sqrt{2}|w|'_x.$$ 
Since $x$ is $(\lambda, \delta/4)$- Lyapunov regular,
$\mathfrak{R}_{\delta/4}(f^kx)\leq e^{\delta/4|k|}\mathfrak{R}(x)$ and for $v\in E^u(x)$, $k\leq 0$, $$|D_xf^kv|_{f^kx}\leq \mathfrak{R}_{\delta/4}(f^k(x))e^{k\lambda-k\delta/4}|v|_x.$$ 
Hence $|D_xf^kv|_{f^kx}\leq \mathfrak{R}_{\delta/4}(x)e^{k\lambda-k\delta/2}|v|_x$ and 
$$|v|'^2_x\leq \sum_{m\leq 0}(\mathfrak{R}_{\delta/4}(x))^2e^{2m(\lambda-\delta/2)}e^{-2\lambda m+2\delta m}|v|^2_x=\frac{(\mathfrak{R}_{\delta/4}(x))^2}{1-e^{-\delta}}|v|^2_x.$$ 
A similar computation gives that for $v\in E^{cs}$, $\DS |v|'^2_x\leq \frac{(\mathfrak{R}_{\delta/4}(x))^2}{1-e^{-\delta}}|v|^2_x.$ 

Finally, if we write $w=v^{cs}+v^u$, then 
\be \label{TwoNorms1}
	|w|'^2_x=|v^{cs}|'^2_x+|v^{u}|'^2_x\leq 
	\frac{(\mathfrak{R}_{\delta/4}(x))^2}{1-e^{-\delta}}(|v^{cs}|^2_x+|v^{u}|^2_x).
\ee
 Next denoting by $\theta$ the angle between $v^{cs}$ and $v^u$
and recalling that $\theta\geq \fR_{\delta/4}(x)$ since $x$ is $(\lambda, \delta/4)$--Lyapunov regular, we obtain
$$ |v^{cs}+v^{u}|^2_x=|v^{cs}|_x^2+|v^u|_x^2+2 |v^{cs}|_x |v^u|_x\cos\theta\geq
\left(|v^{cs}|_x^2+|v^u|_x^2 \right)(1-\cos\theta)
+\cos\theta(|v^{cs}|_x+|v^u|_x)^2$$
\be \label{TwoNorms2}
\geq 
\left(|v^{cs}|_x^2+|v^u|_x^2 \right)(1-\cos\theta)
\geq \frac{|v^{cs}|_x^2+|v^u|_x^2}{2 (\fR_{\delta/4}(x))^2}. 
\ee
Combining \eqref{TwoNorms1} and \eqref{TwoNorms2} we get

	$$ |w|'^2_x \leq \frac{(\mathfrak{R}_{\delta/4}(x))^2}{1-e^{-\delta}}2(\mathfrak{R}_{\delta/4}(x))^2(|v^{cs}+v^{u}|^2_x)= \frac{2(\mathfrak{R}_{\delta/4}(x))^4}{1-e^{-\delta}}|v|^2_x). $$

Now let us bound  $\|D_xf\|'_{x\to fx}$ from above, 
all the other inequalities follow essentially the same approach and can be recovered from the computation that follows.

Notice that since $E^{cs}(x)$ and $E^{u}(x)$ are orthogonal w.r.t. the $|\cdot|'_x$ metric and $E^{cs}(fx)$ and $E^{u}(fx)$ are orthogonal w.r.t. the $|\cdot|'_{fx}$ metric and $D_xf$ maps $E^{cs}(x)$ to $E^{cs}(fx)$ and $E^{u}(x)$ to $E^{u}(fx)$, once we control the norms of $D_xf|_{E^{cs}(x)}$ and$D_xf|_{E^u(x)}$, we can easily bound
$D_x  f$. Indeed, if  $\|D_xf|_{E^{cs}(x)}\|'_{x\to fx}\leq A$, \; $\|D_xf|_{E^{u}(x)}\|'_{x\to fx}\leq A$ and $v\in T_xM$, $v=v^{cs}+v^u$, then $$|D_xfv|'^2_{fx}=|D_xfv^{cs}|'^2_{fx}+|D_xfv^u|'^2_{fx}\leq A^2|v^{cs}|'^2_x+A^2|v^{u}|'^2_x=A^2|v|'^2_x.$$

So, take $v\in E^u(x)$. Then 
\begin{eqnarray*}
	|D_xfv|'^2_{fx}&=&\sum_{m\leq 0}|D_{fx}f^m\left(D_xfv\right)|^2_{f^{m+1}x}e^{-2\lambda m+2\delta m}\\
	&=&\sum_{m\leq 0}|D_{x}f^{m+1}v|^2_{f^{m+1}x}e^{-2\lambda (m+1)+2\delta (m+1)}e^{2\lambda-2\delta}\\
	&=&e^{2\lambda-2\delta}\left(|D_xfv|_{fx}^2e^{-2\lambda+2\delta}+\sum_{m\leq 0}|D_{x}f^{m}v|^2_{f^{m}x}e^{-2\lambda m+2\delta m}\right)\\	
	&=&|D_xfv|_{fx}^2+e^{2\lambda-2\delta}|v|'^2_x\leq \|Df\|^2_{C^0}|v|^2_x+e^{2\lambda-2\delta}|v|'^2_x\\
	&\leq&\left(\|Df\|^2_{C^0}+e^{2\lambda-2\delta}\right)|v|'^2_x
	\leq2\|Df\|^2_{C^0}|v|'^2_x
\end{eqnarray*}
where we are using that  $e^{\lambda-\delta}\leq\|Df\|_{C^0}$ 
 since otherwise all Lyapunov exponents
of $f$ would be smaller than $\lambda-\delta.$

Take $v\in E^{cs}(x)$ now and let us bound $|D_xfv|'^2_{fx}$, 
$$	|D_xfv|'^2_{fx}=\sum_{m\geq 0}|D_{fx}f^m\left(D_xfv\right)|^2_{f^{m+1}x}e^{-2\delta m}
	=\sum_{m\geq 0}|D_{x}f^{m+1}v|^2_{f^{m+1}x}e^{-2\delta (m+1)}e^{2\delta}$$
$$	=e^{2\delta}\left(\sum_{m\geq 0}|D_{x}f^{m}v|^2_{f^{m}x}e^{-2\delta m}-|v|_{x}^2\right)
	=e^{2\delta}(|v|'^2_x-|v|^2_x)\leq e^{2\delta}|v|'^2_x\leq  \|Df\|^2_{C^0}|v|'^2_x $$
 where we are using that  $e^\delta\leq\|Df\|_{C^0}$, since we are supposing that $\delta$ is sufficiently 
small, in particular $\delta<\lambda.$
Observe that the bound for vectors on $E^u$ and $E^{cs}$ is not completely symmetric, that is why we wrote both. The bound for $\|(D_xf)^{-1}\|'_{fx\to x}$ is symmetric to $\|D_xf\|'_{x\to fx}$ and hence we omit it. 
\end{proof}

\subsection{Proof of Lemma \ref{cs-foliation}}\label{sec:csfol}
For simplicity we will omit the $\delta$ from the notation for $h_x,\tilde{f}_x$ etc. We start with the following lemma:

 \begin{lemma}\label{basicbound}
For $p\in Lyapreg$  and $\bar z\in\R^D$ let us write $$\left(D_{\bar z}\tilde f_p\right)^{-1}=\left(\begin{matrix}(A_p^{uu})^{-1}+\hat T^{uu}_{\bar z,p}&\hat T^{cu}_{\bar z,p}\\
	\hat T^{uc}_{\bar z,p}&(A_p^{cc})^{-1}+\hat T^{cc}_{\bar z,p}
	\end{matrix}\right)$$ w.r.t. the splitting $\R^u\times\R^{cs}$.  Let 
	\be \label{DefK} \bK:=2\max(\|Df\|_{C^0},\|Df^{-1}\|_{C^0}). \ee 
	We have 
	\begin{enumerate}
	\item $\|(A^{uu}_p)^{-1}\|\leq e^{-\lambda+\delta}$,  $\|D_{\bar y}\tilde f_p\|\leq \bK$ for every $\bar y\in\R^D$;
	\item $\|A^{cc}_p\|\leq e^{\delta}$, $\|(D_{\bar y}\tilde f_p)^{-1}\|\leq \bK$ for every $\bar y\in\R^D$;
	\item $\text{H\"ol}_{\alpha_2}(\hat T^{ab}_{\bar z,p})\leq \delta $ for $a,b=u,c$ in all its combinations, 
	i.e. $$\|\hat T^{ab}_{\bar z,p}-T^{ab}_{\bar z',p}\|\leq  \delta\|\bar{z}-\bar{z}'\|^{\alpha_2},$$ 
	where 
	\be \label{Alpha1} 
	\alpha_2=\min(\alpha/2,\alpha_1); \ee
	\item $\hat T^{ab}_{\bar z,p}=0$ for $\bar z=0$ and for $|\bar z|\geq 1$ for $a,b=u,c$ in all its combinations.
\end{enumerate}

\end{lemma}
\begin{proof}[Proof of Lemma \ref{basicbound}.]
	Let $\hat f_p:T_pM\to T_{fp}M$ be so that $\tilde f_p=L_{fp}^{-1}\circ \hat f_p\circ L_p$. 
	Thus $\hat f_p(z) = \exp_{fp}^{-1} \circ  f\circ \exp_p(z)$ if $\|z\|\leq \mathfrak{r}_{\delta}(p)$  and $\hat f_p(z)=D_p f(z)$ if $\|z\|>2\mathfrak{r}_{\delta}(p)$.
		Since $L_x$ is a linear isometry from $\R^D$ with euclidean metric and $T_xM$ 
		with $|\cdot|'_x$ metric, we can 
	 prove the statements for $\hat f_p$ w.r.t. $|\cdot|'_p$ and $|\cdot|'_{fp}$. Let us write 
	$$\left(D_{\hat z}\hat f_p\right)^{-1}=\left(\begin{matrix}(\hat A_p^{uu})^{-1}+\hat{\hat T}^{uu}_{\hat z,p}&\hat{\hat T}^{cu}_{\hat z,p}\\
		\hat{\hat T}^{uc}_{\hat z,p}&(\hat A_p^{cc})^{-1}+\hat{\hat T}^{cc}_{\hat z,p}
	\end{matrix}\right)$$ w.r.t. the splitting $E^u(fp)\oplus E^{cs}(fp)\to E^u(p)\oplus E^{cs}(p)$.  
In particular $$\hat A_p^{uu}=D_pf|E^u(p)\quad \text{and}\quad \hat A_p^{cc}=D_pf|E^{cs}(p).$$ 
Now \eqref{eq:compla}  gives the bound on $(A^{uu}_p)^{-1}$ and $A^{cc}_p$. 

Now we bound the H\"older norm. Combining \eqref{eq:fC1} with Lemma \ref{complyap} we get  that 
$$
	\|(D_{\hat z_1}\hat f_p)^{-1}-(D_{\hat z_2}\hat f_p)^{-1}\|'_{T_{fp}M\to T_pM}\leq \frac{(\mathfrak{R}_{\delta/4}(p))^2}{\sqrt{1-e^{-\delta}}}\sqrt{2}C_1|\hat z_1-\hat z_2|^{\alpha}_{p} \leq $$
$$ \frac{(\mathfrak{R}_{\delta/4}(p))^2}{\sqrt{1-e^{-\delta}}}(\sqrt{2})^{1+\alpha}C_1|\hat z_1-\hat z_2|'^{\alpha}_{p}	\leq \frac{(\mathfrak{R}_{\delta/4}(p))^2}{\sqrt{1-e^{-\delta}}}(\sqrt{2})^{1+\alpha}C_1(\max_{i=1,2}|\hat z_i|'^{\alpha/2})|\hat z_1-\hat z_2|'^{\alpha/2}_{p}	
$$
	if $|\hat z_1|_p,|\hat z_2|_p\leq \frac{1}{C_1}$.

Finally (see \eqref{PesRad}),  if $|\hat z_1|'_p,|\hat z_2|'_p\leq \mathfrak{r}_{\delta}(p)$ then 
	$$\|(D_{\hat z_1}\hat f_p)^{-1}-(D_{\hat z_2}\hat f_p)^{-1}\|'_{T_{fp}M\to T_pM}\leq\delta |\hat z_1-\hat z_2|'^{\alpha/2}_{p}.$$

Since $E^{cs}(p)$ and $E^u(p)$ are orthogonal w.r.t. $|\cdot|'_p$ this gives the bound for $\hat{\hat T}^{ab}_{\hat z,p}$.

Now we can bound also \begin{eqnarray*}
	\|(D_{\hat z}f_p)^{-1}\|'_{fp\to p}&\leq& \|(D_{0}f_p)^{-1}\|'_{fp\to p}+\|(D_{0}f_p)^{-1}-(D_{\hat z}f_p)^{-1}\|'_{fp\to p}\\&\leq& \sqrt{2}\|Df\|_{C^0}+2^{\alpha/2}\delta \mathfrak{r}_{\delta}^{\alpha/2}(p)\leq 2\|Df\|_{C^0},
\end{eqnarray*}
by taking $\delta$ small enough.
Similarly we can bound H\"older constants for $D_{\bar z}f_p$ and its norm. 
\end{proof}

We shall assume that $\delta>0$ is small enough in terms of   $\bK$ given by \eqref{DefK}
and $\lambda$ such that the following inequalities are satisfied: 
$$\frac{3\delta}{1-e^{-\lambda+\sqrt\delta}}\leq\frac{1}{2\bK^2},\;\;\;\;
(e^{-\lambda+\delta}+\delta) e^{2\delta}\leq e^{-\lambda+\sqrt\delta},$$
and
 $$e^{-\lambda+\sqrt\delta}+2\delta<1,\;\;\;\;\left(1+e^{-\lambda+\sqrt\delta}\right)e^{3\delta}<2,\;\;\;\; \frac{1}{1-2\delta e^{\delta}}\leq e^{3\delta},\;\;\;\; \delta<\frac{1}{100}.$$

We assume also that the number $\alpha_6>0$ is small so that 
$$\bK^{\alpha_6}e^{-\lambda+\sqrt\delta}<1.$$

\begin{lemma}\label{graphtransfcs}
	Let $p\in LyapReg(\delta)$  where $\delta$ is sufficiently small.
	 Given $\bar z\in\R^D$ { and a  map $L:\R^{cs}\to\R^u$ with $\|L\|_{C^0}\leq 1$}, we can define $\Gamma_{cs,p,\bar z}(L):\R^{cs}\to\R^{u}$ so that
	$$\graph\Big(\Gamma_{cs,p,\bar z}(L)\Big)=\left(D_{\bar z}\tilde f_p\right)^{-1}\left(\graph(L)\right).$$ 
	Moreover, for every $\bar z,\bar z_i\in\R^D$ and $\|L\|_{C^0},\|L_i\|_{C^0}\leq 1$, $i=1,2$, then \begin{eqnarray}\label{grder}
		\|\Gamma_{cs,p,\bar z}(L)\|_{C^0}\leq e^{-\lambda+\sqrt\delta}\|L\|_{C^0}+2\delta\min\{1,|\bar z|^{\alpha_6}\}
	\end{eqnarray}
	and 
\begin{eqnarray}\label{grhol}
\|\Gamma_{cs,p,\bar z_1}(L_1)-\Gamma_{cs,p,\bar z_2}(L_2)\|_{C^0}\leq e^{-\lambda+{ \sqrt{\delta}}}\|L_1-L_2\|_{C^0}+6\delta|\bar z_1-\bar z_2|^{\alpha_6}.\end{eqnarray}
	%
\end{lemma}

\begin{proof}  
 Denote $$A^{cc}_{\bar z,p,L}=\left(A^{cc}_p\right)^{-1}+\hat T^{cc}_{\bar z,p}+\hat T^{uc}_{\bar z,p}\circ L,$$
and let 	
$$
\Gamma_{cs,p,\bar{z}}(L):=\left(\left(\left(A^{uu}_p\right)^{-1}+\hat T^{uu}_{\bar z,p}\right)\circ L+\hat T^{cu}_{\bar z,p}\right)\circ (A^{cc}_{\bar z,p,L})^{-1}.
$$
For simplicity of notation we denote $\hat L=\Gamma_{cs,p,\bar{z}}(L)$,
$\|\cdot \|=\|\cdot \|_{C^0}.$
Notice that  $(D_{\bar z}\tilde f_p)^{-1}(\graph(L))=\graph(\hat L)$. 

By Lemma \ref{basicbound}(2)
  \be \label{AccAlmDiag}
   \|(A^{cc}_{\bar z,p,L})^{-1}\|\leq \frac{e^{\delta}}{1-e^{\delta}(\|\hat T^{cc}_{\bar z,p}\|+\|\hat T^{cc}_{\bar z,p}\|\|L\|)}\leq \frac{e^{\delta}}{1-e^{\delta}\delta\min\{1,|\bar z|^{\alpha_2}\}(1+\|L\|)}. \ee
  Since  $\|L\|\leq 1$ and $\delta>0$ is small enough, it follows that
   $\DS \|(A^{cc}_{\bar z,p,L})^{-1}\|\leq e^{4\delta}.$ 
  
  Also since $\|L_i\|\leq 1$,  by Lemma \ref{basicbound}(3) 
  it follows that 
  
   $$ \|A^{cc}_{\bar z_1,p,L_1}-A^{cc}_{\bar z_2,p,L_2}\|\leq \delta|\bar z_1-\bar z_2|^{\alpha_2}(1+\|L_1-L_2\|).$$
   
Hence  using the identity 
$(L')^{-1}-(L'')^{-1}=(L')^{-1} (L''-L')(L'')^{-1}$
valid for arbitrary invertible linear maps $L', L''$ we get
\begin{eqnarray*}
		\|\left(A^{cc}_{\bar z_1,p,L_1}\right)^{-1}-\left(A^{cc}_{\bar z_2,p,L_2}\right)^{-1}\|\leq e^{8\delta}\delta|\bar z_1-\bar z_2|^{\alpha_2}(1+\|L_1-L_2\|).		
	\end{eqnarray*} 	
	Combining this estimate with \eqref{AccAlmDiag} and Lemma \ref{basicbound}(2) we get
 $$\|\hat L\|\leq \left((e^{-\lambda+\delta}+\delta)\|L\|+\delta\min\{1,|\bar z|^{\alpha_2}\}\right)e^{2\delta}\leq e^{-\lambda+\sqrt\delta}\|L\|+2\delta\min\{1,|\bar z|^{\alpha_2}\}$$ 
 which gives \eqref{grder}.
 
 Moreover, 
	\begin{eqnarray*}
		&&\|\Gamma_{cs,p,\bar z_1}(L_1)-\Gamma_{cs,p,\bar z_2}(L_2)\|\\&=&\|\left(\left(\left(A^{uu}_p\right)^{-1}+\hat T^{uu}_{\bar z_1,p}\right)\circ L_1+\hat T^{cu}_{\bar z_1,p}\right)\circ (A^{cc}_{\bar z_1,p,L_1})^{-1} \\
		&-&\left(\left(\left(A^{uu}_p\right)^{-1}+\hat T^{uu}_{\bar z_2,p}\right)\circ L_2+\hat T^{cu}_{\bar z_2,p}\right)\circ (A^{cc}_{\bar z_2,p,L_2})^{-1}\| \\
				&\leq&\|\left(\left(\left(A^{uu}_p\right)^{-1}+\hat T^{uu}_{\bar z_1,p}\right)\circ L_1+\hat T^{cu}_{\bar z_1,p}\right)\circ\left((A^{cc}_{\bar z_1,p,L_1})^{-1}-(A^{cc}_{\bar z_2,p,L_2})^{-1}\right)\| \\
			&+&\|\left(\left(\left(\left(A^{uu}_p\right)^{-1}+\hat T^{uu}_{\bar z_1,p}\right)\circ L_1+\hat T^{cu}_{\bar z_1,p}\right)-\left(\left(\left(A^{uu}_p\right)^{-1}+\hat T^{uu}_{\bar z_2,p}\right)\circ L_2+\hat T^{cu}_{\bar z_2,p}\right)\right)\circ (A^{cc}_{\bar z_2,p,L_2})^{-1}\|
				\end{eqnarray*}
		\begin{eqnarray*}
			&\leq&\|\left(\left(A^{uu}_p\right)^{-1}+\hat T^{uu}_{\bar z_1,p}\right)\circ L_1+\hat T^{cu}_{\bar z_1,p}\|\|(A^{cc}_{\bar z_1,p,L_1})^{-1}-(A^{cc}_{\bar z_2,p,L_2})^{-1}\|
		\\
		&+&\|\left(\left(\left(A^{uu}_p\right)^{-1}+\hat T^{uu}_{\bar z_1,p}\right)\circ L_1+\hat T^{cu}_{\bar z_1,p}\right)-\left(\left(\left(A^{uu}_p\right)^{-1}+\hat T^{uu}_{\bar z_2,p}\right)\circ L_2+\hat T^{cu}_{\bar z_2,p}\right)\|\| (A^{cc}_{\bar z_2,p,L_2})^{-1}\|\\
					&\leq&{e^{8\delta}}\delta|\bar z_1-\bar z_2|^{\alpha_2}(1+\|L_1-L_2\|)
		+\left((e^{-\lambda+\delta}+\delta)\|L_1-L_2\|+2\delta |\bar z_1-\bar z_2|^{\alpha_2}\right)e^{2\delta}
		\end{eqnarray*}
\hskip3mm		$\DS \leq{e^{-\lambda+\sqrt\delta}}\|L_1-L_2\|+6\delta|\bar z_1-\bar z_2|^{\alpha_2}$
 proving \eqref{grhol}.
	\end{proof}

\begin{lemma}\label{distalongcs} 
Let $p\in LyapReg$ and let $\bar z_1,\bar z_2\in\R^{D}$, then $$|\tilde f_p(\bar z_1)-\tilde f_p(\bar z_2)|\leq 2\bK|\bar z^u_1-\bar z^u_2|+e^{2\delta}|\bar z^{cs}_1-\bar z^{cs}_2|,$$ 
 where $\bK=2\max(\|Df\|_{C^0},\|Df^{-1}\|_{C^0})$ (see \eqref{DefK}).

\noindent
Moreover, if $|\bar z^u_1-\bar z^u_2|\leq \frac{1}{2\bK^2}|\bar z^{cs}_1-\bar z^{cs}_2|$ then  
$$|\tilde f_p(\bar z_1)-\tilde f_p(\bar z_2)|\leq {e^{3\delta}}|\bar z_1-\bar z_2|.$$ 
\end{lemma}
\begin{proof} The proof is a simple consequence of the mean value theorem and Lemma \ref{basicbound}.
 \end{proof}

For $n\in \N$ and $k\leq n$ let $L^{cs}_k:\R^{cs}\to\R^{u}$ be such that $$\left(D_{\tilde f^{(n-k)}_{f^kx}(\bar z)}\tilde f^{(k)}_{f^{n-k}x}\right)^{-1}(\R^{cs})=\graph(L_k)$$

\begin{proof}[Proof of Lemma \ref{cs-foliation}]
	We define inductively $\eta_k:\R^{cs}\to\R^u$ by 
	 letting $\eta_0$ to be the constant map
	$\eta_0(\brz)\equiv \left(\tilde f_x^{(n)}(\bar y)\right)^u$ and setting
	$$\graph(\eta_{k})=\left(\tilde f_{f^{n-k}x}\right)^{-1}(\graph(\eta_{k-1}))$$ 
	for $1\leq k\leq n$.  Then we define $\tilde\eta^{cs,n}_{x,\bar y}=\eta_n$. 
	We have that $\graph(L_k)$ is the tangent plane to the $\graph(\eta_k)$ at the point $\brz_k=\tilde f_{f^n x}^{-k} (\brz)$, 
	$$(\bar z^{cs},\eta_k(\bar z^{cs}))=\left(\tilde f^{(n-k)}_{f^kx}\right)(\bar z,\eta_0(\bar z)). 
	$$
	Applying \eqref{grder} inductively to $L_k$ (with $L_0=0$), we get
	 $$\|D\eta_k\|_{C^0}\leq {e^{-\lambda+\sqrt\delta}}\|D\eta_{k-1}\|_{C^0}+3\delta$$ and hence 
	$$\|D\eta_k\|_{C^0}\leq \frac{3\delta}{1-{e^{-\lambda+ \sqrt\delta}}}$$ 
	for all $k\in[0,n]$. 
	
	Recall that $\delta$ is such that $\frac{3\delta}{1-{e^{-\lambda+\sqrt\delta}}}\leq\frac{1}{2{\bK}^2}$. 
	Note that for $\bar z_1,\bar z_2\in \graph(\eta_n)=W^{cs,n}_x(\bar y)$,  we have
	$ |\bar z_1^u-\bar z_2^u|\leq \frac{|\bar z_1^{cs}-\bar z_2^{cs}|}{2\bK^2}$.
	Hence  Lemma \ref{distalongcs} gives that for $k\in [0,n]$, $$|\tilde f_x^{(n-k+1)}(\bar z_1)-\tilde f_x^{(n-k+1)}(\bar z_2)|\leq {e^{3\delta}}|\tilde f_x^{(n-k)}(\bar z_1)-\tilde f_x^{(n-k)}(\bar z_2)|.$$

Let $L_i^{(k)}$ satisfy 	
$$\left(D_{\tilde f^{(n-k)}_{f^kx}(\bar z_i)}\tilde f^{(k)}_{f^{n-k}x}\right)^{-1}(\R^{cs})=graph(L_i^{(k)}).$$
	
 Assuming inductively that 
 \eqref{TSpHolder} 
holds for $k-1$, that is
 $$\|L_1^{(k-1)}{(\bar{z}_1)}-L_1^{(k-1)}{(\bar{z}_2)}\|\leq \frac{12\delta}{1-{e^{-\lambda+\sqrt\delta}}}|\tilde f_x^{(n-k+1)}(\bar z_1)-\tilde f_x^{(n-k+1)}(\bar z_2)|^{\alpha_6},$$
(note that  \eqref{TSpHolder} trivially holds for $k=0$)
 and applying 
 \eqref{grhol} we get 
 \begin{eqnarray*}
		\|L_1^{(k)}{(\bar{z}_1)}-L_1^{(k)}{(\bar{z}_2)}\|
		&\leq&{e^{-\lambda+\sqrt\delta}}\|L_1^{(k-1)}{(\bar{z}_1)}-L_1^{(k-1)}{(\bar{z}_2)}\|+6\delta |\tilde f_x^{(n-k+1)}(\bar z_1)-\tilde f_x^{(n-k+1)}(\bar z_2)|^{\alpha_6}\\
		&\leq&\left(\frac{12\delta{e^{-\lambda+\sqrt\delta}}}{1-{e^{-\lambda+\sqrt\delta}}}+6\delta\right) |\tilde f_x^{(n-k+1)}(\bar z_1)-\tilde f_x^{(n-k+1)}(\bar z_2)|^{\alpha_6}\\
		&\leq&\left(\frac{2 {e^{-\lambda+\sqrt\delta}}}{1-{e^{-\lambda+\sqrt\delta}}}+1\right) 6\delta {e^{3\alpha\delta}}|\tilde f_x^{(n-k)}(\bar z_1)-\tilde f_x^{(n-k)}(\bar z_2)|^{\alpha_6}\\	
		&=&	\left(1+{e^{-\lambda+\sqrt\delta}}\right){e^{3\alpha_6\delta}} \frac{6\delta}{1-{e^{-\lambda+\sqrt\delta}}} |\tilde f_x^{(n-k)}(\bar z_1)-\tilde f_x^{(n-k)}(\bar z_2)|^{\alpha_6}
		\end{eqnarray*}
\hskip40mm		$\DS \leq	 \frac{12\delta}{1-{e^{-\lambda+\sqrt\delta}}} |\tilde f_x^{(n-k)}(\bar z_1)-\tilde f_x^{(n-k)}(\bar z_2)|^{\alpha_6}		$
	\end{proof}

\begin{corollary}\label{expcs}
	Let $x\in LyapReg$. If $\bar z_1,\bar z_2$ belong to the same $\tW_{x}^{cs,n}$ leaf then $$|\tilde f_x^{(k)}(\bar z_1)-\tilde f_x^{(k)}(\bar z_2)|\leq {e^{3k\delta}}|\bar z_1-\bar z_2|$$ for $k\in[0,n]$.
\end{corollary}
\begin{proof}
This follows by applying inductively the second inequality of Lemma \ref{distalongcs}.
\end{proof}

\subsection{H\"older foliations.}\label{sub:hf}

\begin{proof}[Proof of Lemma \ref{lem:compl}]
Recall the coordinates $(a,b)$ 
 introduced in \eqref{BXiR}.
 Let $\tb=b-\eta_z(a).$ 
 In coordinates $(a, \tb)$ 
 $\cW$ is a graph of a function
$ \tb=\teta(a)$ where $ \teta$ is defined for $\|a\|\leq \xi$ 
(since $ \partial{\cW}\bigcap B_i(r)=\emptyset$) and moreover there is $\bra$ with 
$\|\bra\|\leq (1-\heps)\xi$ such that $\|\teta(\bra)\|\leq (1-\heps) r.$
By H\"older 
continuity of $E^u$ on $P_\tau$ 
 (Lemma \ref{LmIDHold})
it follows that on  the ball of radius $\xi$
$$
 \| \teta(a)-\teta(\bra) \|\leq K(\tau) \xi^{1+\alpha_2}, $$
whence
$$ \|\teta(a)\|\leq (1-\heps) r+K(\tau) \xi^{1+\alpha_2} \leq r.
$$
Next, pick $y\in \hcW.$  By Lemma \ref{cs-foliation},
$\cF(y)$ is given by a graph of a function 
$a= \phi(b)$
with $\nabla\phi$ small. 
 Applying the Implicit Function Theorem to the equation
$\DS a=\phi(\tb+\eta_z(a))$  we see that $\cF(y)$ can also be given by the equation
$a=\tphi(\tb)$
with $\nabla\tphi$ small.
 Since $y\in (1-\heps)B_i$ we have
$\DS \|\tphi(0)\|\leq (1-\eps)\xi.$ Therefore on  the ball of radius $r$
$$ \|\tphi(\tb)\|\leq (1-\eps)\xi+C r\leq \xi. $$ 
Now the fact that $\cW$ and $\cF(y)$ intersect follows from transversality.
More precisely, we need to show that the system
$$ \tb=\teta(a), \quad a=\tphi(\tb)$$ has the unique solution. 
The existence of the solution follows from the Fixed Point Theorem applied to
map of $\{\|a\|\leq \xi\}\times \{\|\tb\|\leq r\}$ defined by
$\Phi(a,b)=(\tphi(b), \teta(a)).$
The uniqueness follows since if there were two intersection points, then the line joining them would belong 
to both $\cC^u_{\fb}$ and $\cC^{cs}_{\fb}.$ Alternatively, the H\"older regularity of $E^u$ and $E^{cs}$ 
(see \eqref{TSpHolder} and \ref{fakeu-foliation})
imples that $\Phi$ is a contraction if both $r$ and $\xi$ are small. 
\end{proof}

\subsection{Hadamard Perron, unstable foliation}
In this subsection we state the results analogous  to ones proven in \S \ref{sec:csfol} 
but with the unstable direction instead of the center stable directions. The proofs for the unstable direction are analogous to the proofs for the center stable direction and hence we will omit them.

\begin{lemma}\label{fakeu-foliation}
 There exists $\alpha_4,\tdelta_f>0$ such that the following holds. Fix $\delta\in (0,\tdelta_f)$, 
and $x\in LyapReg$. Then for every $\bar y\in\R^D$, there is $\tilde\eta^{u}_{x,\bar y}:\R^{u}\to\R^{cs}$ such that 
$$\tW^{u}_{x}(\bar y)=graph\left(\tilde\eta^{u}_{x,\bar y}\right),$$ 
$$\|D\tilde\eta^{u}_{x,\bar y}\|_{C^0}\leq \frac{3\delta}{1-{e^{-\lambda+\sqrt\delta}}}$$ 
and 
$$[D\tilde\eta^{u}_{x,\bar y}]_{{C^{\alpha_4}}}\leq  \frac{12\delta}{1-{e^{-\lambda+\sqrt\delta}}}. $$
\end{lemma}

	%



\begin{corollary} \label{CrWuBC}
	If $\bar z_1,\bar z_2$ belong to the same $\tW_{f^kx}^{u}$ leaf then $$|(\tilde f_x^{(k)})^{-1}(\bar z_1)-(\tilde f_x^{(k)})^{-1}(\bar z_2)|\leq e^{k(-\lambda+2\delta)}|\bar z_1-\bar z_2|$$ for $k\in[0,n]$.
\end{corollary}


We finish with the following lemma:
\begin{lemma}\label{admisibleconv} 
There exists $\alpha_7>0$ such that the following holds:
Let $\bar z\in \R^D$ and $E=graph(L)$ where $L:\R^u\to\R^{cs}$ and $\|L\|\leq  \bf q\leq 1$.  Then $D_{\bar z}\tilde f_x^{(k)}(E)=graph (L_k)$ where $L_k:\R^u\to\R^{cs}$ and $\|L_k\|\leq 2{\bf q}$. Moreover, if we call $\bar z_k=\tilde f_{x}^{(k)}(\bar z)$  then $$\|L_k\|\leq {e^{k(-\lambda+\sqrt\delta)}}\|L\|+2\delta\sum_{i=1}^k{e^{(k-i)(-\lambda+\sqrt\delta)}}\min\{1,|\bar z_k|^{\alpha_7}\}.$$

	Finally if $L_i$ and $\bar z_i$ are as above, $\|L_i\|\leq 1$, and we define  $L^{(k)}_i$ as above  we  get that 
	\be\label{eq:adml}
\|L^{(k)}_1-L^{(k)}_2\|\leq {e^{k(-\lambda+\sqrt\delta)}}\|L_1-L_2\|+6\delta\sum_{i=1}^{k}{e^{(k-i)(-\lambda+\sqrt\delta)}}|\tilde f_x^{(i)}(\bar z_1)-\tilde f_x^{(i)}(\bar z_2)|^{\alpha_7}.\ee
\end{lemma}
\begin{proof} The proof is analogous to the proof of Lemma \ref{cs-foliation}.
\end{proof}

\section{Exponential mixing.}
\label{AppEM}

Recall that a $C^{1+\alpha}$ map $f$ on $M$ preserving a measure $\mu$ is exponentially mixing if there are constants
$C>0$,$\br>0$ and $\eta_\br>0$ such that
\be \label{EtaR}
\left|\int_M\phi(x)\psi(f^nx)d\mu-\int_M\phi \,d\mu \int_M\psi \,d\mu \right|
\leq Ce^{-\eta_\br n}\|\phi\|_\br \|\psi\|_\br,
\ee
for all $\phi,\psi\in C^{\br}(M)$.
Examples of exponentially mixing maps
include  volume preserving
Anosov diffeomorphsims \cite{Bow, PP90}, time one maps of contact Anosov flows \cite{L04},
mostly contracting systems \cite{C04, D00} (including the examples constructed in \cite{SW00},
and \cite{D04})
partially hyperbolic translations on
homogeneous spaces \cite{KM96},
and partially hyperbolic automorphisms of nilmanifolds \cite{GS}.
We also note that a product of exponentially mixing diffeomorphisms is exponentially mixing
(see e.g. \cite[Theorem A.4]{DFL}) and, more generally, sufficient conditions for mixing of skew 
products with exponentially mixing base and fibers are discussed in \cite{DDKN1} 
(see Theorem \ref{ThEMSkew} of the present paper).

Here we discuss several properties of exponential mixing maps used in our proof.

\begin{lemma}
\label{LmEMAllCr}
Suppose that $\mu$ is a smooth measure. Then if \eqref{EtaR} holds for some $\br>0$ then it
holds for all $\tilde{\br}>0$ (with possibly a different exponent $\eta_{\tilde{\br}}$).
\end{lemma}

\begin{proof}
Note that if \eqref{EtaR} holds 
for some $\br>0$ then it holds for all larger $\tilde{\br}$ with 
$\eta_{\tilde{\br}}=\eta_\br.$ Therefore it suffices to prove that \eqref{EtaR} holds 
for all sufficiently small $\tilde{\br}.$ So we will assume below that $\tilde{\br}\leq 1.$ 
Note that given a function $\phi$ and $\eps>0$ we can find a function $\phi_\eps$ such that
$\DS \|\phi_\eps\|_{\br}\leq C_1 \eps^{-\br} \|\phi\|_{\tilde{\br}}$ and 
$\DS \|\phi-\phi_\eps\|_{C^0} \leq C_1 \eps^{\tilde{\br}} \|\phi\|_{\tilde{\br}}.$
Indeed, applying a partition of unity we can assume that $\phi$ is supported in a single coordinate chart. Using the coordinates coming from that chart  it suffices to prove the result for functions supported 
on compact domain in $\reals^d$, in which case we can just take
$$\phi_\eps(x)=\frac{1}{\eps^d} \int_{\reals^d} \phi(y) p\left(\frac{y-x}{\eps}\right) dy $$
where $p(\cdot)$ is a probability density supported on a unit ball in $\reals^d.$

Now given arbitrary $\phi, \psi\in C^{\tilde{\br}}$ let $\phi_\eps$ and $\psi_\eps$ denote 
the approximations described above. Then
$$ \int_M\phi(x)\psi(f^nx)d\mu=\int_M \phi_\eps(x) \int_M \psi_\eps(f^n x)
+O\left(\eps^{\tilde{\br}} \|\phi\|_{\tilde{\br}} \|\psi\|_{\tilde{\br}}\right) $$
$$ =\int_M\phi_\eps \,d\mu \int_M\psi_\eps \,d\mu +O\left(\eps^{\tilde{\br}} \|\phi\|_{\tilde{\br}} \|\psi\|_{\tilde{\br}}\right) 
+O\left(\eps^{-2 \br} e^{-\eta_\br n} \|\phi\|_{\tilde{\br}} \|\psi\|_{\tilde{\br}}\right) $$
$$ =\int_M\phi \,d\mu \int_M\psi \,d\mu +O\left(\eps^{\tilde{\br}} \|\phi\|_{\tilde{\br}} \|\psi\|_{\tilde{\br}}\right) 
+O\left(\eps^{-2 \br} e^{-\eta_\br n} \|\phi\|_{\tilde{\br}} \|\psi\|_{\tilde{\br}}\right) $$
Taking $\eps=e^{-\eta_\br n/(2\br+\tilde{\br})}$ we obtain that $f$ is exponentially mixing on $C^{\tilde{\br}}$
with
\be \label{EtaTbr}
\eta_{\tilde{\br}}=\frac{\tilde{\br} \eta_\br}{2\br+\tilde{\br}} \ee
as claimed.
\end{proof}

Recall that parallelograms $B(\xi, r)$ are defined by \eqref{BXiR}.
\begin{lemma}\label{lem:expmixB}
There exists $\eta_2>0$ such that for every $\hat{\epsilon}>0$ there exists $n_{\hat{\epsilon}}\in \N$ such that for every $n\geq n_{\epsilon'}$,  every $\xi,r,\xi',r'\geq e^{-3\eta_2 n}$ we have 
$$
\mu(f^{n}(B(\xi,r))\cap B(\xi',r'))\in (1-\hat{\epsilon},1+\hat{\epsilon})\mu(B(\xi,r)\mu(B(\xi',r')).
$$
\end{lemma}
\begin{proof}
The proof follows a standard argument where we approximate the characteristic function of $B(\xi,r)$ and $B(\xi',r')$ by smooth functions with a controlled error. We then deduce the result from exponential mixing for smooth functions. We provide the details to obtain an explicit expression for $\eta_2.$

Let $\psi_{\rho, \eps}$ denote a function which is 1 on the cube of size $\rho$ centered at the origin,
is $0$ outside of the cube of size $\rho+\eps$ and has Lipshitz norm  of order $O(1/\eps).$
Recalling coordinates $(a, b)$ introduced in \eqref{BXiR} and let $\tb=b-\eta_z(a).$ 
Consider the function $\Psi_\eps(x)=\psi_{\xi, \eps}(a(x))  \psi_{r, \eps}(b(x))$ and let $K(\tau):=\max_{x\in P_\tau}\brho(x)$ (see \eqref{MesMultiDisc}).

Then $$ \|\Psi_\eps\|_{Lip}=O(\eps), \quad
\|\Psi_\eps-1_{B(\xi, r)}\|_{L^2}= O\Big(K(\tau)\sqrt{\eps \Delta}\Big), $$
where $\DS \Delta=\xi^{d^u-1} r^{d_{cs}} +\xi^{d^u} r^{d_{cs}-1}$
and the  second equality relies on \eqref{MesMultiDisc}.
Thus letting $\Psi_{\eps}'$ be the similar approximation for $1_{B(\xi', r')}$ we get
$$ \mu(f^{n}(B(\xi,r))\cap B(\xi',r'))
=\mu\left(\Psi_\eps \left(\Psi_\eps'\circ f^n\right)\right) 
+O\left( K(\tau)\sqrt{\eps \Delta} \right)
$$
$$=\mu(\Psi_\eps)\mu(\Psi'_\eps)
+O\left( K(\tau)\sqrt{\eps \Delta} \right)
+O\left(\eps^{-2} e^{-\eta_1 n}\right) 
=$$$$
\mu(B(\xi,r)\mu(B(\xi',r'))
+O\left(K(\tau) \sqrt{\eps \Delta} \right)
+O\left(\eps^{-2} e^{-\eta_1 n}\right) $$
where $\eta_1$ is the mixing exponent for Lipschitz functions.
Choosing $\eps=e^{-\eta_1n/5}/\Delta^{1/10}$ we get both $O(\cdot)$ terms are of the same order,
whence 
\be \label{ParaMix}
  \mu(f^{n}(B(\xi,r))\cap B(\xi',r'))=\mu(B(\xi,r)\mu(B(\xi',r'))+O\left(K(\tau)\Delta^{2/5} e^{-\eta_1 n/5}\right). 
\ee  
We want that the first term on the RHS to be much larger than the second. Notice that $K(\tau)$ only depends on $\tau$ which depends on the $\epsilon$, so this term will be absorbed simply by taking large enough $n$.
By \eqref{MesMultiDisc} the first term is of order $\xi^{2 d_u} r^{2 d_{cs}}$. 
Now a direct computation using
the bound $\xi,r,\xi',r'\geq e^{-3\eta_2 n}$ shows that 
the first term in \eqref{ParaMix} dominates 
provided that
\be \label{Eta2-Eta}
\eta_2<\eta_1\left(\frac{ D}{5}-\frac{2}{15}\right).
\ee
where  $D=\dim(M)$ and $\eta_1$ is given by \eqref{EtaTbr}, with $\tr=1$, that is,
$\DS \eta_1=\begin{cases} \eta & \text{if } \br<1, \\
 \frac{\eta}{1+2\br} & \text{if } \br\geq 1. \end{cases} $
\end{proof}

We will finish this section with proving Lemma \ref{cor:CE}:

 \begin{proof}[Proof of Lemma \ref{cor:CE}]
Since $\#\supp(\mu)>1$, there exists $\kappa>0$ and $p,p'\in supp(\mu)$ with $d(p,p')>\kappa$. Let $A_1,A_2$ be balls of radius $\kappa/10$ centered at $p,p'$ respectively. By definition $c'=\min(\mu(A_1),\mu(A_2))>0$. Let $B\subset M$ be a set with $\mu(B)\geq 1-c'/2$. 

We have the following claim:\\
{\bf CLAIM:} There exist $\bar{\kappa}>0$ such that for every $r>0$, $i=1,2$ there exists $z^i_r\in A_i\cap B$ with 
$$
\mu(O_{r}(z^i_r))> \bar{\kappa}\cdot vol(O_{r}(z^i_r)).
$$ 
\begin{proof}[Proof of {\bf CLAIM}] Consider the Besicovitch cover $\{O_{r}(z)\}_{z\in A_i\cap B}$ of the set $A_i\cap B$. By the Besicovitch covering theorem, $\DS A_i\cap B\subset \bigcup_{i=1}^{c_{dim M}}\bigcup_{j\leq m_i}O_{r}(z_{ij})$, where for every $i\leq c_{dim M}$ and every $j,j'\leq m_i$, $j\neq j'$, $O_r(z_{ij})\cap O_{r}(z_{ij'})=\emptyset$.
Then 
$$
c'/2\leq \mu (A_i\cap B)\leq \sum_{i\leq c_{dim M}}\sum_{j\leq m_i}\mu(O_{r}(z_{ij}))\leq \bar{\kappa}\sum_{i\leq c_{dim M}}\sum_{j\leq m_i}vol(O_{r}(z_{ij}))\leq \bar{\kappa}c_{dim M}vol(M),
$$
which is a contradiction if $\bar{\kappa}$ is small enough.
\end{proof}
We will use the above claim for $r=e^{-\bar{\eta} n}$, for $\bar{\eta}$ to be determined in what follows. Let $\phi_i\in C^{\infty}(M)$ be a function such that $0\leq \phi_i\leq 1$,  $\phi_i\equiv 1$ on $O_{r}(z^i_r)$ and $\phi_i\equiv 0$ outside $O_{2r}(z^i_r)$. 
 Suppose that $\bar{\eta}\br \leq \eta/4$
so that 
$\|\phi_i\|_\br\leq C'\cdot e^{\eta n/4}$. Using exponential mixing for $\phi_1$ and $\phi_2$, by the above bounds on $C^\br$ norms, we get for any $i,j\in \{1,2\}$,
$$
\Big|\int_{M}\phi_i\circ f^n\cdot \phi_j d\mu -\int_M \phi_id\mu\int_M\phi_jd\mu\Big|\leq 
CC'\cdot e^{-\eta n/2}.
$$
 Suppose now that $\breta d <\eta/5.$
Since $\phi_i$ equals $1$ on $O_{r}(z^i_r)$, by the {\bf CLAIM} we get that:
$$
\int_{M}\phi_i\circ f^n\cdot \phi_j d\mu\geq \bar{\kappa}\cdot vol(O_{r}(z^i_r))\cdot vol(O_{r}(z^i_r))- CC'e^{-\eta n/2}>0, 
$$
Since $\phi_i$ is $0$ outside  $O_{2r}(z^i_r)$, the above for $i=1$ and $j=1,2$ gives that $f^n(O_{2r}(z^1_r))$ intersects both  $O_{2r}(z^1_r)$ and  $O_{2r}(z^2_r)$.  Since $z^i_r\in A_i$ and $d(A_1,A_2)\geq \kappa/2$ it follows that 
$$
\diam( f^n(O_{2r}(z^1_r))\geq \kappa/4.
$$
This finishes the proof by taking 
\be \label{Heta}
\heta=\bar{\eta}/2=\eta \min\left(\frac{1}{10 D}, \frac{1}{8\br}\right).
\ee
and $\DS c= \min\left(\frac{\kappa}{4}, \frac{c'}{2}\right).$ 
\end{proof}

\section{Covering.}
\begin{lemma} 
\label{LmUniCover}
For each $K$ there exists $\eps_0$ such that for $\eps\leq\eps_0$ the following holds.
Let $(\Omega, \nu)$ be a measure space, $B \subset D\subset \Omega$ be sets such that 
$\nu(B)\leq \eps^3 \nu(D).$ Let $ \{Q_x\}_{x\in D}$ 
be a measurable family of sets such that
if $R_y=\{x: y\in Q_x\}$ then there is $v$ such that for all $x, y$
\be\label{Round}
\frac{v}{K}\leq \nu(Q_x)\leq K v, \quad
\nu(R_y)\leq K v. \ee 
Then 
$$ \nu(x\in D: \nu(Q_x\cap B)\geq \eps \nu(Q_x))\leq \eps \nu(D). 
$$
\end{lemma}

\begin{proof}
Let $Y$ be a random point in $\Omega$ obtained as follows. First choose $X$ uniformly from $D$
and then choose $Y$ uniformly in $Q_X.$ Note that $Y$ has bounded density with respect to $\nu,$ 
namely the density is
$$ p(y)=\frac{\nu(R_y)}{\int_D \nu(Q_x) d\nu(x)}\leq \frac{K^2}{\nu(D)}. $$
Thus 
$$ \mathbb{P}(Y\in B)=\int_B p(y) d\nu(y)\leq K^2 \eps^3. $$
On the other hand 
$$ \mathbb{P}(Y\in B)=\frac{1}{\nu(D)} \int_D \frac{\nu(Q_x\cap B)}{\nu(Q_x)} dx. $$
So by Markov inequality the set of $x$ where the integrand is greater than $\eps$ has measure
smaller than $\frac{\eps^2}{K^2} \nu(D). $
\end{proof}


\begin{thebibliography}{99}
\bibitem{ALP} Alves J. F., Luzzatto S., Pinheiro, V.
{\em Markov structures and decay of correlations for non-uniformly expanding dynamical systems,}
Ann. Inst. H. Poincare {\bf 22} (2005) 817--839.

\bibitem{AP10} Alves J. F., Pinheiro, V.
{\em Gibbs-Markov structures and limit laws for partially hyperbolic attractors with mostly expanding central direction,} 
Adv. Math. {\bf 223}  (2010) 1706--1730.

\bibitem{AnSin}  Anosov D. V., Sinai Ya. G.
{\em Certain smooth ergodic systems,} Russ. Math. Surv. {\bf 22} (1967) 103--167.

\bibitem{AV07}  Avila A., Viana M. {\em Simplicity of Lyapunov spectra: a sufficient criterion,} 
Port. Math. {\bf 64} (2007) 311--376.

\bibitem{AVW15} Avila A.,Viana M., Wilkinson A. 
{\em Absolute continuity, Lyapunov exponents and rigidity I: geodesic flows,} JEMS 
{\bf 17} (2015) 1435--1462.

\bibitem{BP} L. Barreira, Y. Pesin  {\it Non Uniform Hyperbolicity. Dynamics of systems with nonzero Lyapunov exponents.} Encyclopedia of Math. and  Appl. {\bf 115} (2007) Cambridge Univ. Press, Cambridge, xiv+513 pp. 

\bibitem{BG} Bj\"orklund M.,  Gorodnik A. {\em Central limit theorems for group actions which are exponentially mixing of all orders,} J. Anal. Math. {\bf 141} (2020) 457--482.

\bibitem{Bow74} Bowen R. {\em Bernoulli equilibrium states for Axiom A diffeomorphisms,}
Math. Systems Theory {\bf 8} (1974/75)  289--294.

\bibitem{Bow} Bowen, R.: {\it Equilibrium states and the ergodic theory of Anosov diffeomorphisms. 2d revised edition.}
Springer Lecture Notes in Math. {\bf 470} (2008), viii+75 pp.


\bibitem{BrPes} 
Brin M. I., Pesin Ya. B.
{\em Partially hyperbolic dynamical systems,} (Russian)
Izv. Akad. Nauk SSSR Ser. Mat. {\bf 38} (1974) 170--212.

	\bibitem{BRH} A. Brown, F. Rodriguez Hertz {\it Measure rigidity for random dynamics on surfaces and related skew products.} arXiv:1506.06826 

	\bibitem{BRHZ} A. Brown, F. Rodriguez Hertz, Z. Wang {\it Smooth ergodic theory of $\Z^d$-actions.} arXiv:1610.09997 


\bibitem{BW10} Burns K., Wilkinson A. 
{\it On the ergodicity of partially hyperbolic systems,} Ann. of Math. {\bf 171} (2010) 451--489. 

\bibitem{C04} Castro A. {\it Fast mixing for attractors with a mostly contracting central direction,} 
Ergodic Th. Dynam. Systems {\bf 24} (2004) 17--44.

\bibitem{CH96}
Chernov N. I., Haskell C. 
{\em Nonuniformly hyperbolic K-systems are Bernoulli,} 
Ergodic Theory Dynam. Systems {\bf 16} (1996) 19--44.

\bibitem{CFS} Cornfeld I. P., Fomin S. V., Sinai Ya. G. 
{\em Ergodic theory,} Grundlehren der Mathematischen Wissenschaften {\bf  245} (1982)
 Springer, New York,  x+486 pp.

\bibitem{D00}  Dolgopyat D. {\em On dynamics of mostly contracting diffeomorphisms,}
  Comm. Math. Phys. {\bf 213} (2000)  181--201.
  
\bibitem{D04}   
Dolgopyat D. {\em On differentiability of SRB states for partially hyperbolic systems,} 
Invent. Math. {\bf 155} (2004) 389--449. 


\bibitem{DDKN1} Dolgopyat D., Dong C.,  Kanigowski A., Nandori P.
{\em Mixing properties of generalized $T, T^{-1}$ transformations,} arXiv:2004.07298,
to appear in Israel J. Math.

\bibitem{DDKN2} Dolgopyat D., Dong C.,  Kanigowski A., Nandori P.
{\em Flexibility of statistical properties for smooth systems satisfying the central limit theorem,}
arXiv:2006.02191.

\bibitem{DFL} Dolgopyat D., Fayad B.,  Liu S.
{\em Multiple Borel Cantelli Lemma in dynamics and MultiLog law for recurrence,}
arXiv:2103.08382.

\bibitem{DK19} Dong C.,  Kanigowski A.
{\em Bernoulli property for certain skew products over hyperbolic systems,}
 arXiv:1912.08132.
 
\bibitem{GS} Gorodnik A., Spatzier R. {\em Exponential mixing of nilmanifold automorphisms,} 
J. Anal. Math. {\bf 123} (2014) 355--396.

\bibitem{GS19} Gouezel S., Stoyanov L.
{\em Quantitative Pesin theory for Anosov diffeomorphisms and flows,} 
Ergodic Th. Dynam. Systems {\bf 39} (2019) 159--200.

\bibitem{GPS} Grayson M., Pugh C., Shub M.
{\em Stably ergodic diffeomorphisms,}
Ann. of Math.  {\bf 140} (1994) 295--329.


\bibitem{HPS} Hirsch M. W., Pugh C. C.,  Shub M. 
{\em Invariant manifolds,} Lecture Notes in Math. {\bf 583} (1977)  Springer,
 Berlin-New York. ii+149 pp.
 
\bibitem{Host91}  Host B.
{\em Mixing of all orders and pairwise independent joinings of systems with singular spectrum,} 
Israel J. Math. {\bf 76} (1991) 289--298.

\bibitem{Ka-Bern} Kanigowski A. {\em Bernoulli property for homogeneous systems,}
arXiv:1812.03209.

\bibitem{KRHV} Kanigowski A., Rodriguez Hertz F., Vinhage K.
{\em On the non-equivalence of the Bernoulli and K properties in dimension four,} 
J. Mod. Dyn. {\bf 13} (2018) 221--250. 

\bibitem{Ka80} Katok A. {\em Smooth non-Bernoulli K-automorphisms,} 
Invent. Math. {\bf 61} (1980) 291--299. 

\bibitem{HK} 
Katok A., Hasselblatt B. {\em Introduction to the modern theory of dynamical systems,} 
With a supplementary chapter by Katok and L. Mendoza. Encyclopedia of Math and Appl. {\bf 54} (1995)
Cambridge Univ. Press, Cambridge,  xviii+802 pp.


\bibitem{Kat71} Katznelson Y. 
{\em Ergodic automorphisms of $\mathbb{T}^n$ are Bernoulli shifts,} 
Israel J. Math. {\bf 10} (1971) 186--195. 

\bibitem{KM96} Kleinbock, D. Y., Margulis, G. A.:
{\em Bounded orbits of nonquasiunipotent flows on homogeneous spaces,} 
{\it AMS Transl. Ser. 2} {\bf 171} (1996), 141--172.


\bibitem{Led81}
Ledrappier F.
{\em Some properties of absolutely continuous invariant measures on an interval,}
Ergodic Theory Dynam. Systems {\bf 1} (1981) 77--93.

\bibitem{L04}  Liverani, C.,{\em On contact Anosov flows},  Ann. of Math. (2), {\bf 159} (2004), no. 3, 1275--1312.

\bibitem{OrnsteinWeiss} 
Ornstein D. S., Weiss B. {\em Geodesic flows are Bernoullian,} 
Israel J. Math. {\bf 14} (1973) 184--198.

\bibitem{OW2} Ornstein D. S., Weiss B. {\em On the Bernoulli nature of systems with some hyperbolic structure}, Ergod. Theory Dyn. Sys. {\bf 18}(2), (1998), 441-456.


\bibitem{PP90}  Parry, W., Pollicott, M. {\em Zeta functions and the periodic orbit structure of hyperbolic dynamics,} 
Ast\'{e}risque {\bf 187-188} (1990) 268 pp.

\bibitem{Pes76} Pesin Ya. B.
{\em Families of invariant manifolds that correspond to nonzero characteristic exponents,}
Math. USSR-Izv. {\bf 40} (1976) 1261--1305.

\bibitem{Pes77} Pesin Ya. B.
{\em Geodesic flows in closed Riemannian manifolds without focal points,} (Russian)
Izv. Akad. Nauk SSSR  {\bf 41} (1977) 1252--1288.

\bibitem{PS97}    Pugh C. C., Shub M. 
{\em Stably ergodic dynamical systems and partial hyperbolicity,} J. Complexity {\bf 13} (1997) 125--179.

\bibitem{PS00}    Pugh C. C. Shub M. 
{\em Stable ergodicity and julienne quasi-conformality,} JEMS {\bf 2} (2000) 1--52.

\bibitem{PSW97}    Pugh C. C., Shub M., Wilkinson A. {\em H\"older foliations,} Duke Math. J. 
{\bf 86} (1997) 517--546. 

\bibitem{Rat74} Ratner M.
{\em Anosov flows with Gibbs measures are also Bernoullian,}
Israel J. Math. {\bf 17} (1974) 380--391.

\bibitem{Roh49} Rohlin, V. A.
{\em On endomorphisms of compact commutative groups,}
Izvestiya Akad. Nauk SSSR.  {\bf 13} (1949) 329--340.

\bibitem{Rokhlin} Rohlin V. A. {\em On the fundamental ideas of measure theory,} 
AMS Translation {\bf 71} (1952) 55 pp.



\bibitem{RS61} Rohlin V. A., Sinai Ya. G.
{\em The structure and properties of invariant measurable partitions,} (Russian) 
Dokl. Akad. Nauk SSSR {\bf 141} (1961) 1038--1041. 


\bibitem{Rud} Rudolph D. {\em Asymptotically Brownian skew products give non-loosely Bernoulli K-automorphisms.} Invent. Math., {\bf91} (1988) 105--128.

\bibitem{SW00} Shub M., Wilkinson A.
{\em Pathological foliations and removable zero exponents,} Invent. Math. {\bf 139} (2000) 495--508. 

\bibitem{Sin66} Sinai Ya. G.
{\em Classical dynamic systems with countably-multiple Lebesgue spectrum-II,} (Russian)
Izv. Akad. Nauk SSSR {\bf 30} (1966) 15--68.

\bibitem{Sin} Sinai\ Ya.\ G.\,{\em The hierarchy of stochastic properties of deterministic systems}, Encyclopaedia Math.
Sciences {\bf 100} (2000) 106--108.

		

	
	

		

\end{thebibliography}
\end{document}